\documentclass[11pt]{article}
\usepackage[a4paper, margin=1in]{geometry}

\usepackage[utf8]{inputenc}
\usepackage[T1]{fontenc}
\usepackage{lmodern}
\usepackage{microtype}

\usepackage{amsmath, amssymb, amsthm, mathtools}
\usepackage{cancel}
\usepackage{bm} % Bold math symbols
\usepackage{dsfont} % Double stroke for indicator functions
\numberwithin{equation}{section}
\usepackage[normalem]{ulem}

\allowdisplaybreaks[4]

\newtheoremstyle{mystyle}%   % Style name
    {}%                      % Space above
    {}%                      % Space below
    {\itshape}%              % Body font
    {}%                      % Indent amount for the first line
    {\bfseries}%             % Head font
    {.}%                     % Punctuation after head
    {5pt}%                   % Space after head
    {\thmname{#1}\thmnumber{ #2}\thmnote{ #3}}

\theoremstyle{mystyle}
\newtheorem{theorem}{Theorem}[section]
\newtheorem{lemma}[theorem]{Lemma}
\newtheorem{proposition}[theorem]{Proposition}
\newtheorem{corollary}[theorem]{Corollary}

\theoremstyle{definition}
\newtheorem{definition}[theorem]{Definition}
\newtheorem{remark}[theorem]{Remark}
\newtheorem{assump}[theorem]{Assumption}
\newtheorem{example}[theorem]{Example}

\newtheorem*{theorem*}{Theorem}

\newtheorem{thmintro}{Theorem}

\usepackage{enumitem}

\usepackage{graphicx}
\usepackage{caption}
\usepackage{subcaption}
\usepackage{booktabs} % Professional tables
\usepackage{float} % Improved floating options

\usepackage{xurl}
\usepackage{hyperref}
\hypersetup{
    colorlinks=true,
    linkcolor=black,
    citecolor=black,
    urlcolor=black
}

\usepackage{cleveref}
\crefname{enumi}{Assumption}{Assumptions}

\usepackage[
  style=numeric,
  citestyle=numeric-comp,
  giveninits=true,
  maxcitenames=2,      % Maximum number of author names to display in citations
  mincitenames=1,      % Minimum number of names to keep when there are too many authors
  maxbibnames=3,        % Use et al. if there are 3 or more authors
  isbn=false,
]{biblatex}
\AtEveryBibitem{%
  \ifentrytype{book}
    {\clearfield{pages}\clearfield{pagetotal}}
    {}%
}
\AtEveryBibitem{\iffieldundef{doi}{}{\clearfield{url}\clearfield{urlyear}\clearfield{urlmonth}\clearfield{urlday}}}
\usepackage[dvipsnames]{xcolor}
\usepackage[most]{tcolorbox}
\newcommand{\N}{\mathbb{N}}
\newcommand{\R}{\mathbb{R}}
\newcommand{\E}{\mathbb{E}}

\newcommand{\1}{\mathds{1}}

\newcommand{\bs}[1]{\boldsymbol{#1}}
\newcommand{\mcal}[1]{\mathcal{#1}}
\newcommand{\mbb}[1]{\mathbb{#1}}
\newcommand{\js}{\operatorname{JS}}
\newcommand{\leb}{\operatorname{Leb}}

\newcommand{\D}{\mathbb{D}}

\usepackage{changepage}

\title{Kinetic Mean-Field Planning via Schrödinger Bridges: Existence, Penalization, and Consistency Bounds}

\author{Erhan Bayraktar\thanks{This author is partially supported by the National Science Foundation under grant DMS-2507940 and by the Susan M. Smith chair.}
  \quad Zhiyan Ding \quad Hiroaki Horikawa\\[6pt]
  \normalsize Department of Mathematics, University of Michigan, United States\\
  \normalsize \texttt{erhan@umich.edu}, \texttt{zyding@umich.edu}, \texttt{hhiroaki@umich.edu}
}

\date{}

\begin{document}
\maketitle

\begin{abstract}
We study a mean-field planning problem for a swarm of interacting, non-cooperative agents whose
dynamics are stochastic and kinetic. The noise acts on the velocity only, and the interaction (of Cucker--Smale type for instance) enters through the drift.
For a fixed flow, the problem reduces to a Schr\"odinger bridge problem, and we first prove the existence of an equilibrium by a Schauder fixed-point argument.
Then we replace the hard terminal constraint by a penalty function and show that equilibria of the penalized problems converge, along subsequences, to an equilibrium of the original problem.
Since equilibria are not unique in general, we instead quantify the consistency error, which measures how far an equilibrium of the penalized problem is from being an equilibrium of the original one. The rate is of order $n^{-1/2}$ for the relative-entropy penalty.
The proof relies on the quantitative stability of entropic optimal transport and on a pointwise linear growth bound on the score of the degenerate reference dynamics, which is of independent interest.
\end{abstract}

\tableofcontents

\section{Introduction} \label{sec:intro}

In this paper, we study a mean-field planning problem (MFP) for the stochastic kinetic dynamics:
\begin{equation}
\label{eq:intro_controlled_SDE}
  \left\{
  \begin{aligned}
    & dX_t = V_t \, dt, \\
    & dV_t = \bigl[ f(t, X_t, V_t, \mu_t) + \alpha_t \bigr] \, dt + dW_t, \\
    & (X_0, V_0) \sim \mu_{\mathrm{init}},
  \end{aligned}
  \right.
\end{equation}
where $\bs{\mu} : t \mapsto \mu_t$ is a flow of probability measures on $\R^{2d}$, $\alpha$ is the
control process, and $W$ is a $d$-dimensional Brownian motion. Given the initial and terminal
configurations $\mu_{\mathrm{init}}$ and $\mu_{\mathrm{fin}}$, we seek a triplet
$(\hat{\alpha}, \hat{g}, \hat{\bs{\mu}})$ of a control, a terminal cost, and a flow of probability
measures satisfying:
\begin{flalign}
  &\text{(Optimality)}
  && \hat{\alpha} \in \underset{\alpha \in \mcal{D}}{\arg\min}\,
     \E \left[\frac{1}{2} \int^T_0 |\alpha_t|^2 \, dt + \hat{g}(X_T, V_T) \right],
  && \hspace{4em} \label{eqn:optimality} \\
  &\text{(Consistency)}
  && \hat{\mu}_t = \mcal{L}(\hat{X}_t, \hat{V}_t), \quad t \in [0,T],
  && \label{eqn:consistency} \\[6pt]
  &\text{(Terminal condition)}
  && \hat{\mu}_T = \mu_{\mathrm{fin}},
  && \label{eqn:terminal}
\end{flalign}
where, in \eqref{eqn:optimality}, $\mcal{D}$ is the set of controls for which \eqref{eq:intro_controlled_SDE} admits a weak solution
and $(X, V)$ solves \eqref{eq:intro_controlled_SDE} with $\bs{\mu} = \hat{\bs{\mu}}$, and, in \eqref{eqn:consistency}, $(\hat{X}, \hat{V})$ solves \eqref{eq:intro_controlled_SDE} with $\bs{\mu} = \hat{\bs{\mu}}$ and $\alpha = \hat{\alpha}$.  In words, under the incentive $\hat{g}$, the swarm reaches $\mu_{\mathrm{fin}}$ as an equilibrium of a (non-cooperative) mean-field game (MFG) in which each agent solves the incentivized minimal-action problem \eqref{eqn:optimality}.

Our choice of model, namely kinetic dynamics with an interacting drift, is motivated by the Cucker--Smale flocking model\footnote{For a technical reason, our result covers a modified version of this model. See Example~\ref{example:Cucker--Smale}.}~\cite{cucker2007}, which describes how individual birds or fish adjust their velocities to those of their neighbors and eventually form a coordinated flock. Game-theoretic formulations of this model were proposed in \cite{nourian2010synthesis,nourian2011mean}, where alignment enters the running cost of each agent as a penalty on the deviation of its velocity from those of its neighbors. In the present paper, alignment instead enters the dynamics through the drift $f$: the uncontrolled dynamics ($\alpha \equiv 0$) are themselves of Cucker--Smale type, and the control $\alpha$ describes the deviation of each agent from the flocking behavior.
Following the flock is thus free, whereas any deviation from it costs energy. Hence, if the observed terminal configuration $\mu_{\mathrm{fin}}$ cannot be reached by flocking alone, the agents must have deviated, and the MFP asks what incentive makes such deviations worthwhile and along which paths the swarm is transported between the two configurations.

To solve the MFP, we consider the following constrained minimization problem:
\begin{equation}
\tag*{$\mathrm{SB}(\bs{\mu})$}
  \inf \Bigl\{ H\bigl(Q \,\big|\, R^{\bs{\mu}}\bigr)
  \;:\; Q \in \mcal{P}(\Omega), \ Q_0 = \mu_{\mathrm{init}}, \ Q_T = \mu_{\mathrm{fin}} \Bigr\},
\end{equation}
where $\Omega = C([0,T]; \R^{2d})$, $Q_t$ denotes the time-$t$ marginal of $Q$,
$H(\cdot \,|\, \cdot)$ is the relative entropy, and $R^{\bs{\mu}}$ is the law of the uncontrolled dynamics \eqref{eq:intro_controlled_SDE}, i.e.\ with $\alpha \equiv 0$, which we call the reference process.
It is a variant of the Schr\"odinger bridge (SB) problem, first introduced by Schr\"odinger \cite{schrodinger1931umkehrung, schrodinger1932theorie} to determine the most likely paths of a cloud of particles, relative to a reference measure, given its initial and final configurations.
As F\"ollmer~\cite{Follmer_randomfields} observed, the energy $\E[\frac{1}{2}\int_0^T |\alpha_t|^2 \, dt]$ of an admissible control equals the relative entropy $H(Q \,|\, R^{\bs{\mu}})$ of the law $Q$ of the controlled process by Girsanov's theorem.
Thus, SB($\bs{\mu}$) is equivalent to minimizing the energy under the boundary constraints.

We now define an equilibrium: denoting by $P^{\bs{\mu}}$ the unique minimizer of \ref{eq:SB_mu}, we call a continuous flow $\hat{\bs{\mu}}$ a mean-field equilibrium (MFE) if $\hat{\mu}_t = P^{\hat{\bs{\mu}}}_t$ for all $t \in [0,T]$.
An MFE yields a solution of the MFP as follows.
The minimizer $P^{\hat{\bs{\mu}}}$ admits the classical decomposition
$dP^{\hat{\bs{\mu}}}/dR^{\hat{\bs{\mu}}} = \varphi(Z_0)\,\psi(Z_T)$, where $\varphi, \psi$ are the Schr\"odinger potentials, and it is then also the unique minimizer of the unconstrained problem (see Remark~\ref{remark:MFP_SB})
\begin{equation}
\label{eq:SB_unconstrained}
  \inf \Bigl\{ H\bigl(Q \,\big|\, R^{\hat{\bs{\mu}}}\bigr) - \E^Q \bigl[ \log \psi(Z_T) \bigr]
  \;:\; Q \in \mcal{P}(\Omega), \ Q_0 = \mu_{\mathrm{init}} \Bigr\}.
\end{equation}
Hence, given an MFE $\hat{\bs{\mu}}$, the triplet
$(\hat{\alpha}, \hat{g}, \hat{\bs{\mu}})$ solves the MFP, where $\hat{g} := -\log \psi$ and
$\hat{\alpha}$ is the control associated with $P^{\hat{\bs{\mu}}}$ by
$dP^{\hat{\bs{\mu}}}/dR^{\hat{\bs{\mu}}} = \mcal{E}\bigl(\int_0^{T} \hat{\alpha}_t \cdot dW_t\bigr)$.
Conversely, from every solution of the MFP, we can construct the MFE. For a more precise argument see Remark~\ref{remark:MFP_SB}.

Unlike most of the existing literature on MFP~\cite{poretta2014planning, Achdou2012planning, orrieri2019variational_planning, graber2019planning_masstransport, gomes2018firstorder_planning, francesco2022approx_planning, benamou2019, ren2023pathdependent_mfp, chen2018steering, mimikos2024regularity_mfp, cardaliaguet2024freeboundary}, we consider kinetic dynamics with degenerate noise in which the interaction among the agents enters
through the drift of the velocity equation. Consequently the Hamiltonian of the Hamilton--Jacobi--Bellman (HJB) equation~\eqref{eq:pde} is non-separable; see Section~\ref{sec:lit_review} for the related literature.

The first of our four main results is the existence of an MFE; informally,
it reads as follows.
\begin{thmintro}[Theorem~\ref{thm:existence_mfe}]
\label{thm:informal_existence}
Under suitable conditions on the drift $f$ and on the marginals
$\mu_{\mathrm{init}}, \mu_{\mathrm{fin}}$, there exists an MFE $\hat{\bs{\mu}}$.
\end{thmintro}

Although we can show the existence of the equilibrium, we cannot expect uniqueness under our setting because the monotonicity conditions used in most of the MFP literature do not hold, see a detailed discussion in Remark~\ref{remark:non-uniqueness}. Furthermore, in~Appendix~\ref{appendix:non-uniqueness}, we provide an explicit example of non-uniqueness of the MFE under the assumptions of this work. Existence of an equilibrium can be proven by applying Schauder's fixed-point theorem to the
best-response map $\bs{\mu} \mapsto (P^{\bs{\mu}}_t)_{t \in [0,T]}$, which is well-defined thanks to the well-posedness of the Schr\"odinger bridge problem~\ref{eq:SB_mu}. The key ingredient is a bound on the relative entropy $H(P^{\bs{\mu}} \,|\, R^{\bs{\mu}})$ which is uniform in $\bs{\mu}$.
This follows from a Gaussian lower bound on the transition density of $R^{\bs{\mu}}$, uniform in $\bs{\mu}$, due to \cite[Theorem~1.1]{delarue_density_estimates}; see Section~\ref{sec:sketch_existence} for details.

After the existence result, we also consider a soft-constraint problem, i.e.\ an MFG, to approximately find the MFE and study its convergence rates. What distinguishes the planning problem from a standard MFG is that the terminal cost of the
agent is often not prescribed a priori. A natural way to bring the problem back to the standard setting is to
replace the terminal constraint by a penalty: for $\lambda > 0$ and a fixed flow $\bs{\mu}$,
consider
\begin{equation}
\tag*{$\lambda$-$\mathrm{SB}(\bs{\mu})$}
    \inf_{Q \in \mcal{P}(\Omega)} \Bigl\{
    H(Q \,|\, R^{\bs{\mu}}) + \lambda\, G(Q_T) \,:\, Q_0 = \mu_{\mathrm{init}}
    \Bigr\},
\end{equation}
where $G \ge 0$ is a convex function on $\mcal{P}(\R^{2d})$ vanishing only at
$\mu_{\mathrm{fin}}$. Intuitively, this measures the deviation of the terminal law from $\mu_{\mathrm{fin}}$;
the leading example is $G(\nu) = H(\nu \,|\, \mu_{\mathrm{fin}})$. Denoting by
$P^{\lambda, \bs{\mu}}$ the unique minimizer of \ref{eq:soft_SB_mu}, we call a flow
$\bs{\mu}^\lambda$ a $\lambda$-MFE if $\mu^\lambda_t = P^{\lambda, \bs{\mu}^\lambda}_t$ for all
$t \in [0,T]$. Although the penalty depends on the law $Q_T$ rather than on the terminal state, a
$\lambda$-MFE is precisely an MFE of a standard MFG in which the agent faces the terminal cost
$\lambda \frac{\delta G}{\delta \nu}(\mu^\lambda_T, \cdot)$; see
Remark~\ref{remark:potential_argument}. For a fixed $\bs{\mu}$, \ref{eq:soft_SB_mu} is the
soft-constrained Schr\"odinger bridge problem studied recently in the context of generative
modeling \cite{garg2024soft, ma2025soft-schr} and of the mean-field Schr\"odinger bridge (MFSB)
\cite{hernandez_tangpi_2025}. Penalized formulations of MFP go back to
\cite{Achdou2012planning, poretta2014planning}.

Writing $n$-MFE for a $\lambda$-MFE with $\lambda = n \in \N$, we have, as expected, the following convergence.
\begin{thmintro}[Theorem~\ref{thm:qualitative}]
\label{thm:informal_qualitative}
Under the assumptions of Theorem~\ref{thm:informal_existence} and suitable conditions on $G$, a
$\lambda$-MFE exists for every $\lambda > 0$, and every sequence of $n$-MFEs
$\{\bs{\mu}^n\}_{n \in \N}$ has a subsequence converging to an
MFE $\bs{\mu}^\infty$ in the space of continuous flows.
\end{thmintro}
Thus equilibria of the hard-constrained problem are obtained as limits of equilibria of standard
MFGs. The existence part follows the fixed-point argument of Theorem~\ref{thm:informal_existence}, and the convergence part is proven using the compactness of $\left\{ P^{n, \bs{\mu}^n} \right\}_{n \in \N}$, which follows from the uniform bound on the second moments of the path measures; see Section~\ref{sec:sketch_qualitative}.

Since we do not know about uniqueness of equilibria in general, the convergence in
Theorem~\ref{thm:informal_qualitative} holds only along subsequences, and it is not clear whether $\bs{\mu}^n$ converges to a unique limit when $n\rightarrow\infty$.

However, we can still quantify how far an $n$-MFE is from
being an MFE. Writing $P^{\bs{\mu}^n}$ for the solution of the hard-constrained problem
$\mathrm{SB}(\bs{\mu}^n)$, i.e.\ for the best response to the flow $\bs{\mu}^n$,
the quantity $\sup_{t \in [0,T]} \mcal{W}_1(\mu^n_t, P^{\bs{\mu}^n}_t)$ vanishes exactly when
$\bs{\mu}^n$ is an MFE and measures the violation of the equilibrium condition otherwise; we call
it the consistency error of $\bs{\mu}^n$. Quantities of this type are used in computational methods for MFGs \cite{guo2024mfomo, yu2024conv_mfp}. Our third result provides a rate for this error.
\begin{thmintro}[Theorem~\ref{thm:consistency_error} and Corollary~\ref{cor:consistency}]
\label{thm:informal_consistency}
In addition to the assumptions of Theorem~\ref{thm:informal_existence}, let
$G(\nu) = H(\nu \,|\, \mu_{\mathrm{fin}})$, let $\mu_{\mathrm{fin}}$ satisfy suitable regularity conditions, and let $f$ be smooth and bounded together with all its space derivatives with the bounds uniform in $\bs{\mu}$. Then there exists $C$, independent of
$n$, such that for every $n$-MFE $\bs{\mu}^n$,
\[
  \sup_{t \in [0,T]} \mcal{W}_1(\mu^n_t, P^{\bs{\mu}^n}_t)
  \le \mcal{W}_1\bigl(P^{n, \bs{\mu}^n}, P^{\bs{\mu}^n}\bigr) \le C n^{-1/2}.
\]
\end{thmintro}
Whereas \cite{Achdou2012planning,poretta2014planning,mimikos2024regularity_mfp} estimate the terminal error $\|\mu^n_T - \mu_{\mathrm{fin}}\|$, our result extends the estimate to the whole flow of marginals on $[0,T]$.
Whether the rate $n^{-1/2}$ is sharp is left open; an explicit example in
Remark~\ref{remark:lower_bound} shows that no rate faster than $n^{-1}$ can hold in general.
The proof is based on the classical equivalence between \ref{eq:SB_mu} and
an entropic optimal transport (EOT) problem with cost $-\log p^{\bs{\mu}}_T$, where
$p^{\bs{\mu}}_T$ is the transition density of the reference dynamics. Under this equivalence,
the joint laws $P^{n, \bs{\mu}^n}_{0,T}$ and $P^{\bs{\mu}^n}_{0,T}$ at times $0$ and $T$ solve the same EOT problem with different terminal
marginals, $\mu^n_T$ and $\mu_{\mathrm{fin}}$, so that the problem reduces to the stability of
EOT with respect to its terminal marginal. The quantitative stability of the EOT value due to Eckstein and Nutz
\cite{eckstein2022quant_stability_EOT} controls the Jensen--Shannon divergence of the two optimal
plans, which we then lift to the path measures; see
Section~\ref{sec:sketch_consistency_error}.

Applying the stability result requires a bound on the gradient of the EOT cost
$c = -\log p^{\bs{\mu}}_T$, that is, on the score of the reference dynamics, uniformly in $\bs{\mu}$.
This is our last main result.
\begin{thmintro}[Theorem~\ref{thm:linear_growth}]
\label{thm:informal_linear_growth}
Let $f$ satisfy the assumptions in Theorem~\ref{thm:informal_consistency}. Then for every flow $\bs{\mu}$, the transition density $(z,w) \mapsto p^{\bs{\mu}}_T(z,w)$
of the reference dynamics $R^{\bs{\mu}}$ is infinitely differentiable, and there exists $C$,
independent of $\bs{\mu}$, such that
\begin{equation*}
  \bigl| \nabla_{z,w} \log p^{\bs{\mu}}_T(z,w) \bigr| \le C \bigl( 1 + |z| + |w| \bigr).
\end{equation*}
\end{thmintro}
To our knowledge, this is the first pointwise linear bound on the score of a kinetic diffusion with a nonlinear drift and the first one on the derivative in the terminal variable $w$ except for the linear case where the kernel is Gaussian. See Section~\ref{sec:lit_review} for more details.
To prove it, we follow Sheu's approach~\cite[Theorem~B]{sheu1991_someestimates}, which deals with non-degenerate dynamics with time-independent coefficients. The argument is based on Fleming's logarithmic transformation \cite{fleming1985stochastic}, which regards $J(t,z) := -\log p^{\bs{\mu}}(t,T;z, w)$ as the value function of a stochastic control problem, for a fixed $w$.

This approach requires the regularity of $(s,t,z,w)\mapsto p^{\bs{\mu}}(s,t;z,w)$, namely $C^3$ in $(z,w)$ and $C^1$ in $(s,t)$, together with bounds on $\nabla_{z,w}p^{\bs\mu}$ that are uniform in $\bs{\mu}$. This cannot be taken from H\"ormander's theorem~\cite{hormander}: it requires coefficients that are smooth in $(t,z)$, whereas for a fixed $\bs{\mu}$ the drift $(t,z)\mapsto f(t,z,\mu_t)$ is merely continuous in $t$. The versions of the theorem that allow coefficients which are only measurable in $t$~\cite{Chaleyat1984hypoelipticity} or H\"older continuous in $t$~\cite{cattiaux2002hypoeliptic} are not applicable either, as they require the restricted H\"ormander condition, in which Lie brackets with the drift are not allowed. This condition fails for kinetic dynamics, whose noise acts on the velocity only. Since we could not find a suitable result for our case, we give a self-contained proof in Appendix~\ref{appendix:smoothness} based on Malliavin calculus. We apply Malliavin integration by parts to the driftless kinetic process, whose Malliavin covariance matrix is deterministic and explicitly invertible, and then restore the drift by Girsanov's theorem. Here, the boundedness of $f$ and the additive noise structure of the reference process play a key role.

\subsection{Related literature} \label{sec:lit_review}
\paragraph{Mean-field planning problem.}
Introduced independently by Lasry and Lions \cite{lasry2007mean} and by Huang et al.\ \cite{huang2006large}, MFGs have been recognized as a powerful framework for modeling large systems of interacting agents. The theory analyzes the limiting behavior of a dynamic crowd as the population size $N \to \infty$. As a characterization of the MFE, a coupled system of the HJB equation and the Fokker--Planck (FP) equation was introduced in \cite{lasry2007mean}. The former characterizes the optimality of the representative agent, while the latter describes the flow of the distribution of the agent's state.

Formally, our problem corresponds to the following PDE system (see Remark~\ref{remark:cost_coupling}):
\begin{equation} \label{eq:pde}
\left\{ \begin{aligned}
&\partial_t u + v \cdot \nabla_x u + \tfrac{1}{2}\Delta_v u + H\bigl(t,x,v,\nabla_v u,\rho(t,\cdot)\bigr) = 0, \\
&\partial_t \rho + \nabla_x \cdot (v\rho) + \nabla_v \cdot \bigl( \rho\, \nabla_p H\bigl(t,x,v,\nabla_v u,\rho(t,\cdot)\bigr) \bigr) = \tfrac{1}{2}\Delta_v \rho, \\
&\rho(0,\cdot) = \mu_{\mathrm{init}}, \qquad \rho(T,\cdot) = \mu_{\mathrm{fin}},
\end{aligned} \right.
\end{equation}
where the Hamiltonian $H$ is given by:
\[
H(t,z,p,\rho) \;:=\; \inf_{\beta\in\mathbb{R}^d}\Bigl\{\, p\cdot\beta + \tfrac12|\beta - f(t,z,\rho)|^2 \Bigr\}
\;=\; p\cdot f(t,z,\rho) - \tfrac12|p|^2.
\]
Here, $\rho(t,\cdot)$ plays the role of $\mu_t$ and the optimal control is given in feedback form by $\hat{\alpha}_t = -\nabla_v u(t, X_t, V_t)$. In contrast to the system arising in the standard MFG literature, the terminal condition is imposed on the FP equation, $\rho(T,\cdot) = \mu_{\mathrm{fin}}$, rather than on the value function $u$. In this sense, our problem is a variant of the mean-field planning problem introduced by Lions~\cite{lions2009cours}, who obtained the well-posedness of the first-order case and the non-degenerate second-order case with an (essentially) quadratic Hamiltonian. Without an a priori terminal cost $g$, the problem can be regarded as an inverse problem: find a terminal cost $g$ such that the population reaches the terminal distribution $\mu_{\mathrm{fin}}$ at time $T$ as the outcome of the non-cooperative individual optimization with $u(T,\cdot) = g$.
A related inverse incentive problem is studied in~\cite{bayraktar2021terminal} using Schr\"odinger bridges to characterize equilibria of terminal ranking games and the rank-based rewards implementing prescribed attainable terminal distributions. Their setting features non-degenerate Brownian dynamics and interaction through rewards, whereas our model has kinetic dynamics and interaction through the drift.

Since its introduction by Lions~\cite{lions2009cours}, the mean-field planning problem has been studied in a variety of settings; see, for instance, \cite{poretta2014planning, Achdou2012planning, orrieri2019variational_planning, graber2019planning_masstransport, gomes2018firstorder_planning, francesco2022approx_planning, benamou2019, ren2023pathdependent_mfp, chen2018steering, mimikos2024regularity_mfp, cardaliaguet2024freeboundary}. To our knowledge, however, none of these works allows a degenerate diffusion, except for the linear-quadratic model of \cite{chen2018steering}, where the dynamics is linear, the interaction is through the mean, and the solution is explicit.

In the existing literature on MFP problems, existence of an equilibrium is obtained either by approximating the problem with penalized MFG systems (\ref{eq:soft_SB_mu} in our case), and passing to the limit \cite{poretta2014planning}, or by exploiting a variational structure \cite{orrieri2019variational_planning, graber2019planning_masstransport}. In the linear-quadratic setting, \cite{chen2018steering} constructs the solution explicitly by exploiting the well-posedness of the SB problem for linear dynamics.
Some of those works establish the uniqueness of the equilibrium with a separable Hamiltonian under the Lasry--Lions monotonicity of the cost coupling, but uniqueness may fail without such structure.

\paragraph{Variants of the Schr\"odinger bridge problem.}
The Schr\"odinger bridge problem was originally posed by Schr\"odinger
\cite{schrodinger1931umkehrung, schrodinger1932theorie} as an inverse problem to find the most
likely evolution of a cloud of independent particles between two observed configurations; see
\cite{leonard_survey_2014} for a great survey. The problem has recently attracted renewed attention because of its applicability to computational optimal transport \cite{cuturi2013sinkhorn} and generative modeling \cite{de2021diffusion}, among others.
While most of this literature deals with non-degenerate reference processes, Chiarini et
al.~\cite{chiarini_etal_kineticSB} are among the few to study the kinetic Schr\"odinger bridge
problem, in which the reference process is the underdamped Langevin dynamics. Among other things, they quantified the turnpike property of the marginals of the optimally controlled process.

For interacting particles, Backhoff et al.\ \cite{backoff_etal_menafield_schrodinger} introduced the MFSB problem. It is shown in \cite{hernandez_tangpi_2025} that a relaxed version of the MFSB arises as the $N \to \infty$ limit of constrained cooperative control problems, in which the $N$ players minimize the average cost
$\frac{1}{2N} \sum^N_{i=1}\E \left[ \int^T_0 |\alpha^i_t|^2 dt \right]$
subject to a terminal constraint. In the limit this leads to a McKean--Vlasov control problem, whereas our non-cooperative setting leads to a classical stochastic control problem coupled with a fixed-point condition. While preparing this manuscript, we became aware of the recent manuscript \cite{eldesoukey2026schr}, which studies the cooperative counterpart of our problem for the Cucker--Smale dynamics and proposes a numerical algorithm for computing its solution. In the present paper, we keep the drift $f$ general so we can cover other models too.

\paragraph{Consistency error and computational approaches to MFGs.}
The consistency error introduced in Theorem~\ref{thm:informal_consistency} quantifies the extent to which the environment $\bs{\mu}$ is reproduced by the flow of marginal laws of the dynamics controlled by a certain strategy given the environment. In the optimization formulation of \cite{guo2024mfomo}, the violation of the consistency condition enters the objective function, which characterizes the (possibly non-unique) MFE as its minimizers. The $L^2$-distance between a flow and its best response is used in \cite{yu2024conv_mfp} to analyze fictitious play \cite{cardaliguet2017fictitious}. There it is shown that, under the Lasry--Lions strong monotonicity of the cost coupling, which in particular yields uniqueness, together with the strong convexity of the running cost, this distance is comparable to the distance to the unique equilibrium and controls the exploitability \cite{perrin2020fictitious}. No such monotonicity is available in the present setting, and thus we believe the consistency error is the natural quantity for which a rate can still be established without uniqueness.

\paragraph{Estimates on the kinetic score.}
For kinetic diffusions with nonlinear drifts, various estimates are already available: gradient estimates and Harnack inequalities for the associated semigroup $(P_t)_{t \ge 0}$~\cite{wang2013derivative,baudoin2017bakry,baudoin2020gradient, zhang2010bismut, guillin2012_degenerate_FP}, two-sided Gaussian bounds on the kernel~\cite{konakov2010explicit,delarue_density_estimates,ren2025heat}, and bounds on its derivatives~\cite{kusuoka1985applicationsII,chaudru2023heat}.
However, pointwise estimates on the score $\nabla\log p$ are scarce.
Bismut-type formulas~\cite{zhang2010bismut,guillin2012_degenerate_FP,wang2013derivative}
control the score only in an averaged sense: the gradient bounds they provide
for $P_tf$ amount to bounds on $\nabla_z\log p_t(z,\cdot)$ in $L^q(p_t(z,\cdot))$,
$q<\infty$.
\cite[Theorem~4.2]{zhang2010bismut} represents $\nabla_z\log p_T(z,w)$ pointwise as a conditional expectation along the bridge, without estimating it.
To the best of our knowledge, outside the linear case~\cite{chen2018optimal,hou2026heat}, no bound of linear growth on the score of the transition kernel is available for kinetic diffusions with a nonlinear drift.

\subsection{Organization of the paper}
The rest of the paper is organized as follows.
Section~\ref{sec:setting} introduces the setting, the assumptions and some preliminary results, and states our four main results precisely. Section~\ref{sec:strategy} gives an overview and intuitive understandings of the proof of each theorem and explains where the main difficulties lie.
Section~\ref{sec:proofs} contains the detailed proofs, and Section~\ref{sec:discussion} discusses open problems and future directions. Appendix~\ref{appendix:deferred_proofs} collects the proofs of the lemmas stated without proof in the main text, and Appendix~\ref{appendix:smoothness} proves the smoothness of the transition density.

\section{Settings and main results}
\label{sec:setting}
\subsection{Notation}

Throughout the paper we fix a finite time horizon $T > 0$.

\paragraph{Wasserstein spaces.}
Let $(\mcal{X}, d_{\mcal{X}})$ be a Polish space and let $p \ge 1$.
We denote by $\mcal{P}(\mcal{X})$ the set of all Borel probability measures on
$\mcal{X}$. Fixing a reference point $x_0 \in \mcal{X}$, we set
\begin{equation*}
  M_p(\mu) := \left( \int_{\mcal{X}} d_{\mcal{X}}(x, x_0)^p \, \mu(dx) \right)^{1/p},
  \qquad \mu \in \mcal{P}(\mcal{X}),
\end{equation*}
and define
$\mcal{P}_p(\mcal{X}) := \{ \mu \in \mcal{P}(\mcal{X}) : M_p(\mu) < \infty \}$,
which by the triangle inequality does not depend on the choice of $x_0$.
We equip $\mcal{P}_p(\mcal{X})$ with the $p$-Wasserstein distance
\begin{equation*}
  \mcal{W}_p(\mu, \nu)
  := \left( \inf_{\pi \in \Pi(\mu,\nu)}
       \int_{\mcal{X} \times \mcal{X}} d_{\mcal{X}}(x,y)^p \, \pi(dx,dy) \right)^{1/p},
\end{equation*}
where $\Pi(\mu,\nu)$ is the set of couplings of $\mu$ and $\nu$.
For different Polish spaces $\mcal{X},\tilde{\mcal{X}}$, we use the same notation $M_p$ and $\mcal{W}_p$, by slight abuse of notation.

\paragraph{Path space.}
We take as sample space the space of continuous paths
$\Omega := C([0,T]; \R^{2d})$, equipped with the supremum norm
$\| \omega \|_\infty := \sup_{t \in [0,T]} | \omega_t |$, and we let
$\mcal{F} := \mcal{B}(\Omega)$ be its Borel $\sigma$-algebra, so that
$(\Omega, \mcal{F})$ is a measurable space.
We write $Z = (Z_t  )_{t \in [0,T]}$ for the canonical process
$Z_t(\omega) := \omega_t$, and $\mbb{F} = (\mcal{F}_t)_{t \in [0,T]}$ for the
canonical filtration $\mcal{F}_t := \sigma(Z_s : 0 \le s \le t)$; whenever a
probability measure $P$ on $(\Omega, \mcal{F})$ is fixed, $\mbb{F}$ is understood
to be augmented by the $P$-null sets and made right-continuous, so that it
satisfies the usual conditions.
Since $\Omega$ is a separable Banach space, we always take $x_0 = 0$ in the definition of
$M_p$ on $\Omega$; thus $\mcal{P}_1(\Omega)
= \{ P \in \mcal{P}(\Omega) : \int_\Omega \|\omega\|_\infty \, P(d\omega) < \infty \}$.

\paragraph{Marginals and relative entropy.}
For $P \in \mcal{P}(\Omega)$ and $t \in [0,T]$ we write
$P_t := (Z_t)_{\#} P \in \mcal{P}(\R^{2d})$ for the time-$t$ marginal of $P$, and
$(P_t)_{t \in [0,T]}$ for the associated flow of marginals.
Given $P, R \in \mcal{P}(\Omega)$, the relative entropy of $P$ with respect to $R$
is
\begin{equation*}
  H(P \,|\, R) :=
  \begin{cases}
    \displaystyle \int_\Omega \log \frac{dP}{dR} \, dP, & P \ll R, \\[2mm]
    +\infty, & \text{otherwise.}
  \end{cases}
\end{equation*}

\subsection{Hard constraint problem}\label{sec:hard_constraint}
Let $f : [0,T] \times \R^d \times \R^d \times \mcal{P}_1(\R^{2d}) \to \R^d$, and let
$\mu_{\mathrm{init}}, \mu_{\mathrm{fin}} \in \mcal{P}(\R^{2d})$ denote the initial and the target
distributions. Given a flow of probability measures
$\bs{\mu} = (\mu_t)_{t \in [0,T]} \in C([0,T]; \mcal{P}_1(\R^{2d}))$, we consider the
kinetic dynamics
\begin{equation}
\label{eq:reference_SDE}
     dX_t = V_t \, dt, \quad
     dV_t = f(t, X_t, V_t, \mu_t) \, dt + dW_t, \quad
     (X_0, V_0) \sim \mu_{\mathrm{init}},
\end{equation}
where $W$ is a standard $d$-dimensional Brownian motion (BM).

Under Assumption~\ref{assump:coefficient} below, \eqref{eq:reference_SDE} admits a unique
weak solution, and we denote by $R^{\bs{\mu}} \in \mcal{P}(\Omega)$ its law on path
space; equivalently, $R^{\bs{\mu}}$ is the unique probability measure on
$(\Omega, \mcal{F})$ under which the canonical process
$Z = (X, V)$, $Z_t(\omega) := \omega_t$, solves \eqref{eq:reference_SDE}.

The Schr\"odinger bridge problem associated with $\bs{\mu}$ is
\begin{equation*} \label{eq:SB_mu}  \tag*{SB($\bs{\mu}$)}
  \inf \Bigl\{ H\bigl(Q \,\big|\, R^{\bs{\mu}}\bigr)
  \;:\; Q \in \mcal{P}(\Omega), \ Q_0 = \mu_{\mathrm{init}}, \ Q_T = \mu_{\mathrm{fin}} \Bigr\},
\end{equation*}

It is established by \cite{Follmer_randomfields} that \ref{eq:SB_mu} has the same value as the minimal action problem:
\begin{equation*}
\inf\E^{Q} \left[ \frac{1}{2} \int^T_0 |\alpha_t|^2 \, dt
    \right]
    \label{eq:minimial_action}
\end{equation*}
where the infimum is taken over $\R^d$-valued $\mbb{F}$-progressively measurable processes $\alpha$ and $Q \in \mcal{P}(\Omega)$ such that $Z=(X,V)$ is the solution to the controlled SDE
\begin{equation}
\label{eq:controlled_SDE}
d X_t = V_t dt, \quad
d V_t = \left(f(t, X_t, V_t, \mu_t) + \alpha_t \right) dt +  dW^Q_t,
\end{equation}
subject to $Q_0 = \mu_{\mathrm{init}}$ and $Q_T = \mu_{\mathrm{fin}}$. Here,
$
  W^{Q}_t := V_t - V_0
  - \int_0^t \bigl( f(s, X_s, V_s, \mu_s) + \alpha_s \bigr) \, ds,
$
is a standard $d$-dimensional $(\mbb{F}, Q)$-BM.

Although the following result is standard, similar arguments will be used later in the paper (see Lemma~\ref{lemma:P_continuous} and Corollary~\ref{cor:Psi_equicontinuous}).
\begin{lemma} \label{lemma:entropy_is_kinetic_energy}
For each $Q \ll R^{\bs{\mu}}$ with $Q_0 = R^{\bs{\mu}}_0$ and $H(Q |  R^{\bs{\mu}}) < \infty$, there exists an $\mbb{F}$-progressively measurable process $\alpha$ such that
\begin{equation*}
    H(Q | R^{\bs{\mu}}) = \E^Q\left[  \frac{1}{2} \int^T_0 |\alpha_t|^2 dt \right].
\end{equation*}
\end{lemma}
\begin{proof}
By Girsanov's theorem, there exists some $\mbb{F}$-progressively measurable process $\alpha$ such that $\frac{dQ}{dR^{\bs{\mu}}} = \mcal{E}(\int^T_0 \alpha_t \, dW^{R^{\bs{\mu}}}_t), \, Q\text{-a.s.}$, where $\mcal{E}(M)$ denotes the stochastic exponential for a continuous local martingale $M$. If $H(Q|R^{\bs{\mu}}) < \infty$, then $\E^Q \int^T_0 |\alpha_t|^2 \, dt < \infty$ by \cite[Lemma~2.6]{follmer_entropy}. Therefore, we compute
\begin{equation*}
    H(Q | R^{\bs{\mu}})
        = \int \log \frac{dQ}{d R^{\bs{\mu}}}\, dQ
        = \E^Q\left[ \int^T_0 \alpha_t \, dW^{Q}_t + \frac{1}{2} \int^T_0 |\alpha_t|^2 dt\right]
        = \E^Q\left[  \frac{1}{2} \int^T_0 |\alpha_t|^2 dt\right].         \qedhere
\end{equation*}
\end{proof}

\begin{remark}[Equivalence to an MFP with a cost coupling]
\label{remark:cost_coupling}
Let $R^0 \in \mcal{P}(\Omega)$ denote the law of the driftless kinetic process
\begin{equation}
\label{eq:driftless}
d X_t = V_t \, dt, \qquad
d V_t = dW^{R^0}_t ,
\end{equation}
where $W^{R^0}$ is a $d$-dimensional Brownian motion under $R^0$.
Let $Q \in \mcal{P}(\Omega)$ satisfy $Q_0 = \mu_{\mathrm{init}}$ and $H(Q \,|\, R^0) < \infty$, and let $\tilde{\alpha}$ be the drift of $V$ under $Q$, i.e.\ $dV_t = \tilde{\alpha}_t \, dt + dW^Q_t$ for a $Q$-Brownian motion $W^Q$, as in Lemma~\ref{lemma:entropy_is_kinetic_energy}.
Let $f$ be bounded as Assumption~\ref{assump:bounded_derivative}.
Since $R^{\bs{\mu}}$ and $R^0$ are mutually absolutely continuous by Girsanov's theorem, with
$\log \frac{dR^{\bs{\mu}}}{dR^0} = \int_0^T f(t,X_t,V_t,\mu_t) \cdot dW^{R^0}_t - \frac12 \int_0^T |f(t,X_t,V_t,\mu_t)|^2 \, dt$,
and since $W^{R^0}_t = V_t - V_0 = W^Q_t + \int_0^t \tilde{\alpha}_s \, ds$ under $Q$, we obtain
\begin{align*}
H(Q \,|\, R^{\bs{\mu}})
&= H(Q \,|\, R^0) - \int_\Omega \log \left( \frac{dR^{\bs{\mu}}}{dR^0}\right) dQ \\
&= \E^Q \left[ \frac{1}{2} \int^T_0 |\tilde{\alpha}_t|^2 \, dt -  \int^T_0 f(t,X_t, V_t, \mu_t) \cdot dW^{R^0}_t + \frac{1}{2} \int^T_0  |f(t,X_t, V_t, \mu_t)|^2 \, dt  \right] \\
&= \E^Q \left[ \frac{1}{2} \int^T_0 |\tilde{\alpha}_t|^2 \, dt -  \int^T_0 f(t,X_t, V_t, \mu_t) \cdot \bigl( dW^{Q}_t + \tilde{\alpha}_t \, dt \bigr) + \frac{1}{2} \int^T_0  |f(t,X_t, V_t, \mu_t)|^2 \, dt \right] \\
&=  \E^Q \left[ \frac{1}{2} \int^T_0 \bigl|\tilde{\alpha}_t - f(t,X_t, V_t, \mu_t)\bigr|^2 \, dt \right],
\end{align*}
where we have used the martingality of the stochastic integral with respect to $W^Q$.
Hence $\mathrm{SB}(\bs{\mu})$ is equivalent to the stochastic control problem for the driftless kinetic process \eqref{eq:driftless}, in which the velocity drift $\tilde{\alpha}$ is the control and the running cost is $\frac12 |\tilde{\alpha}_t - f(t,X_t,V_t,\mu_t)|^2$; that is, the interaction enters only through the cost, which penalizes the deviation of the control from $f$.
The Hamiltonian associated with \eqref{eq:pde} incorporates this cost function. Note that $H$ is not separable in $(p,\mu)$ because of the cross term $p \cdot f(t,x,v,\mu)$.
\end{remark}

We introduce the assumptions on the boundary constraints:
\begin{assump}
\label{assump:initial_and_target}
$\,$
\begin{enumerate} [label=\textup{(\roman*)}, ref=\textup{(\roman*)}]
    \item \label{assump:densities}
      $\mu_{\mathrm{init}}, \mu_{\mathrm{fin}} \in \mcal{P}_1(\R^{2d})$ are absolutely continuous with
      respect to the Lebesgue measure on $\R^{2d}$; we denote their densities by
      $p_{\mathrm{init}}$ and $p_{\mathrm{fin}}$, so that $\mu_{\mathrm{init}}(dz) = p_{\mathrm{init}}(z) \, dz$
      and $\mu_{\mathrm{fin}}(dz) = p_{\mathrm{fin}}(z) \, dz$.

    \item \label{assump:log_bound}
    $\log p_{\mathrm{fin}}\in L^1(\mu_{\mathrm{fin}})$.

    \item \label{assump:second_moment_and_exp_moment}
    There exists $\theta >0$ such that
    $\int_{\R^{2d}} e^{\theta |z|^2 }\mu_{\mathrm{init}} (dz)  < \infty$, and $M_2(\mu_{\mathrm{fin}}) < \infty$.
\end{enumerate}
\end{assump}
\begin{remark}
Let us elaborate on the assumptions. The assumptions \ref{assump:densities} and \ref{assump:log_bound} are to guarantee the well-posedness of the problem \ref{eq:SB_mu}, which we prove in Proposition~\ref{prop:well-posedness_SB_mu}.
The assumption \ref{assump:second_moment_and_exp_moment} is used to prove Lemma~\ref{lemma:bdd_exp_moment}, which eventually leads to the uniform bound on the second moments of $P^{\bs{\mu}}$ in Proposition~\ref{prop:bdd_2nd_moment}.
\end{remark}

The following is the assumption for the drift coefficient of the velocity process:
\begin{assump}
\label{assump:coefficient}
Let $f: [0,T] \times \R^{2d}  \times \mcal{P}_1(\R^{2d}) \to \R^{d}$ be a jointly measurable function.
There exists $C_f>0$ such that:
\begin{enumerate} [label=\textup{(\roman*)}, ref=\textup{(\roman*)}]
    \item \label{assump:lipschitz}
    For any $z, w \in \R^{2d}$ and $\mu, \nu \in \mcal{P}_1(\R^{2d})$,
    \begin{equation*}
        | f(t,z,\mu) -  f(t,w,\nu)| \leq C_f  \left( |z-w| + \mcal{W}_1(\mu, \nu) \right)
    \end{equation*}
    for any $t \in [0,T]$.

    \item \label{assump:linear_growth} For any $z \in \R^{2d}$ and $\mu \in \mcal{P}_1(\R^{2d})$, it holds
    \begin{equation*}
        | f(t,z,\mu)| \leq C_f(1 + |z|),
    \end{equation*}
    for any $t \in [0,T]$.
\end{enumerate}
\end{assump}
\begin{remark}
    From Assumption \ref{assump:coefficient}, the standard well-posedness result guarantees that the reference equation \eqref{eq:reference_SDE} has a unique strong solution, and thus the reference path measure $R^{\bs{\mu}} \in \mcal{P}(\Omega)$ is well-defined. The independence of $C_f$ from $\bs{\mu}$ is inherited by the constant $C$ in Lemma~\ref{lemma:kernel_bound}, $M$ in \eqref{eq:bound_entropy_1}, and consequently $K$ in Proposition~\ref{prop:bdd_2nd_moment}. This fact plays a crucial role in applying Schauder's fixed point theorem.
\end{remark}

\begin{example}\label{example:Cucker--Smale}
Let us consider the variation of the Cucker--Smale model:
\begin{equation*}
f(x, v, \mu) := \int_{\R^{2d}} \Phi(x - y) \Gamma(w - v) \mu(dy,dw),
\end{equation*}
where
\begin{equation*}
    \Phi(x) = \dfrac{\kappa}{(1 + |x|^2)^\beta} , \quad
    \Gamma(v) = \dfrac{v}{\sqrt{1 + |v|^2}} .
\end{equation*}
Both $\Phi$ and $\Gamma$ are bounded, uniformly Lipschitz functions, and the Lipschitz constant of $f$ only depends on these functions.
\end{example}

We first state the existence of the transition density and its lower bound. The proof is provided in Appendix~\ref{sec:proof_of_kernel_lowerbound}.
\textbf{Throughout the paper, $C$ denotes a generic positive constant that may change from line to line and is independent of $\bs\mu$.}

As in \cite[Lemma~2.1]{chiarini_etal_kineticSB}, we rely on the Gaussian estimates of the transition kernel \cite[Theorem~1.1]{delarue_density_estimates} to obtain the following.
\begin{lemma}
\label{lemma:kernel_bound}
Let Assumption~\ref{assump:coefficient} hold true.
Then, for every
$\bm{\mu}\in C\big([0,T];\mathcal{P}_1(\mathbb{R}^{2d})\big)$ and every starting
point $z\in\mathbb{R}^{2d}$, the law of $Z_T$ started from $Z_0=z$ admits a
density $w\mapsto p^{\bm{\mu}}_T(z,w)$.
Furthermore, there exists a constant $C = C( C_f,T, d) > 0$ independent of $\bs{\mu}$, such that
\begin{equation}
p^{\bs{\mu}}_T( z , w )
\geq \frac{1}{C} \exp \left( - C \left( |z|^2  +  |w|^2 \right) \right),
\quad \forall \bs{\mu} \in C([0,T] ; \mcal{P}_1(\R^{2d})).
\end{equation}
\end{lemma}

Granting this lemma, we obtain the well-posedness of \ref{eq:SB_mu}:
\begin{proposition} \label{prop:well-posedness_SB_mu}
Fix $\bs {\mu} \in C([0,T];\mcal{P}_1(\R^{2d}))$. Under Assumptions \ref{assump:initial_and_target} and \ref{assump:coefficient}, there exists a unique solution $P^{\bs{\mu}} \in \mcal{P}(\Omega)$ to \textup{\ref{eq:SB_mu}}.
\end{proposition}
\begin{proof}
By \cite[Lemma~2.4]{leonard_survey_2014},
we know the value of \ref{eq:SB_mu} is the same as that of the following static problem:
\begin{equation*}
\inf \left\{ H (\pi |R^{\bs{\mu}}_{0,T} ) \,;\, \pi \in \mcal{P}(\R^{2d}\times\R^{2d}),   \pi_0 = \mu_{\mathrm{init}}, \pi_T = \mu_{\mathrm{fin}} \right\},
\end{equation*}
and it suffices to show that this is finite. In particular, we prove that $H( \mu_{\mathrm{init}} \otimes \mu_{\mathrm{fin}}| R^{\bs{\mu}}_{0,T})<\infty$. By Lemma~\ref{lemma:kernel_bound}, $R_{0,T}^{\bs{\mu}}$ has a positive density and is equivalent to $\mu_{\mathrm{init}} \otimes \textup{Leb}_{\R^{2d}}$. Using Assumption \ref{assump:initial_and_target} \ref{assump:densities}, $\mu_{\mathrm{init}} \otimes \mu_{\mathrm{fin}} \ll R_{0,T}^{\bs{\mu}}$.
Therefore,
\begin{align}
H( \mu_{\mathrm{init}} \otimes \mu_{\mathrm{fin}}| R^{\bs{\mu}}_{0,T})
    &= \int \log \frac{p_{\mathrm{init}} (z) p_{\mathrm{fin}}(w)} {p_{\mathrm{init}} (z)  p_T^{\bs{\mu}}(z,w)}  d\mu_{\mathrm{init}} \otimes \mu_{\mathrm{fin}}(z,w) \notag \\
    &= \int \log p_{\mathrm{fin}} \, d\mu_{\mathrm{fin}} - \int \log p_T^{\bs{\mu}} \, d\mu_{\mathrm{init}} \otimes \mu_{\mathrm{fin}} \notag \\
    &\leq \int \log p_{\mathrm{fin}} \, d\mu_{\mathrm{fin}} - \int \log \frac{1}{C}e^{-C (|z|^2 + |w|^2)} \, d\mu_{\mathrm{init}} \otimes \mu_{\mathrm{fin}} \notag \\
    &\leq C(1 + M^2_2(\mu_{\mathrm{init}}) + M^2_2(\mu_{\mathrm{fin}})) =: M < \infty, \label{eq:bound_entropy_1}
\end{align}
where we have used Assumption \ref{assump:initial_and_target} \ref{assump:log_bound}-\ref{assump:second_moment_and_exp_moment} and Lemma~\ref{lemma:kernel_bound}.
This constant $M$ is independent of the choice of $\bs{\mu}$.
\end{proof}
An immediate consequence of the above proof is the following:
\begin{corollary} \label{cor:bdd_relative_entropy}
There exists a constant $M = M(C_f, T, d, M_2(\mu_{\mathrm{init}}), M_2(\mu_{\mathrm{fin}}), H(\mu_{\mathrm{fin}}|\leb) ) > 0$ independent of the choice of $\bs{\mu} \in C([0,T]; \mcal{P}_1(\R^{2d})) $ such that
\begin{equation}
    H(P^{\bs{\mu}} | R^{\bs{\mu}}) \leq M, \quad \forall \bs{\mu} \in C([0,T]; \mcal{P}_1(\R^{2d})).
\end{equation}
\end{corollary}

\begin{definition}
For $\bs{\mu} \in C\left([0,T] ; \mcal{P}_1(\R^{2d}) \right)$, let $P^{\bs{\mu}}$ be the unique solution of \textup{\ref{eq:SB_mu}}.
We call $\bs{\mu}$  a mean-field equilibrium, if
\begin{equation}
    \mu_t = {P}^{\bs{\mu}}_t, \forall t \in [0,T].
\end{equation}
\end{definition}

We state our first main result:
\begin{theorem}
\label{thm:existence_mfe}
Let Assumptions~\ref{assump:initial_and_target} and \ref{assump:coefficient} be in force.
Then, there exists a mean-field equilibrium $\bs{\mu} \in C\left([0,T] ; \mcal{P}_1(\R^{2d}) \right)$.
\end{theorem}

\begin{remark}[On uniqueness/non-uniqueness of MFE]
\label{remark:non-uniqueness}
Uniqueness of the MFE fails in general under our assumptions: Appendix~\ref{appendix:non-uniqueness} provides a drift $f$ satisfying Assumptions~\ref{assump:coefficient} and~\ref{assump:bounded_derivative} for which a continuum of MFEs exists.
Nonuniqueness in the absence of monotonicity also occurs in finite-state mean field games; see~\cite{bayraktar2020nonuniqueness} for a two-state model with anti-monotone running costs.
Whether uniqueness holds for the modified Cucker--Smale model of Example~\ref{example:Cucker--Smale} remains open. The standard sufficient conditions do not apply to it: from the form of the Lagrangian $\frac12|\beta - f(z,\mu)|^2$ of Remark~\ref{remark:cost_coupling}, one checks that neither the Lasry--Lions monotonicity condition nor the displacement monotonicity condition of~\cite{meszaros2024mean} is satisfied by this model.
\end{remark}

\begin{remark}[Correspondence between MFE and solutions of the MFP]
\label{remark:MFP_SB}
We make precise the correspondence, stated in Section~\ref{sec:intro}, between MFE and solutions of the MFP.
Fix a flow $\bs{\mu} \in C([0,T];\mcal{P}_1(\R^{2d}))$ and write $R := R^{\bs{\mu}}$, $P := P^{\bs{\mu}}$, $p_T := p^{\bs{\mu}}_T$ and $R_{0,T} := (Z_0,Z_T)_{\#}R$, so that $R_{0,T}(dz,dw) = \mu_{\mathrm{init}}(dz)\,p_T(z,w)\,dw$.

\smallskip\noindent
\textit{(i) Schr\"odinger potentials.}
Denoting $\pi := (Z_0,Z_T)_{\#}P$, we have $dP/dR = (d\pi/dR_{0,T})(Z_0,Z_T)$, and $\pi$ solves the EOT problem between $\mu_{\mathrm{init}}$ and $\mu_{\mathrm{fin}}$ with cost $c := -\log p_T$  (see Section~\ref{sec:sketch_consistency_error}).
By Lemma~\ref{lemma:kernel_bound} and the Gaussian upper bound in \cite[Theorem~1.1]{delarue_density_estimates}, $|c(z,w)| \le C(1+|z|^2+|w|^2)$, so $c \in L^1(\mu_{\mathrm{init}} \otimes \mu_{\mathrm{fin}})$ and \cite[Theorem~4.2]{nutz2021introduction} yields measurable functions $\varphi, \psi^{\mathrm{fin}} \colon \R^{2d} \to (0,\infty)$ with $\log\varphi \in L^1(\mu_{\mathrm{init}})$, $\log\psi^{\mathrm{fin}} \in L^1(\mu_{\mathrm{fin}})$, unique a.e.\ up to $(a\varphi, \psi^{\mathrm{fin}}/a)$, $a > 0$, such that $\pi = (\varphi \otimes \psi^{\mathrm{fin}})\, e^{-c}\,\mu_{\mathrm{init}} \otimes \mu_{\mathrm{fin}}$.
Since $R_{0,T} = e^{-c}\,\mu_{\mathrm{init}} \otimes \mathrm{Leb}$, this is the product form $dP/dR = \varphi(Z_0)\,\psi(Z_T)$ with
\begin{equation*}
  \varphi > 0, \qquad \psi := \psi^{\mathrm{fin}}\, p_{\mathrm{fin}} \ge 0, \qquad \{\psi > 0\} = \{p_{\mathrm{fin}} > 0\}.
\end{equation*}
Accordingly the incentive $\hat{g} := -\log\psi \colon \R^{2d} \to (-\infty,+\infty]$ equals $+\infty$ outside the support of $\mu_{\mathrm{fin}}$, and $\hat{g} \in L^1(\mu_{\mathrm{fin}})$ since $\log p_{\mathrm{fin}} \in L^1(\mu_{\mathrm{fin}})$ by Assumption~\ref{assump:initial_and_target}.

\smallskip\noindent
\textit{(ii) From an MFE to a solution of the MFP.}
Consider the unconstrained problem \eqref{eq:SB_unconstrained}, i.e.\ the minimization of
\begin{equation*}
  J(Q) := H(Q\,|\,R) + \E^{Q}\bigl[\hat{g}(Z_T)\bigr]
\end{equation*}
over $Q \in \mcal{P}(\Omega)$ with $Q_0 = \mu_{\mathrm{init}}$, with the convention $J(Q) := +\infty$ whenever $H(Q\,|\,R) = \infty$ or $Q(\hat{g}(Z_T) = +\infty) > 0$. We claim that $P$ is the unique minimizer of $J$; in particular the terminal constraint $Q_T = \mu_{\mathrm{fin}}$ is automatically satisfied by the minimizer.
Indeed, let $H(Q\,|\,R) < \infty$, $Q_0 = \mu_{\mathrm{init}}$ and $Q(\psi(Z_T) = 0) = 0$. Then $Q \ll P$, because $\varphi(Z_0) > 0$ $Q$-a.s.\ by $Q_0 = \mu_{\mathrm{init}}$, and
\begin{equation*}
  \log\frac{dQ}{dR} = \log\frac{dQ}{dP} + \log\varphi(Z_0) + \log\psi(Z_T), \qquad Q\text{-a.s.}
\end{equation*}
The left-hand side is $Q$-integrable since $H(Q\,|\,R) < \infty$, $\log\varphi(Z_0)$ is $Q$-integrable since $\log\varphi \in L^1(\mu_{\mathrm{init}})$, and $(\log\tfrac{dQ}{dP})^{-}$ is $Q$-integrable. Hence $(\log\psi(Z_T))^{+} \in L^1(Q)$, so that $\E^{Q}[\hat{g}(Z_T)]$ is well defined in $(-\infty,+\infty]$, and taking expectations gives
\begin{equation}
\label{eq:gibbs_identity}
  H(Q\,|\,R) + \E^{Q}\bigl[\hat{g}(Z_T)\bigr] = H(Q\,|\,P) + \int \log\varphi\, d\mu_{\mathrm{init}}.
\end{equation}
The last term is a finite constant independent of $Q$, so $J(Q) \ge J(P)$ with equality if and only if $Q = P$. Thus, once we obtain an MFE $\hat{\bs{\mu}}$, $\hat{g} =- \log \psi$ is an incentive function constituting a solution of the MFP as discussed in Section~\ref{sec:intro}.

\smallskip\noindent
\textit{(iii) From a solution of the MFP to an MFE.}
Conversely, let $(\hat{\alpha}, \hat{g}, \hat{\bs{\mu}})$ solve the MFP, where $\hat{g}$ is such that $\int \hat{g}\, d\mu_{\mathrm{fin}}$ is finite (otherwise \eqref{eqn:optimality} is void), and let $\hat{P}$ be the law of the controlled process. By Girsanov's theorem, $H(\hat{P}\,|\,R^{\hat{\bs{\mu}}}) = \E[\frac12\int_0^T |\hat{\alpha}_t|^2\,dt] < \infty$, and $\hat{P}_T = \hat{\mu}_T = \mu_{\mathrm{fin}}$ by \eqref{eqn:consistency} and \eqref{eqn:terminal}. Let $Q$ be admissible for $\mathrm{SB}(\hat{\bs{\mu}})$ with $H(Q\,|\,R^{\hat{\bs{\mu}}}) < \infty$. Then the control $\alpha^{Q}$ associated with $Q$ by Lemma~\ref{lemma:entropy_is_kinetic_energy} is admissible and $\E^{Q}[\hat{g}(Z_T)] = \int \hat{g}\,d\mu_{\mathrm{fin}} = \E^{\hat{P}}[\hat{g}(Z_T)]$, so that \eqref{eqn:optimality} reduces to $H(\hat{P}\,|\,R^{\hat{\bs{\mu}}}) \le H(Q\,|\,R^{\hat{\bs{\mu}}})$. Hence $\hat{P}$ solves $\mathrm{SB}(\hat{\bs{\mu}})$, $\hat{P} = P^{\hat{\bs{\mu}}}$ by uniqueness, and $\hat{\mu}_t = \hat{P}_t = P^{\hat{\bs{\mu}}}_t$ for all $t$: $\hat{\bs{\mu}}$ is an MFE.
\end{remark}

\subsection{Soft constraint problem}
For $\lambda > 0$, we consider the following soft constraint problem
\begin{equation*}
    \inf_{Q \in \mcal{P}(\Omega)} \left\{
    H(Q | R^{\bs{\mu}}) + \lambda G(Q_T) \,:\, Q_0 = \mu_{\mathrm{init}}
    \right\} \tag*{$\lambda$-SB($\bs{\mu}$)} \label{eq:soft_SB_mu}
\end{equation*}
for a terminal cost function $G$ satisfying the conditions:
\begin{assump} \label{assump:terminal_cost}
    $G: \mcal{P}(\R^{2d}) \to [0,\infty]$ satisfies the following:
\begin{enumerate}[label=\textup{(\roman*)}, ref=\textup{(\roman*)}]
    \item $G$ is convex,
    \item $G$ is l.s.c.\ with respect to the weak topology, and
    \item \label{assump:g_zero}$G(\nu) = 0 \iff \nu = \mu_{\mathrm{fin}}$.
\end{enumerate}
\end{assump}

The next lemma is used to guarantee the well-posedness of \ref{eq:soft_SB_mu}. The proof is provided in Appendix~\ref{appendix:proof_compact_sublevel}.
\begin{lemma} \label{lemma:compact_sublevel}
For any $c >0 $, the sublevel set $\left\{ Q \in \mcal{P}(\Omega) :  H(Q|R^{\bs{\mu}}) \leq c  \right\}$ is compact in the weak topology.
\end{lemma}
Thanks to this compactness result, \ref{eq:soft_SB_mu} is guaranteed to be well-posed.
\begin{proposition} \label{prop:wellposedness_soft_SB}
    Let Assumption~\ref{assump:initial_and_target}, \ref{assump:coefficient}, and \ref{assump:terminal_cost} hold. Then, for $\bs{\mu} \in C([0,T], \mcal{P}_1(\R^{2d}))$,  there exists a unique solution $P^{\lambda,\bs{\mu}} \in \mcal{P}(\Omega)$ to \textup{\ref{eq:soft_SB_mu}}.
\end{proposition}
\begin{proof}
The uniqueness is immediate from the strict convexity of the relative entropy and the convexity of $G$.
Moreover, the value of \ref{eq:soft_SB_mu} is finite. Indeed, take the solution $P^{\bs{\mu}}$ to \ref{eq:SB_mu}. Then, $H(P^{\bs{\mu}} |  R^{\bs{\mu}} ) < \infty$ and $G(P^{\bs{\mu}}_T) = 0$ by Assumption~\ref{assump:terminal_cost} \ref{assump:g_zero} since $P^{\bs{\mu}}_T = \mu_{\mathrm{fin}}$.

We now show the existence. Observe that the left hand side of the inclusion
\begin{equation*}
\left\{ Q \in \mcal{P}(\Omega) :  H(Q|R^{\bs{\mu}}) + \lambda G(Q_T)\leq c  \right\}
\subset \left\{ Q \in \mcal{P}(\Omega) :  H(Q|R^{\bs{\mu}}) \leq c  \right\}
\end{equation*}
is compact with respect to the weak topology.
Indeed, we know that the right hand side is compact (see, Lemma~\ref{lemma:compact_sublevel}), and thus it suffices to check that the left hand side is closed with respect to the weak topology, but this is immediate by $H$ and $G$ being l.s.c.

Take a minimizing sequence $\{P^n\}_{n\in \N} \subset \mcal{P}(\Omega)$ with $P^n_0 = \mu_{\mathrm{init}}$ such that
\begin{equation*}
\lim_{n\to \infty} ( H(P^n | R^{\bs{\mu}}) + \lambda G(P^n_T) )
= \inf_{Q \in \mcal{P}(\Omega)} \left\{
    H(Q | R^{\bs{\mu}}) + \lambda G(Q_T) \,:\, Q_0 = \mu_{\mathrm{init}}
    \right\}.
\end{equation*}
By the compactness, there exists a subsequence of $\{P^n\}_{n\in \N}$, which we still denote by the same, converging to some $P^\star \in \mcal{P}(\Omega)$ with $P^\star_0 = \mu_{\mathrm{init}}$ in the weak topology. This $P^\star$ is a solution to \ref{eq:soft_SB_mu}.
Indeed, since $H$ and $G$ are l.s.c., we have
\begin{equation*}
    H(P^\star | R^{\bs{\mu}}) + \lambda G(P^\star_T)
    \leq \liminf_{n \to \infty} ( H(P^n | R^{\bs{\mu}}) + \lambda G(P^n_T) )
    =\inf_{Q \in \mcal{P}(\Omega)} \left\{
    H(Q | R^{\bs{\mu}}) + \lambda G(Q_T) \,:\, Q_0 = \mu_{\mathrm{init}}
    \right\}. \qedhere
\end{equation*}
\end{proof}

As in the hard-constraint case, we define the equilibrium:
\begin{definition}
    We call $\bs{\mu} \in C([0,T]; \mcal{P}(\R^{2d}))$ a $\lambda$-mean-field equilibrium ($\lambda$-MFE), if it satisfies the following fixed point property:
    \begin{equation*}
        P^{\lambda,\bs{\mu}}_t = \mu_t, \quad \forall {t \in [0,T]}.
    \end{equation*}
\end{definition}

\begin{remark}[Relation with standard MFGs]\label{remark:potential_argument}
Problem \ref{eq:soft_SB_mu} is a control problem of McKean--Vlasov type, since the terminal cost depends on the law $Q_T$ of the controlled process at time $T$ rather than on its terminal state. Nevertheless, a $\lambda$-MFE is an equilibrium of a \textit{standard} MFG. Suppose that $G$ admits a linear functional derivative
$\frac{\delta G}{\delta \nu}\colon \mathcal{P}_2(\mathbb{R}^{2d})\times\mathbb{R}^{2d}\to\mathbb{R}$
(cf.\ \cite[Definition 5.43]{carmona2018probabilistic_I}).\footnote{For the purpose of this remark it suffices that the one-sided directional derivative $\frac{d}{d\varepsilon}G(\mu_T+\varepsilon(\nu-\mu_T))|_{\varepsilon=0^+}$ exists at the equilibrium terminal law $\mu_T$ and is represented by a function
$\frac{\delta G}{\delta\nu}(\mu_T,\cdot)$; for $G(\cdot)=H(\cdot\,|\,\mu_{\mathrm{fin}})$ this function is $\log\frac{d\mu_T}{d\mu_{\mathrm{fin}}}$.}
Then $\bs{\mu}$ is a $\lambda$-MFE if and only if it is an MFE of the MFG in which the representative agent solves
\begin{equation*}
\inf_{Q\in\mathcal{P}(\Omega)}
\left\{ H(Q\,|\,R^{\bs{\mu}}) + \lambda \int_{\mathbb{R}^{2d}} \frac{\delta G}{\delta \nu}(\mu_T, z)\, Q_T(dz)
\ :\ Q_0=\mu_{\mathrm{init}}\right\},
\end{equation*}
for a given flow $\bs{\mu}$.
By Girsanov's theorem, this is equivalent to the classical stochastic control problem:
\begin{equation*}
\inf_{\alpha}\
\mathbb{E}\left[
\int_0^T \frac{1}{2}|\alpha_t|^2\,dt + \lambda\, \frac{\delta G}{\delta \nu}(\mu_T, Z_T)
\right],
\end{equation*}
where $Z = (X,V)$ is the controlled kinetic process \eqref{eq:controlled_SDE} with the terminal cost $\lambda\frac{\delta G}{\delta\nu}(\mu_T,\cdot)$.

For instance, for the squared $L^2$-distance $G(\nu)=\frac12\|\nu-\mu_{\mathrm{fin}}\|_{L^2}^2$
(identifying measures with their densities), the terminal cost of the associated MFG is
$\lambda\,(\mu_T(x)-\mu_{\mathrm{fin}}(x))$. With $\lambda=1/\varepsilon$, this is exactly the
penalized planning problem of \cite{poretta2014planning, Achdou2012planning,
mimikos2024regularity_mfp}. Our choice $G(\cdot) =H(\cdot\,|\,\mu_{\mathrm{fin}})$ in Assumption~\ref{assump:Talagrand} below is its entropic counterpart, with terminal cost $\lambda\log\frac{d\mu_T}{d\mu_{\mathrm{fin}}}$.

The mechanism behind this equivalence is the same as in potential (or variational) MFGs, see \cite[Section~2.6]{lasry2007mean}.
Here only the terminal cost has this potential structure, while the interaction through the drift of $R^{\bs{\mu}}$ does not, which is why the existence of a $\lambda$-MFE still requires the fixed-point argument of Section~\ref{sec:proof_existence}.
\end{remark}

We are ready to state our second main theorem.
\begin{theorem} \label{thm:qualitative}
Let Assumptions~\ref{assump:initial_and_target}, \ref{assump:coefficient}, and \ref{assump:terminal_cost} hold.
For any $\lambda > 0$, there exists a $\lambda$-MFE $\bs{\mu} \in  C([0,T];\mcal{P}_1(\R^{2d}))$.
Additionally, for each $n$, let $\bs{\mu}^n \in  C([0,T];\mcal{P}_1(\R^{2d})) $ be an $n$-MFE.
Then, there exists a subsequence $\left\{\bs{\mu}^{n_k}\right\}_{k \in \N}$ such that $\bs{\mu}^{n_k} \to \bs{\mu}^\infty$ for some MFE $\bs{\mu}^\infty \in C([0,T]; \mcal{P}_1(\R^{2d}))$, i.e.
\begin{equation*}
    P^{\bs{\mu}^\infty}_t = \bs{\mu}^\infty_t, \quad t \in [0,T].
\end{equation*}
\end{theorem}

\subsection{Quantifying the consistency error}
To obtain the rate for the consistency error as in Corollary~\ref{cor:consistency}, we impose the additional assumptions:
\begin{assump}
\label{assump:bounded_derivative}
The drift function $f$ is continuous in time and has bounded space derivatives of all orders with the bounds independent of the choice of $\bs{\mu} \in C([0,T]; \mcal{P}_1(\R^{2d}))$, i.e.\ for each $k \in \N_0$, there exists some constant $\Lambda_k > 0$ independent of $\bs{\mu}$, such that
\begin{equation*}
\sup_{t\in[0,T], z \in \R^{2d}} |\nabla^k_z f(t,z,\mu_t)| \leq \Lambda_k.
\end{equation*}
\end{assump}
\begin{example}
The variant of the Cucker--Smale model introduced in Example~\ref{example:Cucker--Smale} satisfies this assumption.
Indeed, let us consider differentiating $\frac{P_n}{(1 + x^2)^m}$, where $P_n$ is a polynomial of order $n$.
Then,
\begin{align*}
\frac{P'_n (1+x^2)^m - 2mP_n (1 + x^2)^{m-1} x  }{(1+x^2)^{2m}}
=
\frac{P'_n (1+x^2) - 2mP_n x  }{(1+x^2)^{m+1}},
\end{align*}
thus, the order of the numerator only increases by $1$, while that of the denominator increases by $2$.
Therefore, both $\Phi$ and $\Gamma$ have bounded derivatives of all orders.
\end{example}
\begin{assump}
\label{assump:Talagrand}
We assume
\begin{enumerate} [label=\textup{(\roman*)}, ref=\textup{(\roman*)}]
    \item \label{assump:KL_terminal} $G(\nu) = H(\nu | \mu_{\mathrm{fin}})$.

    \item \label{assump:log_sobolev}$\mu_{\mathrm{fin}}$ satisfies the logarithmic Sobolev inequality, that is, there exists a constant $\alpha>0$ such that, for every smooth function
    $f:\mathbb{R}^{2d}\to[0,\infty)$ with $\int_{\mathbb{R}^{2d}} f\,d\mu_{\mathrm{fin}}=1$,
    \begin{equation*}\label{eq:LSI}
      2\alpha\int_{\mathbb{R}^{2d}} f\log f\,d\mu_{\mathrm{fin}}
      \;\le\;
      \int_{\{f>0\}}\frac{|\nabla f|^{2}}{f}\,d\mu_{\mathrm{fin}} .
    \end{equation*}

    \item \label{assump:nabla_p_fin_linear_growth} $\log p_{\mathrm{fin}}$ is in $C^1$ and the gradient of it is of linear growth, that is, there exists $C_{p_{\mathrm{fin}}}$ such that $| \nabla \log p_{\mathrm{fin}} (y) | \leq C_{p_{\mathrm{fin}}}(1 + |y|)$.
\end{enumerate}
\end{assump}

\begin{remark}
Note that under \ref{assump:KL_terminal}, Assumption~\ref{assump:terminal_cost} is satisfied. This choice is the same as \cite{garg2024soft}. The assumption \ref{assump:log_sobolev} is used to invoke the so-called Talagrand inequality (or Otto--Villani inequality) (cf. \cite[Theorem~5.2]{Gigli2013logsobolev}), i.e.\ $\mcal{W}_2(\nu,  \mu_{\mathrm{fin}}) \leq  C \sqrt{H(\nu | \mu_{\mathrm{fin}})}$ for all $\nu \in \mcal{P}(\R^{2d})$ for some $C>0$ only depending on $\alpha$. The assumption~\ref{assump:nabla_p_fin_linear_growth} is solely to improve the rate in Theorem~\ref{thm:consistency_error} by the argument in Lemma~\ref{lemma:bootstrap_h_n}. Without (iii), our estimate degrades to $\mcal{O}(n^{-1/4})$.
\end{remark}
Under these assumptions, our third main result gives the upper bound for this error:
\begin{theorem} \label{thm:consistency_error}
Let Assumptions \ref{assump:initial_and_target}, \ref{assump:coefficient}, \ref{assump:bounded_derivative} and \ref{assump:Talagrand} be in force. Then there exists a constant
\[
C =C( T,d, \theta, \int e^{\theta|z|^2}d\mu_{\mathrm{init}},  M_2(\mu_{\mathrm{fin}}), H(\mu_{\mathrm{fin}} | \leb), \alpha, C_{p_{\mathrm{fin}}}, \Lambda_0, \Lambda_1, \Lambda_2 ) > 0,
\]
such that for every $\bs{\mu} \in C([0,T];\mcal{P}_1(\R^{2d}))$ and $n \in \N$, the following bounds hold:
\begin{equation*}
\mcal{W}_1(P^{n,\bs{\mu}}, P^{\bs{\mu}})
\leq
C n^{- 1/2},
\end{equation*}
where $P^{n,\bs{\mu}}$ is the solution to $n$-\textup{SB($\bs{\mu}$)}, and $P^{\bs{\mu}}$ is that of \textup{SB($\bs{\mu}$)}.
\end{theorem}
As an immediate consequence of this stability, we obtain the rate of the consistency error.
\begin{corollary}\label{cor:consistency}
Under the same assumptions of Theorem~\ref{thm:consistency_error}, let $C>0$ be as in the statement of Theorem~\ref{thm:consistency_error}. Let $\bs{\mu}^n$ be an $n$-MFE. Then, we have
\begin{equation*}
\sup_{t \in [0,T]}\mcal{W}_1( \mu^n_t, P^{\bs{\mu}^n}_t)
\leq
C n^{- 1/2}.
\end{equation*}
\end{corollary}

\begin{remark} \label{remark:lower_bound}
Although we do not know about the optimality of the rate, a simple case where $f\equiv0$, and $\mu_{\mathrm{init}}, \mu_{\mathrm{fin}}$ are Gaussian  showcases that the consistency error is bounded from below by $c\,n^{-1}$ for some $c>0$. See Appendix~\ref{appendix:lower_bound}.
\end{remark}

Obtaining this rate involves the smoothness of $p^{\bs{\mu}}_T$, which is of independent interest.
\begin{theorem} \label{thm:linear_growth}
Let Assumptions~\ref{assump:coefficient} and \ref{assump:bounded_derivative} hold.
Then, $(z,w) \mapsto p_T^{\bs{\mu}}(z,w)$ is infinitely differentiable.
Furthermore, there exists a constant $C = C(T, d, \Lambda_0, \Lambda_1, \Lambda_2)>0$ such that
\begin{equation*}
    \left|\nabla_{z,w} \log p_T^{\bs{\mu}} (z,w) \right| \leq C(1 + |z|+ |w|).
\end{equation*}
\end{theorem}

\section{Overview of the proofs}
\label{sec:strategy}
In this section, we provide a brief overview of the proof of each theorem provided in Section~\ref{sec:proofs}.

\subsection{Theorem~\ref{thm:existence_mfe}}
\label{sec:sketch_existence}
By Proposition~\ref{prop:well-posedness_SB_mu}, the map $\Psi: \bs{\mu} \mapsto \Psi(\bs{\mu}) := (P^{\bs{\mu}}_t)_{t \in [0,T]}$ is well-defined. Note that $\Psi(\bs{\mu})$ is not the optimal path measure $P^{\bs{\mu}}$ itself but the flow of its marginals, and that a fixed point of $\Psi$ is an MFE by definition.

We find a fixed point by Schauder's theorem: if $\mcal{E}$ is a non-empty closed convex subset of a normed vector space $\mcal{V}$ and $\Phi : \mcal{E} \to \mcal{E}$ is continuous with $\Phi(\mcal{E})$ relatively compact in $\mcal{V}$, then $\Phi$ has a fixed point in $\mcal{E}$ (cf.\ \cite[Theorem~4.32]{carmona2018probabilistic_I}). We apply it in the normed vector space $C([0,T]; \mcal{M}_1(\R^{2d}))$, where $\mcal{M}_1(\R^{2d})$ denotes the space of finite signed measures $\mu$ with $\int |z| \, d|\mu| < \infty$, equipped with the following variant of the Kantorovich--Rubinstein norm, $\| \mu \|_{\mathrm{KR}} := |\mu(\R^{2d})| + \sup \{ \int f \, d\mu \,;\, \|f\|_{\mathrm{Lip}} \leq 1, \, f(0) = 0 \}$, which induces the $\mcal{W}_1$-metric on $\mcal{P}_1(\R^{2d})$. The proof consists of the following three steps.

\subsubsection*{Step 1: A closed convex set $E \subset C([0,T]; \mcal{P}_1(\R^{2d}))$ with $\Psi(E) \subset E$.}
We take $E$ of the form \eqref{def:set_E}, that is, the set of flows whose second moments are bounded by a constant $K$ uniformly in $t \in [0,T]$, and we look for $K$ such that $\Psi(\bs{\mu}) \in E$ for all $\bs{\mu} \in E$. The key is the uniform second moment bound of Proposition~\ref{prop:bdd_2nd_moment}, which we obtain from the Donsker--Varadhan inequality: for $\alpha > 0$, $\E^{P^{\bs{\mu}}} [\alpha \sup_{t \in [0,T]} |Z_t|^2]$ is bounded by $H(P^{\bs{\mu}} | R^{\bs{\mu}})$ plus the logarithm of $\E^{R^{\bs{\mu}}} [ e^{\alpha \sup_{t \in [0,T]} |Z_t|^2} ]$. The first term is bounded by the constant $M$ of Corollary~\ref{cor:bdd_relative_entropy}, which comes from the Gaussian lower bound of the transition kernel (Lemma~\ref{lemma:kernel_bound}). For the second term, since the growth constant $C_f$ is independent of $\bs{\mu}$ and $\mu_{\mathrm{init}}$ has a finite exponential moment (Assumption~\ref{assump:initial_and_target}~\ref{assump:second_moment_and_exp_moment}), we can find $\alpha > 0$ such that it is bounded uniformly in $\bs{\mu}$ (Lemma~\ref{lemma:bdd_exp_moment}). Together with the continuity of $t \mapsto P^{\bs{\mu}}_t$ in $\mcal{W}_1$ (Lemma~\ref{lemma:P_continuous}), this gives $\Psi(E) \subset E$ for the $K$ of Proposition~\ref{prop:bdd_2nd_moment}.

\subsubsection*{Step 2: $\Psi : E \to E$ is continuous (Proposition~\ref{prop:Psi_continuous}).}
Let $\bs{\mu}^n \to \bs{\mu}^\infty$ in $C([0,T]; \mcal{P}_1(\R^{2d}))$; we prove $P^{\bs{\mu}^n} \to P^{\bs{\mu}^\infty}$ in $\mcal{P}_1(\Omega)$. By the relative compactness of $\{ P^{\bs{\mu}^n} \}_{n \in \N}$ in $\mcal{P}_1(\Omega)$ (Lemma~\ref{lemma:relativelycpt_W1}), the sequence has an accumulation point $P^\star$. Along a subsequence converging to $P^\star$, the joint lower semicontinuity of $H$ and the optimality of $P^{\bs{\mu}^n}$ for SB($\bs{\mu}^n$) yield \eqref{eq:P_star_is_optimal}, that is, $H(P^\star | R^{\bs{\mu}^\infty}) \leq \liminf_n H(P^{\bs{\mu}^\infty} | R^{\bs{\mu}^n})$. We then show $\lim_n H(P^{\bs{\mu}^\infty} | R^{\bs{\mu}^n}) = H(P^{\bs{\mu}^\infty} | R^{\bs{\mu}^\infty})$ by computing the Girsanov density of $R^{\bs{\mu}^\infty}$ with respect to $R^{\bs{\mu}^n}$, whose drift difference is bounded by $C_f \sup_{t} \mcal{W}_1(\mu^n_t, \mu^\infty_t)$ (Assumption~\ref{assump:coefficient}~\ref{assump:lipschitz}). Since $P^\star$ satisfies the boundary constraints, the uniqueness of the solution to SB($\bs{\mu}^\infty$) gives $P^\star = P^{\bs{\mu}^\infty}$, and the double subsequence argument concludes $P^{\bs{\mu}^n} \to P^{\bs{\mu}^\infty}$ in $\mcal{P}_1(\Omega)$.

\subsubsection*{Step 3: $\Psi(E)$ is relatively compact (Proposition~\ref{prop:Psi_precompact_range}).}
We apply the Arzel\`a--Ascoli theorem, for which we check: (i) for each $t \in [0,T]$, $\{ P^{\bs{\mu}}_t \}_{\bs{\mu} \in E}$ is relatively compact in $\mcal{P}_1(\R^{2d})$; (ii) $\{ \Psi(\bs{\mu}) \}_{\bs{\mu} \in E}$ is equicontinuous in $C([0,T]; \mcal{P}_1(\R^{2d}))$. Item (i) is immediate from the uniform second moment bound of Proposition~\ref{prop:bdd_2nd_moment}. Item (ii) is Corollary~\ref{cor:Psi_equicontinuous}, which follows from the estimate \eqref{eq:equi_continuous}: the $\mcal{W}_1$-modulus of continuity of $t \mapsto P^{\bs{\mu}}_t$ is controlled by the second moment of $P^{\bs{\mu}}$ and the kinetic energy of the optimal control $\alpha^{\bs{\mu}}$, the latter being equal to $H(P^{\bs{\mu}} | R^{\bs{\mu}})$ (Lemma~\ref{lemma:entropy_is_kinetic_energy}), and both are bounded uniformly in $\bs{\mu}$.

\begin{remark}[Non-degenerate setting]
The degeneracy of the reference process enters the proof of Theorem~\ref{thm:existence_mfe} only through the uniform entropy bound \eqref{eq:bound_entropy_1}, which relies on the Gaussian lower bound for the transition density of \cite[Theorem~1.1]{delarue_density_estimates}, uniform in $\bs{\mu}$.
Consequently, the same argument applies to any model (degenerate or not) whose transition density satisfies a lower bound of this type and whose drift satisfies the conditions of Assumption~\ref{assump:coefficient}.
\end{remark}

\subsection{Theorem~\ref{thm:qualitative}}
\label{sec:sketch_qualitative}
Fix $\lambda > 0$. By Proposition~\ref{prop:wellposedness_soft_SB}, the map $\Psi^\lambda : \bs{\mu} \mapsto (P^{\lambda, \bs{\mu}}_t)_{t \in [0,T]}$ is well-defined, and a $\lambda$-MFE is a fixed point of $\Psi^\lambda$. The key observation is the following relation between \ref{eq:SB_mu} and \ref{eq:soft_SB_mu}:
\begin{equation} \label{eq:bound_soft_by_hard}
  H(P^{\lambda, \bs{\mu}} | R^{\bs{\mu}}) \le H(P^{\lambda, \bs{\mu}} | R^{\bs{\mu}}) + \lambda \, G(P^{\lambda, \bs{\mu}}_T)
  \le H(P^{\bs{\mu}} | R^{\bs{\mu}}) + \lambda \, G(P^{\bs{\mu}}_T) = H(P^{\bs{\mu}} | R^{\bs{\mu}}) \leq M,
\end{equation}
which follows from $G \geq 0$, the optimality of $P^{\lambda, \bs{\mu}}$, $G(\mu_{\mathrm{fin}}) = 0$ and Corollary~\ref{cor:bdd_relative_entropy}. Since $M$ is independent of $\lambda$ and $\bs{\mu}$, the uniform bounds of Section~\ref{sec:sketch_existence} hold for $P^{\lambda, \bs{\mu}}$ as well, and the existence of a $\lambda$-MFE follows from the same program.

For the convergence, let $\bs{\mu}^n$ be an $n$-MFE. By \eqref{eq:bound_soft_by_hard} with $\lambda = n$, $\{ P^{n, \bs{\mu}^n} \}_{n \in \N}$ is relatively compact in $\mcal{P}_1(\Omega)$ as in Lemma~\ref{lemma:relativelycpt_W1}; let $P^\star$ be a limit point and $\bs{\mu}^\infty := (P^\star_t)_{t \in [0,T]}$. It suffices to show $P^\star = P^{\bs{\mu}^\infty}$, and \eqref{eq:bound_soft_by_hard} is again the essential tool: it gives $G(P^{n, \bs{\mu}^n}_T) \leq M/n \to 0$, so that $P^\star_T = \mu_{\mathrm{fin}}$ by the lower semicontinuity of $G$ and Assumption~\ref{assump:terminal_cost}~\ref{assump:g_zero}. It also gives $H(P^{n, \bs{\mu}^n} | R^{\bs{\mu}^n}) \leq H(P^{\bs{\mu}^\infty} | R^{\bs{\mu}^n})$, which yields $H(P^\star | R^{\bs{\mu}^\infty}) \leq H(P^{\bs{\mu}^\infty} | R^{\bs{\mu}^\infty})$ in the limit as in Step~2 of Section~\ref{sec:sketch_existence}. Hence $P^\star = P^{\bs{\mu}^\infty}$ by the uniqueness of the solution to SB($\bs{\mu}^\infty$).

\subsection{Theorem~\ref{thm:consistency_error}} \label{sec:sketch_consistency_error}
Fix $\bs{\mu} \in C([0,T]; \mcal{P}_1(\R^{2d}))$ and $n \in \N$, and let $P^n_{\mathrm{soft}}$ and $P_{\mathrm{hard}}$ be the solutions to $n$-SB($\bs{\mu}$) and SB($\bs{\mu}$), respectively. We write $\pi^n_{\mathrm{soft}}, \pi_{\mathrm{hard}}$ for their joint laws at times $0$ and $T$, $\nu^n_{\mathrm{soft}}, \nu_{\mathrm{hard}} = \mu_{\mathrm{fin}}$ for their terminal marginals, and $h_n := H(\nu^n_{\mathrm{soft}} | \nu_{\mathrm{hard}}) = G(\nu^n_{\mathrm{soft}})$.

\subsubsection*{Step 1: Reduction to the stability of entropic optimal transport.}
The starting point is the observation that $P^n_{\mathrm{soft}}$ solves the \textit{hard}-constrained problem with the same reference measure $R^{\bs{\mu}}$ and the terminal constraint $\nu_{\mathrm{hard}}$ replaced by $\nu^n_{\mathrm{soft}}$. Hence, bounding $\mcal{W}_1(P^n_{\mathrm{soft}}, P_{\mathrm{hard}})$ reduces to the stability of SB($\bs{\mu}$) with respect to its terminal constraint; moreover, since $P^n_{\mathrm{soft}}$ and $P_{\mathrm{hard}}$ share the same bridge, it suffices to consider the static problem \ref{eq:SB_mu_nu}. By Lemma~\ref{lemma:SB_EOT}, this is equivalent to the EOT problem \ref{eq:EOT} with the cost $c^{\bs{\mu}} = -\log p^{\bs{\mu}}_T$, and the values of these problems are related by $V^{\bs{\mu}}(\rho) = S^{\bs{\mu}}(\rho) + H(\rho | \leb)$.
The direct application of~\cite[Theorem~3.7]{eckstein2022quant_stability_EOT} yields the Lipschitz stability \eqref{eq:EOT_value_stability} of the value $S^{\bs{\mu}}$ with respect to $\mcal{W}_2$ (Proposition~\ref{prop:value_stability}), with a constant $L$ independent of $n$ and $\bs{\mu}$.

\subsubsection*{Step 2: From the value stability to the plan stability (Proposition~\ref{prop:JS_bound}).}
The value stability controls the Jensen gap of the convex functional $S^{\bs{\mu}}$ at the midpoint $\nu_{1/2} := \frac{1}{2}(\nu^n_{\mathrm{soft}} + \nu_{\mathrm{hard}})$, namely $\frac{1}{2} S^{\bs{\mu}}(\nu^n_{\mathrm{soft}}) + \frac{1}{2} S^{\bs{\mu}}(\nu_{\mathrm{hard}}) - S^{\bs{\mu}}(\nu_{1/2}) \leq C L \, \mcal{W}_2(\nu^n_{\mathrm{soft}}, \nu_{\mathrm{hard}})$ (Proposition~\ref{prop:value_stability} and Lemma~\ref{lemma:midpoint_W2}). On the other hand, the parallelogram identity \eqref{eq:compensation}, applied to the plans and to the marginals, shows that this Jensen gap dominates $\js(\pi^n_{\mathrm{soft}}, \pi_{\mathrm{hard}}) - \js(\nu^n_{\mathrm{soft}}, \nu_{\mathrm{hard}})$, where $\js$ denotes the Jensen--Shannon divergence; see \eqref{eq:key_identity} and Lemma~\ref{lemma:JS_plans}. Since $\js(\nu^n_{\mathrm{soft}}, \nu_{\mathrm{hard}}) \leq h_n / 2$ by \eqref{eq:JS_marginals} and $\mcal{W}_2(\nu^n_{\mathrm{soft}}, \nu_{\mathrm{hard}}) \leq C h_n^{1/2}$ by Talagrand's inequality \eqref{eq:talagrand}, we obtain $\js(\pi^n_{\mathrm{soft}}, \pi_{\mathrm{hard}}) \leq C h_n^{1/2}$. Finally, comparing the value of $n$-SB($\bs{\mu}$) at $P^n_{\mathrm{soft}}$ with that at the admissible $P_{\mathrm{hard}}$, and using the value stability once more, we get $(1 + n) h_n \leq C h_n^{1/2}$, i.e., $h_n \leq C n^{-2}$ (Lemma~\ref{lemma:bootstrap_h_n}). This yields the plan stability \eqref{eq:JS_final}, $\js(\pi^n_{\mathrm{soft}}, \pi_{\mathrm{hard}}) \leq C n^{-1}$ (Proposition~\ref{prop:JS_bound}).

\subsubsection*{Step 3: Lifting to the path measures.}
Since $\js$ is an $f$-divergence and $P^n_{\mathrm{soft}}$, $P_{\mathrm{hard}}$ share the same bridge, $\js(P^n_{\mathrm{soft}}, P_{\mathrm{hard}}) = \js(\pi^n_{\mathrm{soft}}, \pi_{\mathrm{hard}})$. The $T_1$-type inequality for $\js$ (Lemma~\ref{lemma:T_1_inequality_for_JS}), together with the uniform second moment bound of Proposition~\ref{prop:bdd_2nd_moment}, then gives $\mcal{W}_1(P^n_{\mathrm{soft}}, P_{\mathrm{hard}}) \leq C \js(\pi^n_{\mathrm{soft}}, \pi_{\mathrm{hard}})^{1/2} \leq C n^{-1/2}$, which is Theorem~\ref{thm:consistency_error}.

\begin{remark}[Literature on quantitative stability of EOT]
Quantitative stability of EOT has been studied in various settings
\cite{eckstein2022quant_stability_EOT, bayraktar2025stability,
carlier_displacement_smoothnessEOT,
deligiannidis2024quantitative_stability_IPFP,
carlier2020_differential_multimarginal_SB,
chiarini_etal_kineticSB, chiarini2024semiconcavity}.
Results allowing non-compact marginals and unbounded costs include
\cite{eckstein2022quant_stability_EOT, bayraktar2025stability,
chiarini2024semiconcavity}.
In \cite{bayraktar2025stability}, the authors establish Wasserstein stability of optimal couplings for a class
of divergence regularizations, including relative entropy,
under polynomial moment assumptions.

Here we use the value-stability result of
\cite[Theorem~3.7]{eckstein2022quant_stability_EOT}.
Its cost condition follows from the linear growth bound on
$\nabla\log p^{\bs{\mu}}$ proved in
Section~\ref{sec:proof_lineargrowth}.
The passage to path measures is then obtained through our
Jensen--Shannon argument.

The approach of \cite{chiarini2024semiconcavity} instead uses
semiconcavity assumptions on the cost and the Schr\"odinger
potential. The additional second-order estimates needed for
that approach are not established by our method.
\end{remark}

\begin{remark}[On the choice of the Jensen--Shannon divergence] \label{remark:sketch_choice_of_JS}
Under the same assumptions, \cite[Theorem~3.11]{eckstein2022quant_stability_EOT} directly gives $\mcal{W}_2(\pi^n_{\mathrm{soft}}, \pi_{\mathrm{hard}}) \leq C n^{-1/4}$.
However, a Wasserstein bound on the plans does not lift to the path measures, which is why we want to use an $f$-divergence in Step~2.
Among the class of $f$-divergences, the natural candidate would be the relative entropy (or its symmetrized version). In fact, the same estimate as in Lemma~\ref{lemma:T_1_inequality_for_JS} holds for this choice too (this holds because the constant is allowed to depend on both measures, in contrast to the usual $T_1$ inequality).
However, to bound the relative entropy, we will need the (Lipschitz) stability of the map $\rho \mapsto \log \psi^\rho$ in addition to the value stability obtained in Step~1, where $\psi^\rho$ is the Schr\"odinger potential associated with the terminal marginal $\rho$.
In fact, the relative entropy between two optimal plans is the Bregman divergence of
$V^{\bs{\mu}}$: $H(\pi | \tilde{\pi}) = V^{\bs{\mu}}(\rho) - V^{\bs{\mu}}(\tilde{\rho}) - \int \log\tilde{\psi} \, d(\rho - \tilde{\rho})$, where $\pi, \tilde{\pi}$ are the optimal plans with terminal marginals $\rho, \tilde{\rho}$, and $\tilde{\psi}$ is the Schr\"odinger potential associated with $\tilde{\pi}$.
Such Lipschitz stability in $L^2$-distance is obtained in \cite{carlier2020_differential_multimarginal_SB}, for instance, but only for bounded kernels and marginals whose densities are bounded above and below by a positive constant.
On the other hand, the Jensen gap of $V^{\bs{\mu}}$, which dominates $\js(\pi^n_{\mathrm{soft}}, \pi_{\mathrm{hard}})$, is controlled by the value stability alone.

Our approach for obtaining the stability of the path measures might be useful in other settings where estimates on the transition kernel are accessible but the potentials are not.
Indeed, our argument depends on the transition kernel (through Theorem~\ref{thm:linear_growth}) only in Step~1. Thus, once a Lipschitz stability of the EOT value with respect to $\mcal{W}_2$ is obtained by whatever means, together with a control of $\mcal{W}_2(\nu^n_{\mathrm{soft}}, \nu_{\mathrm{hard}})$ and $H(\nu^n_{\mathrm{soft}} | \nu_{\mathrm{hard}})$ for the terminal marginals, Steps~2 and~3 apply accordingly.
\end{remark}

\subsection{Theorem~\ref{thm:linear_growth}}
\label{sec:sketch_linear_growth}
As mentioned in Section~\ref{sec:intro}, our proof follows the approach of \cite[Theorem~B]{sheu1991_someestimates}, which establishes the same type of estimate for $\nabla \log p$ in the non-degenerate setting. The smoothness of $p^{\bs{\mu}}$ stated in Theorem~\ref{thm:linear_growth} is proved in Appendix~\ref{appendix:smoothness}, and we take it for granted in this overview (see Remark~\ref{remark:sketch_smoothness}). In this subsection, we write $p$ for $p^{\bs{\mu}}$.

\subsubsection*{Step 1: Control formulation and the energy identity}
Fix $w \in \R^{2d}$ and let $J(t,z) := -\log p(t,T;z,w)$. By Fleming's logarithmic transformation, $J$ solves the HJB equation \eqref{eq:HJB}, and it is therefore interpreted as the value function of the stochastic control problem that minimizes the quadratic energy $\frac{1}{2} \E [\int_t^T |\alpha_s|^2 \, ds]$ of a drift control $\alpha$ acting on the velocity component, under the constraint that the process reaches $w$ at time $T$; heuristically, this corresponds to the terminal cost $g := -\log \delta_w$. The optimal control is given in feedback form by $\alpha_t = -\nabla_v J(t, Z_t)$, and the optimally controlled process is the bridge process \eqref{eq:bridge_process}, namely the reference process conditioned on $Z_T = w$. Along the bridge process, we have the identity \eqref{eq:bridge-rep} of Lemma~\ref{lemma:control_nabla_J_by_J}, which is nothing but the dynamic programming principle:
\begin{equation*}
\E \left[ \frac{1}{2} \int_0^t |\nabla_v J(s, Z_s)|^2 \, ds \right]
= J(0, z) - \E \left[ J(t, Z_t) \right].
\end{equation*}
Applying the two-sided Gaussian estimates of \cite[Theorem~1.1]{delarue_density_estimates} to the right hand side, we see that the minimal energy
$\frac{1}{2} \E  \int_0^t |\nabla_v J(s, Z_s)|^2 \, ds$ is bounded by $C (1 + |z|^2 + |w|^2)$.

\subsubsection*{Step 2: Bounding $\nabla_z J$ by the minimal energy (Lemma~\ref{lemma:bounds_by_E(nabla_vJ)})}
In view of Step~1, the linear growth of $\nabla_z \log p$ follows once $|\nabla_z J(s_0, z)|$ is bounded by $\E [ \int_{s_0}^{s_0 + \delta} |\nabla_v J(t, Z_t)|^2 \, dt ]^{1/2}$, over a time window $\delta$ of order $T - s_0$. This is the content of Lemma~\ref{lemma:bounds_by_E(nabla_vJ)}.
The main tool is the evolution \eqref{eq:evolutions_of_nabla_J} of the processes $\partial_{x_\alpha} J(t, Z_t)$ and $\partial_{v_\beta} J(t, Z_t)$ along the bridge process:
\begin{equation*}
\left\{
\begin{split}
d \, \partial_{x_\alpha} J(t, Z_t) &= - \big( \partial_{x_\alpha} b \cdot \nabla_v  J \big)(t, Z_t) \, dt + \nabla_v \partial_{x_\alpha}  J(t, Z_t) \cdot dW^Q_t , \\
d \, \partial_{v_\beta} J(t, Z_t) &= -
\left[
\partial_{x_\beta} J
+ \partial_{v_\beta} b \cdot \nabla_v J
\right] (t, Z_t) \, dt
+  \nabla_v \partial_{v_\beta}  J(t, Z_t) \cdot dW^Q_t .
\end{split}
\right.
\end{equation*}
which is derived in Appendix~\ref{appendix:derivation_of_evolution_nabla_J} by differentiating the HJB equation \eqref{eq:HJB} in the space variables and applying It\^o's formula.

We start with the estimate of $\partial_{x_\beta} J$. From the SDE of $\partial_{x_\beta} J(t, Z_t)$ and taking the expectation, we obtain the identity \eqref{eq:partial_x_J_identity}:
\begin{equation*}
 \partial_{x_\beta} J(s_0, z)
 = \E\left[
 \partial_{x_\beta} J(u, Z_{u}) + \int^{u}_{s_0} \big( \partial_{x_\beta} b \cdot \nabla_v J \big)(r, Z_r) \, dr
\right] ,
\end{equation*}
which represents $\partial_{x_\beta} J(s_0, z)$ in terms of $\partial_{x_\beta} J(u, Z_u)$ and $\int_{s_0}^u (\partial_{x_\beta} b \cdot \nabla_v J)(r, Z_r) \, dr$. Here, the stochastic integral is localized by the exit times $\tau_N$ from the balls of radius $N$, and the passage to the limit $N \to \infty$ is justified by Lemma~\ref{lemma:partial_J_bdd}, which relies on the uniform bound on $\nabla p$ of Corollary~\ref{cor:bounded_derivative}. Integrating the above identity over $u \in [s,t]$ gives \eqref{eq:partial_x_J_before_plug_in}:
\begin{equation*}
(t-s) \, \partial_{x_\beta} J(s_0, z)
 = \E\left[
 \int^t_s \partial_{x_\beta} J(u, Z_{u}) \, du + \int^t_s \int^u_{s_0} \big( \partial_{x_\beta} b \cdot \nabla_v J \big)(r, Z_r) \, dr \, du
\right] .
\end{equation*}
The obstacle here is that the right hand side involves $\E [ \int_s^t \partial_{x_\beta} J(u, Z_u) \, du ]$, which is not controlled by the energy identity of Step~1. To replace this term by terms involving $\nabla_v J$ only, we observe that the second equation of \eqref{eq:evolutions_of_nabla_J} contains $\partial_{x_\beta} J$ in its drift.
Due to the localizing argument as above, we obtain the identity \eqref{eq:int_partial_x_J}:
\begin{equation*}
\E\left[\int^t_s \partial_{x_\beta} J(u, Z_u) \, du \right]
= \E\Bigg[ - \partial_{v_\beta} J(t, Z_{t})
+ \partial_{v_\beta} J(s, Z_{s}) - \int^t_s \big( \partial_{v_\beta}b \cdot \nabla_v J \big)(u, Z_u) \, du \Bigg].
\end{equation*}
Notice the right hand side does not include any term with $\partial_x J$.
Plugging this identity into the one above, averaging over $s$ and $t$ (this is the time-averaging argument of \cite[Proof of Lemma~5.4]{sheu1991_someestimates}), and using $|\partial_{v_\beta} J| \leq |\nabla_v J|$, $\| \nabla b \|_\infty \leq C$ and the Cauchy--Schwarz inequality, we obtain the first bound in \eqref{eq:bounds_by_E(nabla_vJ)}.

For $\partial_{v_\beta} J$ (Step~3 of the proof), we follow the same argument starting from the identity \eqref{eq:partial_v_J_identity}, which is also derived from the second equation of \eqref{eq:evolutions_of_nabla_J}. This time, the term $\E [ \int |\partial_{x_\beta} J(u, Z_u)| \, du ]$ is controlled by the estimate for $\partial_{x_\beta} J$ established beforehand, applied at the starting point $(u, Z_u)$, together with the Markov property of the bridge process. This is the reason why Lemma~\ref{lemma:bounds_by_E(nabla_vJ)} is formulated for an arbitrary starting time $s_0 \in [0,T)$ rather than for $s_0 = 0$ only.

Finally, the estimate for the endpoint variable $\nabla_w \log p$ is obtained by running the same program for $I(t, w) := -\log p(0, t; z, w)$, the bridge process being replaced by its time reversal; see Appendix~\ref{appendix:estimate_for_end_point_variable}.

\begin{remark}[Smoothness of $p$] \label{remark:sketch_smoothness}
The derivation of \eqref{eq:evolutions_of_nabla_J} in Appendix~\ref{appendix:derivation_of_evolution_nabla_J} is where the smoothness of $p$ is used: differentiating the HJB equation \eqref{eq:HJB} in the space variables and applying It\^o's formula to $\partial_{x_\alpha} J(t, Z_t)$ and $\partial_{v_\beta} J(t, Z_t)$ require $J$ to be $C^1$ in $t$ and $C^3$ in $z$. In addition, the localization argument in the proof of Lemma~\ref{lemma:bounds_by_E(nabla_vJ)} requires the bound on $\nabla p$ of Corollary~\ref{cor:bounded_derivative}, uniform in $\bs{\mu}$.
\end{remark}

\begin{remark}[Difference from the non-degenerate setting] \label{remark:sketch_difference_from_Sheu}
In the non-degenerate setting of \cite{sheu1991_someestimates}, the control acts on all the coordinates, and hence the energy identity controls the full gradient, $\E [ \int |\nabla J(t, Z_t)|^2 \, dt ]$; the drift terms appearing in the evolution of $\nabla J$ are then directly bounded by the minimal energy. In the degenerate setting, the control acts on the velocity only, so that the energy identity \eqref{eq:bridge-rep} controls $\nabla_v J$ but gives no information on $\nabla_x J$. However, the evolution of $\partial_{v_\beta} J$ in \eqref{eq:evolutions_of_nabla_J} contains the drift term $\partial_{x_\beta} J$, which is the key to get rid of $\nabla_x J$ in the estimates.
\end{remark}

\begin{remark}[On the exponents of $\delta$] \label{remark:sketch_exponent}
In \cite[Theorem~B]{sheu1991_someestimates}, the gradient of $-\log p$ is bounded by $T^{-1/2}$ times the square root of the minimal energy. In Lemma~\ref{lemma:bounds_by_E(nabla_vJ)}, the exponent for $\nabla_v J$ is the same, $\delta^{-1/2}$, whereas the one for $\nabla_x J$ is $\delta^{-3/2}$. This gap in scaling is typical for the kinetic dynamics; the same pair of exponents appears for the kinetic kernel \cite{chaudru2023heat,delarue_density_estimates}. The Gaussian case shows that these exponents are the natural ones. Let $b \equiv 0$, so that $p(t, T; z, w)$ is the Gaussian density with mean $\Theta_\delta(z) = (x + \delta v, v)$ and covariance $\Sigma_\delta$ given in \eqref{eq:Sigma}, with $\delta := T - t$. Writing $a := w_x - x - \delta v$ and $c := w_v - v$, we have by \eqref{eq:Sigma_inverse}
\begin{equation*}
J(t, z) = \frac{6 |a|^2}{\delta^3} - \frac{6 \, a \cdot c}{\delta^2} + \frac{2 |c|^2}{\delta} + \text{const},
\qquad
\nabla_x J = - \frac{12 \, a}{\delta^3} + \frac{6 \, c}{\delta^2},
\qquad
\nabla_v J = - \frac{6 \, a}{\delta^2} + \frac{2 \, c}{\delta} .
\end{equation*}
Since the quadratic form in $J$ is positive definite, $J(t,z)$ and hence the minimal energy in \eqref{eq:bridge-rep} are of the order of $\mcal{Q} := \delta^{-3} |a|^2 + \delta^{-1} |c|^2$, up to an additive constant. Therefore,
\begin{equation*}
|\nabla_x J| \leq C \delta^{-\frac{3}{2}} \left( \delta^{-\frac{3}{2}} |a| + \delta^{-\frac{1}{2}} |c| \right) \leq C \delta^{-\frac{3}{2}} \mcal{Q}^{\frac{1}{2}},
\qquad
|\nabla_v J| \leq C \delta^{-\frac{1}{2}} \left( \delta^{-\frac{3}{2}} |a| + \delta^{-\frac{1}{2}} |c| \right) \leq C \delta^{-\frac{1}{2}} \mcal{Q}^{\frac{1}{2}},
\end{equation*}
which are exactly the exponents in \eqref{eq:bounds_by_E(nabla_vJ)}.
\end{remark}

\section{Proofs} \label{sec:proofs}
\subsection{Proof of Theorem \ref{thm:existence_mfe}} \label{sec:proof_existence}
Throughout this section, we work under Assumptions~\ref{assump:initial_and_target} and \ref{assump:coefficient}.
\subsubsection*{Step 1: Construction  of a closed convex subset $E \subset C\left([0,T] ; \mcal{P}_1(\R^{2d}) \right)$ with $\Psi(E) \subset E$.}
Specifically, we seek $E$ in the form:
\begin{equation}
\label{def:set_E}
E := \left\{
    \bs{\mu} \in C([0,T] ; \mcal{P}_1(\R^{2d}))
    \, \big| \,
    \sup_{t\in[0,T]} M_2(\mu_t) \leq K
\right\},
\end{equation}
for a constant $K > 0$, and $E$ is closed and convex.
To prove $\Psi(E) \subset E$, we establish the uniform moment bound for $\left\{ P^{\bs{\mu}} \right\}_{\bs{\mu}\in E}$, as well as the continuity of $(P_t^{\bs{\mu}})_{t \in [0,T]}$ with respect to $\mcal{W}_1$.

We first demonstrate the exponential moment bound in the following lemma:
\begin{lemma} \label{lemma:bdd_exp_moment}
Let $\theta > 0$ be as in Assumption~\ref{assump:initial_and_target}~\ref{assump:second_moment_and_exp_moment}. There exist constants $\vartheta = \vartheta(T, C_f, \theta) > 0$ and $C = C\bigl(T, d, C_f, \theta, \int e^{\theta|z|^2} \mu_{\mathrm{init}}(dz)\bigr) > 0$, both independent of $\bs{\mu}$, such that
\[
    \E^{R^{\bs{\mu}}} \left[
        e^{\vartheta
           \sup_{t\in[0,T]} |Z_t|^2
        }
    \right] \leq C,
\]
for all $\bs{\mu}\in C([0,T];\mcal{P}_1(\R^{2d}))$, where the canonical process $(Z_t)_{t\in[0,T]} = (\omega_t)_{t\in[0,T]}$ is the solution to the equation \eqref{eq:reference_SDE} under $R^{\bs{\mu}}$.
\end{lemma}
Granting this lemma for the moment, we obtain the uniform moment bound for the solutions to \ref{eq:SB_mu}:
\begin{proposition}
\label{prop:bdd_2nd_moment}
There exists a constant
\[
K = K\Bigl(T, d, C_f, \theta, \int e^{\theta|z|^2} \mu_{\mathrm{init}}(dz), M_2(\mu_{\mathrm{fin}}), H(\mu_{\mathrm{fin}} \,|\, \leb)\Bigr) > 0,
\]
independent of $\bs{\mu}$, such that for any $\bs{\mu} \in C([0,T]; \mcal{P}_1(\R^{2d}))$, it holds that
\begin{equation}
\E^{P^{\bs{\mu}}}
\left[
    \sup_{t\in[0,T]} |Z_t|^2
\right]
\leq K.
\end{equation}
\end{proposition}

\begin{proof}
By the Donsker--Varadhan inequality (cf.\ \cite[Proposition~2.2]{leonard2014someproperty}), we have
\begin{equation*}
\E^{P^{\bs{\mu}}} \left[u\right] \leq  H(P^{\bs{\mu}} | R^{\bs{\mu}}) + \log \E^{R^{\bs{\mu}}} \left[ e^u\right],
\end{equation*}
where $u$ is any bounded continuous function on $\Omega$.
In particular, take $u_m(\omega) := \min \left\{ \vartheta \sup_{t\in[0,T]} |Z_t(\omega)|^2,  m\right\}$, $m \in \N$, with $\vartheta > 0$ as in Lemma~\ref{lemma:bdd_exp_moment}.
Then,
\begin{equation*}
\E^{P^{\bs{\mu}}} \left[ u_m\right] \leq  H(P^{\bs{\mu}} | R^{\bs{\mu}}) + \log \E^{R^{\bs{\mu}}} \left[ e^{u_m}\right].
\end{equation*}
By the monotone convergence theorem, letting $m \nearrow \infty$, we obtain
\begin{equation*}
\E^{P^{\bs{\mu}}} \left[
    \vartheta \sup_{t\in[0,T]} |Z_t|^2
\right]
\leq  H(P^{\bs{\mu}} | R^{\bs{\mu}})
+ \log \E^{R^{\bs{\mu}}} \left[
    e^{
       \vartheta \sup_{t\in[0,T]} |Z_t|^2
    }
\right] \leq M + \log C =: \vartheta K,
\end{equation*}
where $M$ is the constant of Corollary~\ref{cor:bdd_relative_entropy} and $C$ is that of Lemma~\ref{lemma:bdd_exp_moment}. Since $M$, $C$ and $\vartheta$ depend only on the quantities listed in the statement (note that $M_2(\mu_{\mathrm{init}})^2 \le \theta^{-1} \log \int e^{\theta|z|^2} \mu_{\mathrm{init}}(dz)$ by Jensen's inequality), so does $K$.
\end{proof}
It remains to prove the lemma.
\begin{proof}[Proof of Lemma~\ref{lemma:bdd_exp_moment}]
Consider the evolution of $Z = (X, V)$ under the reference path measure $R^{\bs{\mu}}$.
By Assumption~\ref{assump:coefficient}~\ref{assump:linear_growth}, the drift of $Z$ satisfies $|(V_s, f(s, Z_s, \mu_s))| \leq |Z_s| + C_f(1 + |Z_s|) \leq C_0 (1 + |Z_s|)$ with $C_0 := 1 + C_f$, so that
\begin{align*}
 \sup_{0 \leq s\leq t}  \left| Z_s \right| \leq  \left| Z_0 \right| +  \int^t_0 C_0 \left( 1 +  \sup_{0 \leq u\leq s} \left| Z_u \right| \right) \, ds + \sup_{0 \leq s\leq t} |W_s|,
\end{align*}
where $C_0$ is independent of $\bs{\mu}$.
By Gronwall's inequality, we obtain
\begin{align*}
     \sup_{0 \leq s\leq T}  \left| Z_s \right|
     \leq \left( \left| Z_0 \right| + \sup_{0 \leq s\leq T} |W_s|  + C_0 T \right ) e^{C_0 T}.
\end{align*}
Thus, by $(a + b + c )^2 \leq 3(|a|^2 + |b|^2 + |c|^2)$,
\begin{align*}
     \sup_{0 \leq s\leq T}  \left| Z_s \right|^2
     \leq  C_1 \left( \left| Z_0 \right|^2 + \sup_{0 \leq s\leq T} |W_s|^2  + C_2 \right ),
     \qquad C_1 := 3 e^{2 C_0 T}, \quad C_2 := C_0^2 T^2,
\end{align*}
and for $\vartheta > 0$ determined later on,
\begin{equation*} \label{eq:exp(alpha_Z^2)}
     \exp \left( \vartheta \sup_{0 \leq s\leq T}  \left| Z_s \right|^2 \right)
     \leq  \exp\left( \vartheta C_1  \left( \left| Z_0 \right|^2 + \sup_{0 \leq s\leq T} |W_s|^2  + C_2 \right )\right).
\end{equation*}
Set $\eta := \frac{1}{4T}$. Since $x \mapsto e^{\frac{\eta}{2}|x|^2}$ is convex and $e^{\frac{\eta}{2}|W_t|^2}$ is integrable for $t \leq T$, the process $e^{\frac{\eta}{2} |W_t|^2}$ is a positive submartingale, and Doob's maximal inequality gives
\begin{equation*}
    \E^{R^{\bs{\mu}}} \Bigg[ e^{\eta \sup_{t\in[0,T]} |W_t|^2}\Bigg]
    = \E^{R^{\bs{\mu}}} \Bigg[ \sup_{t\in[0,T]}\left( e^{\frac{\eta}{2} |W_t|^2} \right)^2 \Bigg]
    \leq 4 \E^{R^{\bs{\mu}}} \Bigg[ e^{\eta |W_T|^2 } \Bigg]
    = 4 (1 - 2 \eta T)^{-d/2} = 4 \cdot 2^{d/2}.
\end{equation*}
Also, with $\theta > 0$ as in Assumption~\ref{assump:initial_and_target} \ref{assump:second_moment_and_exp_moment},
\begin{equation*}
    \E^{R^{\bs{\mu}}} \left[ e^{\theta |Z_0|^2} \right] = \int_{\R^{2d}} e^{\theta |z|^2} \mu_{\mathrm{init}}(dz) < \infty.
\end{equation*}
Now choose
\begin{equation*}
    \vartheta := \frac{\min \{ \eta, \theta \}}{C_1},
\end{equation*}
which depends only on $T$, $C_f$ and $\theta$. Since $\vartheta C_1 / \theta \leq 1$, Jensen's inequality gives
$\E^{R^{\bs{\mu}}} [ e^{\vartheta C_1 |Z_0|^2} ] \leq \E^{R^{\bs{\mu}}} [ e^{\theta |Z_0|^2} ]^{\vartheta C_1 / \theta} \leq \max \{ 1, \int e^{\theta |z|^2} \mu_{\mathrm{init}}(dz) \}$, and $\vartheta C_1 \leq \eta$ gives $\E^{R^{\bs{\mu}}} [ e^{\vartheta C_1 \sup_{s \leq T} |W_s|^2} ] \leq 4 \cdot 2^{d/2}$.
Therefore, since $Z_0$ and $\sup_{0 \leq s\leq T} |W_s|^2$ are independent under $R^{\bs{\mu}}$,
\begin{align*}
\E^{R^{\bs{\mu}}}\left[e ^{ \vartheta \sup_{0 \leq s\leq T}  \left| Z_s \right|^2 }\right]
&\leq  e^{\vartheta C_1 C_2} \, \E^{R^{\bs{\mu}}} \left[ e^{ \vartheta C_1 \left| Z_0 \right|^2 } \right]
\E^{R^{\bs{\mu}}} \left[ e^{ \vartheta C_1 \sup_{0 \leq s\leq T} |W_s|^2 } \right] \\
&\leq 4 \cdot 2^{d/2} \, e^{\vartheta C_1 C_2} \max \Bigl\{ 1, \int_{\R^{2d}} e^{\theta |z|^2} \mu_{\mathrm{init}}(dz) \Bigr\} =: C,
\end{align*}
and $C$ depends only on $T$, $d$, $C_f$, $\theta$ and $\int e^{\theta|z|^2} \mu_{\mathrm{init}}(dz)$; in particular, it is independent of $\bs{\mu}$.
\end{proof}

We prove the continuity of the flow $\left\{P^{\bs{\mu}}_t\right\}_{t\in[0,T]}$.
\begin{lemma}
\label{lemma:P_continuous}
For every $\bs{\mu} \in C([0,T]; \mcal{P}_1(\R^{2d}))$, $\Psi(\bs{\mu}) \in C([0,T]; \mcal{P}_1(\R^{2d}))$.
\end{lemma}
\begin{proof}
We compute
\begin{align}
\mcal{W}^2_1(P^{\bs{\mu}}_t, P^{\bs{\mu}}_s)
&\leq \mcal{W}^2_2(P^{\bs{\mu}}_t, P^{\bs{\mu}}_s)  \notag \\
&\leq
\E^{P^{\bs{\mu}}} \left[
    \left| Z_t - Z_s \right|^2
\right] \notag \\
&\leq
\E^{P^{\bs{\mu}}} \left[
    \left| \int^t_s V_u \, du \right|^2
    + \left|
        \int^t_s \left( f(u, Z_u , \mu_u) + \alpha^{\bs{\mu}}_u \right) \, du + W^{P^{\bs{\mu}}}_t -   W^{P^{\bs{\mu}}}_s
    \right|^2
\right] \notag \\
&\leq \E^{P^{\bs{\mu}}} \left[
    (t-s) \int^T_0 \left( |V_u|^2 + 4 \left|f(u, Z_u, \mu_u) \right|^2 + 4\left|\alpha^{\bs{\mu}}_u \right|^2 \right)  \, du  + 2d (t-s)
\right] \notag \\
&\leq (t-s) C\left( 1 +  \E^{P^{\bs{\mu}}} \left[
     \sup_{t\in[0,T]}|Z_t|^2 + \int^T_0 \left| \alpha^{\bs{\mu}}_u \right|^2  \, du
\right]
\right)
\notag
 \\
&\leq (t-s) C\left(
1 +  \E^{P^{\bs{\mu}}} \left[  \sup_{t\in[0,T]}|Z_t|^2  \right] + H(P^{\bs{\mu}} | R^{\bs{\mu}})
\right) \notag  \\
&\leq (t-s) C (1 + K + M). \label{eq:equi_continuous}
\end{align}
Therefore, $\Psi(\bs{\mu}) = (P^{\bs{\mu}}_t)_{t\in[0,T]} \in C([0,T]; \mcal{P}_1(\R^{2d}))$, or more precisely, it is $\frac{1}{2}$-H\"older continuous.
\end{proof}
We conclude that for $E \subset C([0,T]; \mcal{P}_1(\R^{2d}))$ defined as in \eqref{def:set_E} with $K$ specified as in Proposition~\ref{prop:bdd_2nd_moment}, we have $\Psi(E) \subset E$.

Since the coefficient of $(t-s)$ in \eqref{eq:equi_continuous} is independent of $\bs{\mu}$, an immediate consequence is:
\begin{corollary} \label{cor:Psi_equicontinuous}
$ \left\{\Psi(\bs{\mu}) \right\}_{\bs{\mu} \in E} \subset C([0,T]; \mcal{P}_1(\R^{2d}))$ is equicontinuous.
\end{corollary}

\subsubsection*{Step 2: $\Psi :E \to E$ is continuous.}
We will use the following estimate to prove the tightness of $\left\{P^{\bs{\mu}}\right\}_{\bs{\mu}\in E}$ in the proposition below.
\begin{lemma}\label{lemma:modulus_Z}
For every $\beta \in (0,1)$ there exists a constant $C > 0$,
independent of $\bs{\mu}$, such that
\begin{equation*}
\E^{P^{\bs{\mu}}} \Bigg[
    \sup_{\substack{s,t\in[0,T] \\ |t-s| \leq \delta}}
    \left| Z_t - Z_s \right|^2
\Bigg]
\leq C \delta^{\beta},
\qquad \delta \in (0,1].
\end{equation*}
\end{lemma}
\begin{proof}
By Kolmogorov's continuity theorem, there is $C_\gamma<\infty$, depending only
on $\gamma :=\beta/2\in(0,1/2)$, $T$ and $d$, such that
$\E^{P^{\bs{\mu}}}\big[\|W^{P^{\bs{\mu}}}\|_{C^\gamma([0,T])}^2\big]\leq C_\gamma$.
Consequently,
\[
    \E^{P^{\bs{\mu}}} \Bigg[
    \sup_{\substack{s,t\in[0,T] \\ |t-s| \leq \delta}}
    \left| W^{P^{\bs{\mu}}}_t - W^{P^{\bs{\mu}}}_s \right|^2
\Bigg]
\leq C_\gamma\,\delta^{2\gamma}=C_\gamma\,\delta^{\beta} .
\]
Computing as in the proof of Lemma~\ref{lemma:P_continuous}, we obtain
\begin{equation*}
\E^{P^{\bs{\mu}}} \Bigg[
    \sup_{\substack{s,t\in[0,T] \\ |t-s| \leq \delta}}
    \left| Z_t - Z_s \right|^2
\Bigg]
\leq
C(1 + K+M)\delta + C_\gamma \delta^\beta. \qedhere
\end{equation*}
\end{proof}

\begin{proposition}
\label{prop:Psi_continuous}
    $\Psi : E \to E $ is continuous.
\end{proposition}
\begin{proof}
The map $C([0,T]; \mcal{P}_1(\R^{2d})) \ni\bs{\mu} \mapsto R^{\bs{\mu}} \in \mcal{P}_1(\Omega)$ is continuous.
Indeed, by Assumption~\ref{assump:coefficient} \ref{assump:lipschitz} and Gronwall's inequality, we have
\begin{equation} \label{eq:reference_continuity}
\mcal{W}_1(R^{\bs{\mu}}, R^{\bs{\nu}}) \leq \E \left[ \sup_{s\in[0,T]} | Z_s^{\bs{\mu}} - Z_s^{\bs{\nu}} |  \right]\leq  C T e^{CT}\sup_{t \in [0,T]}\mcal{W}_1(\mu_t, \nu_t) ,
\end{equation}
where $Z^{\bs{\mu}}$ and $Z^{\bs{\nu}}$ solve the reference SDE~\eqref{eq:reference_SDE} with inputs $\bs{\mu}$ and $\bs{\nu}$, respectively, driven by the same Brownian motion.

Let $(\bs{\mu}^n)_{n\in\N}$ and $\bs{\mu}^\infty$ be such that $\bs{\mu}^n \to \bs{\mu}^\infty$ in $C([0,T];\mcal{P}_1(\R^{2d}))$.
We write
\[
    R^n := R^{\bs{\mu}^n}, \, R^\infty := R^{\bs{\mu}^\infty}, \, P^n := P^{\bs{\mu}^n}, P^\infty := P^{\bs{\mu}^\infty}, \,
    b^n(t,z) := f(t,z, {\mu^n_t}), \, b^\infty(t,z) := f(t,z, {\mu^\infty_t}).
\]
By \eqref{eq:reference_continuity}
, we know $\mcal{W}_{1}(R^n, R^\infty) \to 0$ as $n \to \infty$.

By Lemma~\ref{lemma:relativelycpt_W1} below, $\left\{P^n\right\}_{n\in\N}$ is relatively compact in $\mcal{P}_1(\Omega)$. Therefore, there exists a subsequence of $\left\{P^n\right\}_{n\in\N}$, which we still denote by the same, that converges to some $P^\star \in \mcal{P}_1(\Omega)$ in $\mcal{W}_1$.

We show $P^\infty = P^\star$. By the joint lower semicontinuity of $H(\cdot | \cdot)$,
\begin{equation} \label{eq:P_star_is_optimal}
H(P^\star | R^\infty)
\leq \liminf_{n \to \infty} H(P^n | R^n)
\leq \liminf_{n \to \infty} H(P^\infty | R^n).
\end{equation}
Also, $P^\star$ satisfies the boundary constraints.
Thus, once we prove $\lim_{n \to \infty} H(P^\infty | R^n) = H(P^\infty | R^\infty)$, $P^\infty = P^\star$ follows from the uniqueness of the solution to \ref{eq:SB_mu}.

Consider the density process:
\begin{equation*}
\mcal{E}^n_t
:= \exp\left(
\int^t_0 \left(
b^\infty(s,Z_s) - b^n(s,Z_s)
\right)\, dW^{R^n}_s
- \frac{1}{2} \int^t_0
\| b^\infty(s,Z_s) - b^n(s,Z_s)\|^2 \, ds
\right).
\end{equation*}
By Assumption~\ref{assump:coefficient} \ref{assump:lipschitz}, $| ( b^\infty - b^n )(t,z)| \leq C_f\mcal{W}_1(\mu^n_t, \mu^\infty_t) < \infty$, which is a deterministic bound, uniform in $z$. Thus, by Novikov's condition, $\mcal{E}^n$ is a true $R^n$-martingale and we  define $Q^\infty \in \mcal{P}(\Omega)$ by $\frac{dQ^\infty} {dR^n} := \mcal{E}^n_T$. By Girsanov's theorem, the canonical process $Z =(X,V)$ satisfies
\begin{equation*}
    \left\{
        \begin{aligned}
            & d X_t = V_t dt, \\
            & d V_t = b^\infty(t,Z_t)dt +  dW^{Q^\infty}_t,\\
            & (X_0, V_0) \sim \mu_{\mathrm{init}},
        \end{aligned}
    \right.
\end{equation*}
where $W^{Q^\infty}_\cdot := W^{R^n}_\cdot - \int^\cdot_0 ( b^\infty - b^n )(t,Z_t)\,dt$ is a $Q^\infty$-Brownian motion. By the uniqueness of the solution to the martingale problem, $Q^\infty = R^\infty$ and $W^{Q^\infty}= W^{R^\infty}$.
Consequently, $P^\infty \ll  R^\infty \sim R^n $, and $\dfrac{dR^\infty}{dR^n} = \mcal{E}^n_T$, $R^n$-a.s.
Therefore,
\begin{align*}
H(P^\infty | R^n)
&= \int_{\Omega} \log \left(
\dfrac{d P^\infty}{d R^n}
\right)
 dP^\infty\\
&= \int_{\Omega} \log \left(
\dfrac{d P^\infty}{d R^\infty}
\right)
dP^\infty
+
\int_{\Omega} \log \left(
\dfrac{d R^\infty}{d R^n}
\right)
dP^\infty \\
&= H(P^\infty | R^\infty)
+
\int_{\Omega} \log \left(
\dfrac{d R^\infty}{d R^n}
\right)
dP^\infty .
\end{align*}
Moreover, there exists an $\mbb{F}$-progressively measurable process $\alpha^\infty$ with $\E^{P^\infty}[\int_0^T |\alpha^\infty_t|^2 \, dt] < \infty$ (Lemma~\ref{lemma:entropy_is_kinetic_energy}) such that $W^{R^\infty}_\cdot = W^{P^\infty}_\cdot + \int_0^\cdot \alpha^\infty_t dt$. Thus,
\begin{align*}
& \left| \int_{\Omega} \log \left(
\dfrac{d R^\infty}{d R^n}
\right)
dP^\infty \right| \\
&= \left|  \E^{P^\infty} \left[
\int^T_0 \left( b^\infty- b^n \right)(t,Z_t)
\, dW^{R^\infty}_t
+ \frac{1}{2} \int^T_0
\|  \left( b^\infty- b^n \right)(t,Z_t)\|^2 \, dt
\right] \right|\\
&= \left|  \E^{P^\infty} \left[
\int^T_0 \left( b^\infty- b^n \right)(t,Z_t) \,  \left(d W^{P^\infty}_t
+ \alpha^\infty_t dt
\right)
+ \frac{1}{2} \int^T_0
\| \left( b^\infty- b^n \right)(t,Z_t)\|^2 \, dt
\right] \right|\\
& \leq C \left ( \E^{P^\infty} \left[ \int^T_0 |\alpha^\infty_t|^2 \, dt \right]^{1/2} \sup_{t \in [0,T]} \mcal{W}_1(\mu^n_t, \mu^\infty_t)  + \sup_{t \in [0,T]} \mcal{W}^2_1(\mu^n_t, \mu^\infty_t)  \right)  \quad
\longrightarrow 0, \quad \text{as} \quad  n \to \infty,
\end{align*}
where we have used $\int_0^\cdot (b^\infty-b^n)(t,Z_t)\, dW^{P^\infty}_t $ is $P^\infty$-martingale. Hence, $P^\star = P^\infty$ as claimed.

Since this is true for any convergent subsequence of $\{P^n\}_{n\in\N}$, the original whole sequence $P^n \to P^\infty$ in $\mcal{P}_1(\Omega)$ as $n \to \infty$, and $\sup_{t \in [0,T]} \mcal{W}_1 (P^n_t, P^\infty_t) \to 0$ follows.
\end{proof}

\begin{lemma} \label{lemma:relativelycpt_W1}
$\left\{P^n\right\}_{n\in\N}$ is relatively compact in $\mathcal{P}_1(\Omega)$.
\end{lemma}
\begin{proof}
Since $Z_0\sim\mu_{\mathrm{init}}$ under $P^n$ for every $n$, the family of initial
distributions is trivially tight. Moreover, fixing $\beta\in(0,1)$, Lemma~\ref{lemma:modulus_Z} gives that for any $0 < \delta < 1$, we have $\E^{P^n} \left[ \sup_{|t-s| \leq \delta} |Z_t - Z_s|^2 \right] \leq C \delta^\beta$ for a constant $C>0$ independent of $n$. Hence, we  use \cite[Theorem~7.3]{billingsley} to deduce tightness, and by Prohorov's theorem, it is relatively weakly compact.

In addition, by Proposition~\ref{prop:bdd_2nd_moment}, $\left\{ M_2(P^n) \right\}_{n\in\N}$ is uniformly bounded and thus, $\left\{ P^n \right\}_{n\in\N}$ has uniformly integrable first moments.
Therefore, it is also relatively compact in $\mathcal{P}_1(\Omega)$ (cf.\ \cite[Theorem~5.5]{carmona2018probabilistic_I}).
\end{proof}

\subsubsection*{Step 3: $\Psi(E)$ is relatively compact.}
\begin{proposition}
\label{prop:Psi_precompact_range}
    $\Psi : E \to E $ has a relatively compact range.
\end{proposition}
\begin{proof}
By Corollary~\ref{cor:Psi_equicontinuous}, we know that $\left\{\Psi(\bs{\mu})\right\}_{\bs{\mu} \in E} \subset C([0,T];\mcal{P}_1(\R^{2d}))$ is equicontinuous with respect to $\mcal{W}_1$.
Next, by Proposition~\ref{prop:bdd_2nd_moment}, $\left\{ P^{\bs{\mu}}_t\right\}_{\bs{\mu} \in E}$ has uniformly bounded second moments for each $t\in[0,T]$. Thus, it is relatively weakly compact by Prohorov's theorem. Moreover, since it has uniformly integrable first moments, it is relatively compact in $\mcal{P}_1$.
Therefore, we  apply the Arzel\`a--Ascoli theorem to conclude the proof.
\end{proof}
We are now ready to prove the existence of an MFE:
\begin{proof}[Proof of Theorem~\ref{thm:existence_mfe}]
Since $\|\cdot \|_{\mathrm{KR}}$ induces the $\mcal{W}_1$-metric on $\mcal{P}_1$, by Steps 1-3, we  apply Schauder's theorem in the normed vector space $C([0,T]; \mcal{M}_1(\R^{2d}))$ (cf.\ \cite[Theorem~4.32]{carmona2018probabilistic_I}) to obtain a fixed point of $\Psi$, which is an MFE.
\end{proof}

\subsection{Proof of Theorem \ref{thm:qualitative}}
Throughout this section, we work under Assumptions~\ref{assump:initial_and_target},
\ref{assump:coefficient} and \ref{assump:terminal_cost}.
\subsubsection*{Proof of existence of a $\lambda$-MFE}
Fix $\lambda >0$.
We follow the argument of the hard-constraint case (Theorem~\ref{thm:existence_mfe}), replacing $P^{\bs{\mu}}$ and $\Psi$ by $P^{\lambda, \bs{\mu}}$ and $\Psi^{\lambda}$, respectively. In order to avoid repetition, we only describe the differences from Section~\ref{sec:proof_existence}.

In Step 1, we first check the uniform bound:
\begin{equation*}
    H(P^{\lambda, \bs{\mu}} | R^{\bs{\mu}}) \leq M, \,\text{for all} \,  \bs{\mu} \in C([0,T]; \mcal{P}_1(\R^{2d})).
\end{equation*}
Indeed, $P^{\bs{\mu}}$ is admissible for \ref{eq:soft_SB_mu} and satisfies
$P^{\bs{\mu}}_T=\mu_{\mathrm{fin}}$, whence $G(P^{\bs{\mu}}_T)=0$.
Therefore the value of
\ref{eq:soft_SB_mu} is bounded above by that of \ref{eq:SB_mu}, i.e.
\begin{equation*}
  H(P^{\lambda,\bs{\mu}}\,|\,R^{\bs{\mu}})+\lambda G(P^{\lambda,\bs{\mu}}_T)
  \le H(P^{\bs{\mu}}\,|\,R^{\bs{\mu}})\le M,
\end{equation*}
and the claim follows since $G\ge 0$ by Assumption~\ref{assump:terminal_cost}.

With this bound, all the arguments in Step 1 in Section~\ref{sec:proof_existence} hold.
Continuity of $\Psi^\lambda$ also holds as in Step 2.
\begin{proposition}
\label{prop:Psi_lambda_continuous}
    $\Psi^\lambda : E \to E $ is continuous.
\end{proposition}
\begin{proof}
Let $\{\bs{\mu}^n\}_{n\in\N}$ and $\bs{\mu}^\infty$ be such that $\bs{\mu}^n \to \bs{\mu}^\infty$ in $C([0,T];\mcal{P}_1(\R^{2d}))$.
We write
\[
    R^n := R^{\bs{\mu}^n}, \, R^\infty := R^{\bs{\mu}^\infty}, \, P^n := P^{\lambda, \bs{\mu}^n}, P^\infty := P^{\lambda,\bs{\mu}^\infty}.
\]
By arguments similar to those in Lemma~\ref{lemma:relativelycpt_W1}, there exists a subsequence $\left\{ P^n\right\}_{n\in \N}$, which we still denote by the same, converging in $\mcal{P}_1(\Omega)$ to some $P^\star$. We prove $P^\star = P^\infty$.
To this end, the only difference from the proof of Proposition~\ref{prop:Psi_continuous} is \eqref{eq:P_star_is_optimal}, which is replaced by the following:
\begin{equation*}
H(P^\star | R^\infty) + \lambda G(P^\star_T)
\leq \liminf_{n \to \infty} \left( H(P^n | R^n) + \lambda G(P^n_T) \right)
\leq \liminf_{n \to \infty} \left( H(P^\infty | R^n) + \lambda G(P^\infty_T) \right),
\end{equation*}
where we have used the l.s.c.\ of $G$. Then, as in
Proposition~\ref{prop:Psi_continuous}, it holds $\lim_{n\to\infty} H(P^\infty|R^n)
= H(P^\infty|R^\infty)$. By the uniqueness of the solution to $\lambda$-SB($\bs{\mu}^\infty$), $P^\star = P^\infty$.
\end{proof}
Furthermore, as in Proposition~\ref{prop:Psi_precompact_range}, $\Psi^\lambda(E)$ is relatively compact. Finally, we apply Schauder's theorem, and there exists a $\lambda$-MFE.

\subsubsection*{Proof of convergence of $\lambda$-MFE.}
By the first part of the proof, for each $n \in \N$, there exists an $n$-MFE $\bs{\mu}^n \in E$.

Set $R^n := R^{\bs{\mu}^n}$ and $P^n := P^{n,\bs{\mu}^n}$; since $\bs{\mu}^n$ is an
$n$-MFE, we have $\bs{\mu}^n = (P^n_t)_{t\in[0,T]}$.
We know that
\begin{equation*}
H(P^n | R^n)
\leq
H(P^n | R^n)+ nG(P^n_T) \leq H(P^{\bs{\mu}^n} | R^n) \leq M,
\end{equation*}
the last inequality being Corollary~\ref{cor:bdd_relative_entropy}, with $M$
independent of $n$. Hence, by arguments similar to those in
Lemma~\ref{lemma:relativelycpt_W1}, there is a subsequence of $\{P^n\}_{n\in\N}$,
still denoted by the same symbol, converging in $\mcal{P}_1(\Omega)$ to some $P^\star$.

Set $\bs{\mu}^\infty := (P_t^\star)_{t\in[0,T]}$, $R^\infty := R^{\bs{\mu}^\infty}$.
Since $\mcal{W}_1(P^n, P^\star) \to 0$, it holds $\bs{\mu}^n \to \bs{\mu}^\infty$ in $C([0,T];\mcal{P}_1(\R^{2d}))$, and therefore $\mcal{W}_1(R^n,R^\infty) \to 0$ by \eqref{eq:reference_continuity}.

Let $P^\infty := P^{\bs{\mu}^\infty}$ i.e.\ the solution to SB($\bs{\mu}^\infty$).
We show $P^\star = P^\infty$ in two steps:
(i) $P^\star_0 = \mu_{\mathrm{init}}$ and $P^\star_T = \mu_{\mathrm{fin}}$;
(ii) $H(P^\star | R^\infty) \le H(P^\infty | R^\infty)$.

\emph{Step (i).} By the above, $G(P^n_T) \le M/n$, so the lower semicontinuity of $G$
and $G \ge 0$ give
\begin{equation*}
  0 \le G(P^\star_T) \le \liminf_{n\to\infty} G(P^n_T)
  \le \lim_{n\to\infty} \frac{M}{n} = 0 .
\end{equation*}
By Assumption~\ref{assump:terminal_cost}, $G(\nu) = 0$ implies $\nu = \mu_{\mathrm{fin}}$, whence
$P^\star_T = \mu_{\mathrm{fin}}$. Moreover $P^n_0 = \mu_{\mathrm{init}}$ for every $n$ and
$P^n_0 \to P^\star_0$, so $P^\star_0 = \mu_{\mathrm{init}}$.

\emph{Step (ii).} Since $P^\infty_T = \mu_{\mathrm{fin}}$, we have $G(P^\infty_T) = 0$, and the
optimality of $P^n$ yields
\begin{equation*}
  H(P^n | R^n) \le H(P^n | R^n) + n\,G(P^n_T)
  \le H(P^\infty | R^n) + n\,G(P^\infty_T) = H(P^\infty | R^n).
\end{equation*}
By the joint lower semicontinuity of the relative entropy, and the fact that  $\lim_{n\to\infty} H(P^\infty|R^n)
= H(P^\infty|R^\infty)$ as in
Proposition~\ref{prop:Psi_continuous}, we conclude
\begin{equation*}
  H(P^\star | R^\infty) \le \liminf_{n\to\infty} H(P^n|R^n)
  \le \lim_{n\to\infty} H(P^\infty | R^n) = H(P^\infty | R^\infty).
\end{equation*}

By Step (i), $P^\star$ is admissible for $\mathrm{SB}(\bs{\mu}^\infty)$; by Step (ii)
and the uniqueness of its minimizer, $P^\star = P^\infty$. Consequently
$\bs{\mu}^\infty = (P^\infty_t)_{t\in[0,T]} = \Psi(\bs{\mu}^\infty)$, i.e.\
$\bs{\mu}^\infty$ is an MFE.
\subsection{Proof of Theorem \ref{thm:consistency_error}}
Throughout this section, we work under Assumptions \ref{assump:initial_and_target}, \ref{assump:coefficient}, \ref{assump:terminal_cost}, \ref{assump:bounded_derivative} and \ref{assump:Talagrand}.

\subsubsection*{Step 1: Equivalent formulation with entropic optimal transport}
In this section, we recall that the static formulation of SB is equivalent to the entropic optimal transport (EOT) problem with a certain cost function.

Throughout this section, we grant Theorem~\ref{thm:linear_growth}, which we prove in Section~\ref{sec:proof_lineargrowth}, i.e. for each $\bs{\mu} \in C([0,T]; \mcal{P}_1(\R^{2d}))$, there exists a smooth density $(z,w) \mapsto p^{\bs{\mu}}_T(z,w)$.
We define the cost function $c^{\bs{\mu}} : \R^{2d} \times \R^{2d} \to \R$:
\begin{equation*}
    c^{\bs{\mu}} (z,w) := -\log p^{\bs{\mu}}_T(z,w),
\end{equation*}
and consider the entropic optimal transport (EOT) problem with the initial boundary $\mu_{\mathrm{init}}$ and the final one $\rho \in \mcal{P}_2(\R^{2d})$:
\begin{equation} \tag*{EOT($\bs{\mu},\rho$)} \label{eq:EOT}
    S^{\bs{\mu}}(\rho) := \inf_{\pi \in \Pi(\mu_{\mathrm{init}}, \rho)} \mcal{F}^{\bs{\mu}}_{\rho}(\pi)
\end{equation}
for
\begin{equation*}
     \mcal{F}^{\bs{\mu}}_{\rho}(\pi) :=\int_{\R^{2d}\times \R^{2d}} c^{\bs{\mu}}(z,w) \, d\pi + H(\pi | \mu_{\mathrm{init}} \otimes \rho).
\end{equation*}
When $|H(\rho | \leb)| < \infty$, the identity \eqref{eq:SB_EOT} (see below) holds, and EOT is equivalent to the static version of \ref{eq:SB_mu} with its final boundary condition replaced by $\rho$. More precisely, for the static SB problem:
\begin{equation*} \label{eq:SB_mu_nu}  \tag*{$\textup{SB}_{\mathrm{stat}}$($\bs{\mu},\rho$)}
  V^{\bs{\mu}}(\rho):= \inf_{\pi \in \Pi (\mu_{\mathrm{init}}, \rho)} H\bigl(\pi\,\big|\, R^{\bs{\mu}}_{0,T} \bigr),
\end{equation*}
where
\[
    \Pi(\mu_{\mathrm{init}}, \rho) := \left\{
        \pi \in \mcal{P}(\R^{2d}\times\R^{2d})
        : \pi(\cdot \times \R^{2d}) = \mu_{\mathrm{init}},\, \pi(\R^{2d} \times \cdot) = \rho
        \right\},
\]
the following relation holds:
\begin{lemma} \label{lemma:SB_EOT}
Let $\rho \in \mcal{P}_2(\R^{2d})$ with $|H(\rho\,|\,\leb)| < \infty$.
Then, for every $\pi \in \Pi(\mu_{\mathrm{init}},\rho)$,
\begin{equation} \label{eq:SB_EOT}
  H\bigl(\pi \,\big|\, R^{\bs{\mu}}_{0,T}\bigr)
  = \mcal{F}^{\bs{\mu}}_{\rho}(\pi) + H(\rho \,|\, \leb)
\end{equation}
and
$V^{\bs{\mu}}(\rho)= S^{\bs{\mu}}(\rho) + H(\rho | \leb)$. In particular, the unique solution $\pi \in \Pi (\mu_{\mathrm{init}}, \rho)$ of $ \textup{SB}_{stat} (\bs{\mu}, \rho)$ attains the infimum of $\textup{EOT}(\bs{\mu}, \rho)$ and vice versa.
\end{lemma}
 This type of result is classical (see \cite[(4.4)]{nutz2021introduction}).  The proof in our setting is provided in Appendix~\ref{appendix:proof_SB_EOT} for the sake of completeness.

\subsubsection*{Step 2: The stability of the value of EOT}
Fix $\bs{\mu} \in C([0,T];\mcal{P}_1(\R^{2d}))$ and $n \in \N$. In the sequel, we use the following notation. For $R := R^{\bs{\mu}}, $
\begin{align*}
  &S(\cdot) := S^{\bs{\mu}}(\cdot) \quad (\text{the value of EOT($\bs{\mu}, \cdot$)}),
  &\qquad R_{0,T} := (Z_0,Z_T)_{\#} R, \\
  &P^n_{\mathrm{soft}} := P^{n,\bs{\mu}} \quad(\text{the solution to $n$-SB($\bs{\mu}$)}),
  &\pi^n_{\mathrm{soft}} := (Z_0,Z_T)_{\#}P^n_{\mathrm{soft}},\\
  &P_{\mathrm{hard}} := P^{\bs{\mu}} \quad  (\text{the solution to SB($\bs{\mu}$)}),
  &\pi_{\mathrm{hard}} := (Z_0,Z_T)_{\#} P_{\mathrm{hard}},
\end{align*}
and
\begin{equation*}
  \nu^n_{\mathrm{soft}} := (Z_T)_{\#} P^n_{\mathrm{soft}}, \qquad \nu_{\mathrm{hard}} := \mu_{\mathrm{fin}}, \qquad h_n := H(\nu^n_{\mathrm{soft}} \,|\, \nu_{\mathrm{hard}}) = G(\nu^n_{\mathrm{soft}}),
\end{equation*}
together with the midpoints
\begin{equation*}
  \nu_{1/2} := \tfrac{1}{2}(\nu^n_{\mathrm{soft}} +  \nu_{\mathrm{hard}}) \in \mcal{P}_2(\R^{2d}),
  \qquad
  \pi_{1/2} := \tfrac{1}{2}(\pi^n_{\mathrm{soft}} + \pi_{\mathrm{hard}}) \in \mcal{P}(\R^{2d}\times\R^{2d}).
\end{equation*}

The goal of this step is to establish the stability result for the value of EOT problems:
\begin{proposition} \label{prop:value_stability}
Let $\rho, \tilde{\rho} \in \mcal{P}_2(\R^{2d})$ satisfying $M^2_2(\rho), M^2_2(\tilde{\rho}) \leq K$ for $K > 0$ defined in Proposition~\ref{prop:bdd_2nd_moment}.
There exists $L = L(T, d, K, \Lambda_0, \Lambda_1, \Lambda_2)>0$, such that
\begin{equation} \label{eq:EOT_value_stability}
 |S(\rho) - S( \tilde{\rho})  | \leq L \mcal{W}_2(\rho,  \tilde{\rho}).
\end{equation}
\end{proposition}

To establish this, we will make use of the quantitative stability for the value $S(\cdot)$ of EOT$(\bs{\mu}, \cdot)$, which is established in \cite[Theorem~3.7(ii)]{eckstein2022quant_stability_EOT}.
To apply the result, we need to check that the cost function $c^{\bs{\mu}}$ satisfies ($\mathrm{A}_L$) in \cite[Definition~3.3]{eckstein2022quant_stability_EOT}:
\begin{lemma} \label{lemma:A_L}
Let $c:\R^{2d} \times \R^{2d} \to \R$ be of class $C^1$ with derivative of linear growth, i.e.\ there exists some constant $C_{cost} >0 $ such that
\begin{equation*}
    |\nabla_{z,w} c(z,w) | \leq C_{cost}  (1 + |z| + |w|), \, z,w \in \R^{2d}.
\end{equation*}
Then,  for $ \mu,  \tilde{\mu}, \nu, \tilde{\nu}  \in \mathcal{P}_2(\R^{2d})$, there exists $L=L (M_2(\mu), M_2(\tilde{\mu}),M_2(\nu), M_2(\tilde{\nu}), C_{cost})>0$ such that
\begin{equation*}
\left| \int c \, d(\pi - \tilde{\pi}) \right| \leq L \, \mcal{W}_2(\pi, \tilde{\pi}),  \tag*{($\mathrm{A}_L$)}
\end{equation*}
for all $\pi \in \Pi(\mu, \nu)$ and $\tilde{\pi} \in \Pi(\tilde{\mu}, \tilde{\nu})$.
\end{lemma}
\begin{proof}
Let $\kappa$ be an optimal coupling between $\pi, \tilde{\pi}$.
\begin{align}
\left| \int c \, d(\pi - \tilde{\pi}) \right| &= \left| \int ( c(z,w)  - c(\tilde{z}, \tilde{w}) )\, d\kappa(z,w,\tilde{z},\tilde{w}) \right| \notag \\
&\leq \left| \int \int^1_0  \nabla c (t (z,w) + (1-t) (\tilde{z}, \tilde{w}) ) ((z,w) - (\tilde{z}, \tilde{w}))  dt d\kappa \right|  \notag \\
&\leq C \int ( 1 + |(z,w)| + |(\tilde{z}, \tilde{w})|  ) \,  |(z,w) - (\tilde{z}, \tilde{w})| d \kappa  \notag  \\
&\leq C
\left( \int ( 1 + |(z,w)|^2 + |(\tilde{z}, \tilde{w})|^2 ) d \kappa \right)^{1/2}
\left( \int |(z,w) - (\tilde{z}, \tilde{w})|^2  d \kappa \right)^{1/2} \notag  \\
&\leq \underset{=: L }{\underline{C (1 + M^2_2(\mu) + M^2_2(\tilde{\mu})  + M^2_2(\nu) + M^2_2(\tilde{\nu}))^{1/2} }}  \, \mcal{W}_2(\pi, \tilde{\pi}). \label{eq:const_L}
\end{align}
\end{proof}

We are now ready to prove the proposition.
\begin{proof}[Proof of Proposition~\ref{prop:value_stability}]
By Theorem~\ref{thm:linear_growth}, $c^{\bs{\mu}}$ satisfies the condition for Lemma~\ref{lemma:A_L}. Hence, we apply \cite[Theorem~3.7 (ii)]{eckstein2022quant_stability_EOT} to deduce
\begin{equation*}
 |S(\rho) - S(\tilde{\rho})  | \leq L \mcal{W}_2(\rho, \tilde{\rho})
\end{equation*}
for $L := C_{\mathrm{cost}}(1+4K)^{1/2}$ by \eqref{eq:const_L}, where $C_{\mathrm{cost}} = C_{\mathrm{cost}}(T, d, \Lambda_0, \Lambda_1, \Lambda_2)$ is the constant of Theorem~\ref{thm:linear_growth}; in particular $L$ is independent of $n$ and $\bs{\mu}$.
\end{proof}

\subsubsection*{Step 3: From the stability of the value to the stability of the plans}
Let $(\mcal{X}, \mcal{G})$ be a measurable space.
We now introduce the Jensen--Shannon divergence: for $\alpha,\beta \in \mcal{P}(\mcal{X})$,
\begin{equation*}
  \js(\alpha,\beta)
  := \tfrac{1}{2} H\Bigl(\alpha \,\Big|\, \tfrac{\alpha+\beta}{2}\Bigr)
   + \tfrac{1}{2} H\Bigl(\beta  \,\Big|\, \tfrac{\alpha+\beta}{2}\Bigr).
\end{equation*}

The goal of this step is to obtain the following stability result for the optimal plans with respect to the divergence:
\begin{proposition} \label{prop:JS_bound}
There exists a constant $C =  C(\alpha, C_{p_{\mathrm{fin}}}, L, K) >0$ such that
\begin{equation} \label{eq:JS_final}
  \js\bigl(\pi^n_{\mathrm{soft}}, \pi_{\mathrm{hard}}\bigr)\;\le\; C n^{-1}.
\end{equation}
\end{proposition}

To this end, we need Lemmas~\ref{lemma:compensation}--\ref{lemma:JS_plans}. The next lemma is a so-called parallelogram identity, which is, for instance, presented in the proof of \cite[Lemma~1.9]{nutz2021introduction} for $\sigma$ appearing below being a probability measure.
\begin{lemma} \label{lemma:compensation}
Let $\sigma$ be a $\sigma$-finite measure and let
$\alpha,\beta \in \mcal{P}(\mcal{X})$ with $m := \tfrac{1}{2}(\alpha+\beta)$.
If $H(\alpha\,|\,\sigma)$ and $H(\beta\,|\,\sigma)$ are finite, then
\begin{equation} \label{eq:compensation}
  \tfrac{1}{2} H(\alpha \,|\, \sigma) + \tfrac{1}{2} H(\beta \,|\, \sigma)
  = H(m \,|\, \sigma) + \js(\alpha,\beta).
\end{equation}
\end{lemma}
\begin{proof}
By $\beta \leq 2m$, it holds that $|\frac{d\beta}{dm}| \leq 2$, $m$-a.s. and the same holds  for $\frac{d\alpha}{dm}$. Hence $H(\alpha |m ), H(\beta |m )< \infty$.
Splitting $\log \frac{d\alpha}{d\sigma}
 = \log \frac{d\alpha}{dm} + \log \frac{dm}{d\sigma}$ and likewise for $\beta$, and using
$\tfrac{1}{2}\alpha + \tfrac{1}{2}\beta = m$ in the second term,
\begin{equation*}
  \tfrac{1}{2}H(\alpha|\sigma) + \tfrac{1}{2}H(\beta|\sigma)
  = \tfrac{1}{2}H(\alpha|m) + \tfrac{1}{2}H(\beta|m)
    + \int \log \frac{dm}{d\sigma} \, dm .\qedhere
\end{equation*}
\end{proof}

We shall also use the following elementary estimate
\begin{lemma} \label{lemma:midpoint_W2}
For $\alpha,\beta \in \mcal{P}_2(\R^{2d})$ and $m = \tfrac{1}{2}(\alpha+\beta)$,
\begin{equation*}
  \mcal{W}_2(\alpha, m) \le \tfrac{1}{\sqrt{2}} \, \mcal{W}_2(\alpha,\beta).
\end{equation*}
\end{lemma}

\begin{proof}
Let $\kappa$ be a $\mcal{W}_2$-optimal coupling of $(\alpha,\beta)$ and set
$\xi := \tfrac{1}{2} (\mathrm{id},\mathrm{id})_\# \alpha + \tfrac{1}{2}\kappa$.
Then $\xi \in \Pi(\alpha,m)$ and
$\int |x-y|^2 \, d\xi = \tfrac{1}{2}\mcal{W}_2^2(\alpha,\beta)$.
\end{proof}

The following estimate is used to invoke Lemmas~\ref{lemma:SB_EOT} and \ref{lemma:compensation} in the proof of Lemma~\ref{lemma:JS_plans}.
\begin{lemma} \label{lemma:finite_entropy}
$|H(\nu^n_{\mathrm{soft}}\,|\,\leb)|,|H(\nu_{1/2}\,|\,\leb)| < \infty$.
\end{lemma}
\begin{proof}
By definition, $H(\nu^n_{\mathrm{soft}} | \nu_{\mathrm{hard}}) < \infty$, and thus $\nu^n_{\mathrm{soft}} \ll \nu_{\mathrm{hard}} \ll \leb$.
\begin{equation*}
    \log \frac{d\nu^n_{\mathrm{soft}}}{d\leb} = \log \frac{d\nu^n_{\mathrm{soft}}}{dR_T} + \log \frac{dR_T}{d\leb}.
\end{equation*}
For the first term, we have
\begin{equation*}
0 \leq H(\nu^n_{\mathrm{soft}} | R_T) \leq H(P^n_{\mathrm{soft}} | R) < \infty.
\end{equation*}
For the upper bound of the second term, by the Gaussian upper bound for the kernel (cf. \cite[Theorem~1.1]{delarue_density_estimates}), it holds
\begin{equation*}
\log \frac{dR_T}{d\leb}(w) = \log \int p^{\bs{\mu}}_T(z,w) d\mu_{\mathrm{init}}(z) \leq \log \sup_{z,w} p^{\bs{\mu}}_T(z,w) \leq \log C,
\end{equation*}
and for the lower bound of the second term, by Lemma~\ref{lemma:kernel_bound}, we have
\begin{equation*}
\log \frac{dR_T}{d\leb}(w)
= \log \int p^{\bs{\mu}}_T(z,w) d\mu_{\mathrm{init}}(z)
\geq \log \left(C\int e^{- C(|z|^2 + |w|^2)  } d\mu_{\mathrm{init}}(z) \right)
\geq \log C - C|w|^2
\end{equation*}
and the right hand side is $\nu^n_{\mathrm{soft}}$-integrable by Proposition~\ref{prop:bdd_2nd_moment}. Hence, $|H(\nu^n_{\mathrm{soft}}\,|\,\leb)| < \infty$.

$H(\nu_{1/2}\,|\, R_{T}) < \infty$ follows from $H(\nu_{\mathrm{hard}}\,|\,R_{T}), H(\nu^n_{\mathrm{soft}}\,|\,R_{T})  < \infty$ and the convexity of $H(\cdot | R_T)$. Therefore, the same argument holds to establish $|H(\nu_{1/2}\,|\,\leb)| < \infty$.
\end{proof}

By this finiteness of entropy, we apply Lemma~\ref{lemma:SB_EOT} to ensure $\pi^n_{\mathrm{soft}}$ and $\pi_{\mathrm{hard}}$ are the solutions to EOT$(\bs{\mu}, \nu^n_{\mathrm{soft}})$ and EOT$(\bs{\mu}, \nu_{\mathrm{hard}})$, respectively. We will use this fact in the proof of the next result.
\begin{lemma} \label{lemma:JS_plans}
For $L>0$ as in Proposition~\ref{prop:value_stability}, it holds that
\begin{equation*}
  \js\bigl(\pi^n_{\mathrm{soft}}, \pi_{\mathrm{hard}}\bigr)
  \;\le\;  \js\bigl(\nu^n_{\mathrm{soft}}, \nu_{\mathrm{hard}}\bigr) + \frac{L}{\sqrt{2}} \, \mcal{W}_2(\nu^n_{\mathrm{soft}}, \nu_{\mathrm{hard}}).
\end{equation*}
\end{lemma}

\begin{proof}
Apply Lemma~\ref{lemma:compensation} twice: once for
$\sigma = R_{0,T}$, $\alpha = \pi^n_{\mathrm{soft}}$, $\beta = \pi_{\mathrm{hard}}$, $m = \pi_{1/2}$,
\begin{equation} \label{eq:convex_defect_plans}
  \tfrac{1}{2} H\bigl(\pi^n_{\mathrm{soft}} \,\big|\, R_{0,T}\bigr)
  + \tfrac{1}{2} H\bigl(\pi_{\mathrm{hard}} \,\big|\, R_{0,T}\bigr)
  = H\bigl(\pi_{1/2} \,\big|\, R_{0,T}\bigr) + \js\bigl(\pi^n_{\mathrm{soft}},\pi_{\mathrm{hard}}\bigr),
\end{equation}
and once for $\sigma = \leb$, $\alpha = \nu^n_{\mathrm{soft}}$,
$\beta = \nu_{\mathrm{hard}}$, $m = \nu_{1/2}$ (by Lemma~\ref{lemma:finite_entropy} and Assumption~\ref{assump:initial_and_target} \ref{assump:log_bound}),
\begin{equation} \label{eq:convex_defect_marginals}
  \tfrac{1}{2} H(\nu^n_{\mathrm{soft}} \,|\, \leb) + \tfrac{1}{2} H(\nu_{\mathrm{hard}} \,|\, \leb)
  = H(\nu_{1/2} \,|\, \leb) + \js(\nu^n_{\mathrm{soft}},\nu_{\mathrm{hard}}).
\end{equation}
By applying Lemma~\ref{lemma:SB_EOT}, we have
\begin{equation} \label{eq:two_values}
  H\bigl(\pi^n_{\mathrm{soft}} \,\big|\, R_{0,T}\bigr) = S(\nu^n_{\mathrm{soft}}) + H(\nu^n_{\mathrm{soft}} \,|\, \leb),
  \qquad
  H\bigl(\pi_{\mathrm{hard}} \,\big|\, R_{0,T}\bigr) = S(\nu_{\mathrm{hard}}) + H(\nu_{\mathrm{hard}} \,|\, \leb).
\end{equation}

Note that $\pi_{1/2} \in \Pi(\mu_{\mathrm{init}},\nu_{1/2})$, since $\pi^n_{\mathrm{soft}}$ and $\pi_{\mathrm{hard}}$ have the
same first marginal $\mu_{\mathrm{init}}$; hence Lemma~\ref{lemma:SB_EOT} applies to $\pi_{1/2}$ as
well and yields
$H(\pi_{1/2}\,|\,R_{0,T}) = \mcal{F}^{\bs{\mu}}_{\nu_{1/2}}(\pi_{1/2})
 + H(\nu_{1/2}\,|\,\leb)$.
Substituting this and \eqref{eq:two_values} into \eqref{eq:convex_defect_plans}, and then
eliminating $\tfrac{1}{2}H(\nu^n_{\mathrm{soft}}|\leb) + \tfrac{1}{2}H(\nu_{\mathrm{hard}}|\leb)$ by means of
\eqref{eq:convex_defect_marginals}, we have the following identity
\begin{equation} \label{eq:key_identity}
  \js\bigl(\pi^n_{\mathrm{soft}},\pi_{\mathrm{hard}}\bigr) - \js(\nu^n_{\mathrm{soft}},\nu_{\mathrm{hard}})
  = \tfrac{1}{2}\bigl( S(\nu^n_{\mathrm{soft}}) + S(\nu_{\mathrm{hard}}) \bigr)
    - \mcal{F}^{\bs{\mu}}_{\nu_{1/2}}(\pi_{1/2}).
\end{equation}
Since $\pi_{1/2} \in \Pi(\mu_{\mathrm{init}},\nu_{1/2})$, we have
$\mcal{F}^{\bs{\mu}}_{\nu_{1/2}}(\pi_{1/2}) \ge S(\nu_{1/2})$,
\begin{align*}
  \js\bigl(\pi^n_{\mathrm{soft}},\pi_{\mathrm{hard}}\bigr) - \js(\nu^n_{\mathrm{soft}},\nu_{\mathrm{hard}})
  &\le \tfrac{1}{2}\bigl( S(\nu^n_{\mathrm{soft}}) - S(\nu_{1/2}) \bigr)
     + \tfrac{1}{2}\bigl( S(\nu_{\mathrm{hard}})   - S(\nu_{1/2}) \bigr) \\
  &\le \tfrac{L}{2} \, \mcal{W}_2(\nu^n_{\mathrm{soft}},\nu_{1/2})
     + \tfrac{L}{2} \, \mcal{W}_2(\nu_{\mathrm{hard}},\nu_{1/2})
  && \text{by \eqref{eq:EOT_value_stability}} \\
  &\le \tfrac{L}{\sqrt{2}} \, \mcal{W}_2(\nu^n_{\mathrm{soft}},\nu_{\mathrm{hard}}),
\end{align*}
where the last line follows by Lemma~\ref{lemma:midpoint_W2}.
\end{proof}

\begin{proof}[Proof of Proposition~\ref{prop:JS_bound}]
By Assumption~\ref{assump:Talagrand}, we apply Talagrand's inequality to get
\begin{equation} \label{eq:talagrand}
\mcal{W}_2(\nu^n_{\mathrm{soft}},\nu_{\mathrm{hard}}) \le C \, H(\nu^n_{\mathrm{soft}} \,|\, \nu_{\mathrm{hard}})^{1/2} = C\, h_n^{1/2}.
\end{equation}

Moreover, $\js(\nu^n_{\mathrm{soft}},\nu_{\mathrm{hard}})$ is controlled directly by applying
Lemma~\ref{lemma:compensation} once more, with
$\sigma = \nu_{\mathrm{hard}}$, and using $H(\nu_{\mathrm{hard}}\,|\,\nu_{\mathrm{hard}}) = 0$ and $H(\nu_{1/2}\,|\,\nu_{\mathrm{hard}})\ge 0$,
\begin{equation} \label{eq:JS_marginals}
  \js(\nu^n_{\mathrm{soft}},\nu_{\mathrm{hard}})
  = \tfrac{1}{2} H(\nu^n_{\mathrm{soft}} \,|\, \nu_{\mathrm{hard}}) - H(\nu_{1/2} \,|\, \nu_{\mathrm{hard}})
  \le \frac{1}{2} h_n.
\end{equation}
Combining  Lemma~\ref{lemma:JS_plans}, \eqref{eq:talagrand}, and \eqref{eq:JS_marginals}, we obtain
\begin{equation*}
\js\bigl(\pi^n_{\mathrm{soft}},\pi_{\mathrm{hard}}\bigr)
  \;\le\; \frac{L}{\sqrt{2}}\, \mcal{W}_2(\nu^n_{\mathrm{soft}},\nu_{\mathrm{hard}}) + \js(\nu^n_{\mathrm{soft}},\nu_{\mathrm{hard}})
  \;\le\; C h_n^{1/2} + C h_n^{}
  \;\le\; \frac{C}{n},
\end{equation*}
where we have used Lemma~\ref{lemma:bootstrap_h_n} below for the last inequality.
\end{proof}

\begin{lemma} \label{lemma:bootstrap_h_n}
There exists a constant $C = C(\alpha, C_{p_{\mathrm{fin}}}, L, K) >0$ independent of $n$ such that
\begin{equation*}
h_n \leq \frac{C}{n^2}.
\end{equation*}
\end{lemma}
\begin{proof}
Since $P_{\mathrm{hard}}$ is one of the feasible plans for the $n$-SB($\bs{\mu}$),
\begin{equation*}
H(P^n_{\mathrm{soft}} | R) + n h_n \leq H(P_{\mathrm{hard}} | R),
\end{equation*}
and equivalently,
\begin{equation*}
H(\pi^n_{\mathrm{soft}} | R_{0,T}) + n h_n \leq H(\pi_{\mathrm{hard}}| R_{0,T}).
\end{equation*}
Using Lemma~\ref{lemma:SB_EOT}, we obtain:
\begin{equation*}
n h_n
\leq S(\nu_{\mathrm{hard}}) - S(\nu^n_{\mathrm{soft}})
 + H( \nu_{\mathrm{hard}}| \leb)  -  H( \nu^n_{\mathrm{soft}}| \leb)
\end{equation*}
Plugging the following identity to the above estimate
\begin{equation*}
H(\nu^n_{\mathrm{soft}} | \leb)
= \int \log \frac{d\nu^n_{\mathrm{soft}}}{d\nu_{\mathrm{hard}}} d\nu^n_{\mathrm{soft}}
+ \int \log \frac{d\nu_{\mathrm{hard}}}{d\leb} d\nu^n_{\mathrm{soft}} = h_n + \int \log p_{\mathrm{fin}} d\nu^n_{\mathrm{soft}},
\end{equation*}
we obtain
\begin{equation*}
(1+n) h_n
\leq \underset{:= \mathrm{I}}{\underline{S(\nu_{\mathrm{hard}}) - S(\nu^n_{\mathrm{soft}})}}
+ \underset{:= \mathrm{II}}{\underline{\int \log p_{\mathrm{fin}} d(\nu_{\mathrm{hard}} - \nu^n_{\mathrm{soft}})}}.
\end{equation*}
For $\mathrm{I}$, by Proposition~\ref{prop:value_stability}, and Talagrand's inequality, we have
\begin{equation*}
\mathrm{I} \leq L\mcal{W}_2(\nu^n_{\mathrm{soft}},\nu_{\mathrm{hard}}) \leq C h_n^{1/2}.
\end{equation*}
For $\mathrm{II}$, by the linear growth condition of $\nabla \log p_{\mathrm{fin}}$, we use the exact same argument as in Lemma~\ref{lemma:A_L} to deduce
\begin{equation*}
\mathrm{II} \leq C \mcal{W}_2(\nu^n_{\mathrm{soft}}, \nu_{\mathrm{hard}}) \leq C h^{1/2}_n.
\end{equation*}
Combining these estimates, we conclude
\begin{equation*}
(1+n) h_n \leq C\sqrt{h_n},
\end{equation*}
from which the statement is direct.
\end{proof}

\subsubsection*{Step 4: From the stability of the plans to the stability of the path measures}
In this step, we will lift up the stability result for the plans derived in Proposition~\ref{prop:JS_bound} to that of the path measures. We begin with the following lemma, which is an analogue of the $T_1$-inequality, with the KL-divergence replaced by JS-divergence.
\begin{lemma} \label{lemma:T_1_inequality_for_JS}
For $Q, \tilde{Q} \in \mcal{P}_2(\Omega)$, it holds that
\begin{equation*}
\mcal{W}_1(Q,\tilde{Q})
\le 2\sqrt{\bigl(M_2^2(Q)+M_2^2(\tilde{Q})\bigr)\js(Q,\tilde{Q})}.
\end{equation*}
\end{lemma}
\begin{proof}
Put \( Q_{1/2}=(Q+\tilde{Q})/2\) and write
\[
\frac{dQ}{dQ_{1/2}}=1+u,\qquad \frac{d\tilde{Q}}{dQ_{1/2}}=1-u,
\]
where $|u|\le 1$, since $0 \leq \frac{dQ}{dQ_{1/2}}, \frac{d\tilde{Q}}{dQ_{1/2}} \leq 2 $.

Since
\[
\frac{(1+u)\log(1+u)+(1-u)\log(1-u)}2
\ge \frac{u^2}{2},
\]we have \(\int u^2\,dQ_{1/2}\le2\operatorname{JS}(Q,\tilde{Q})\).
For any \(1\)-Lipschitz function \(\varphi\) normalized by \(\varphi(\omega_0)=0\) with $\omega_0 = \mathbf{0}$, Cauchy–Schwarz gives
\begin{align*}
\left|\int\varphi\,d(Q-\tilde{Q})\right|
&=2\left|\int\varphi u\,dQ_{1/2}\right|\\
&\le2\left(\int \|\omega - \omega_0\|_{\infty}^2\,dQ_{1/2}\right)^{1/2}
       \left(\int u^2\,dQ_{1/2}\right)^{1/2}\\
&\le2\sqrt{\bigl(M_2^2(Q)+M_2^2(\tilde{Q})\bigr)\operatorname{JS}(Q,\tilde{Q})}. \qedhere
\end{align*}
\end{proof}
We are now ready to prove Theorem~\ref{thm:consistency_error}.
\begin{proof}[Proof of Theorem~\ref{thm:consistency_error}]
By Proposition~\ref{prop:JS_bound} and Lemma~\ref{lemma:T_1_inequality_for_JS}, along with the fact that $P^n_{\mathrm{soft}}, P_{\mathrm{hard}}$ share the same bridge, we derive
\begin{align*}
\mcal{W}_1(P^n_{\mathrm{soft}}, P_{\mathrm{hard}})
\leq C \sqrt{\js(P^n_{\mathrm{soft}}, P_{\mathrm{hard}})}
= C \sqrt{\js(\pi^n_{\mathrm{soft}}, \pi_{\mathrm{hard}})}
\leq C n^{-1/2},
\end{align*}
and the claim follows.
\end{proof}

\subsection{Proof of Theorem \ref{thm:linear_growth}} \label{sec:proof_lineargrowth}
Throughout this section, we grant the smoothness of $(s, t, z, w) \mapsto p^{\bs{\mu}}(s, t; z, w)$ on $\{ 0 \leq s < t \leq T \} \times \R^{2d} \times \R^{2d}$. More precisely, we assume that $p^{\bs{\mu}}$ is $C^1$ in $s$ and $t$, and $C^\infty$ in $z$ and $w$. We defer its proof to Appendix~\ref{appendix:smoothness}, which is based on Malliavin calculus. Indeed, we derive the expressions of this density and its derivatives in Proposition~\ref{prop:density_rep} and \ref{prop:derivatives}.

In the sequel, we write $p$ for $p^{\bs{\mu}}$ and $b(t, z) = f(t, z, \mu_t)$, suppressing the dependence on $\bs{\mu}$; all the constants below are independent of $\bs{\mu}$.

\paragraph{The bridge process.}
Fix $z, w \in \R^{2d}$. Recall that $Z$ denotes the canonical process on the path space $\Omega$, and let $R_z$ be the law of the reference process started at $z$ at time $0$ (omitting the dependence on $\bs{\mu}$), so that under $R_z$,
\begin{equation*}
W_t := Z^v_t - v - \int_0^t b(s, Z_s) \, ds, \qquad t \in [0,T],
\end{equation*}
is a Brownian motion and $dZ^x_u = Z^v_u \, du$, $dZ^v_u = b(u, Z_u) \, du + dW_u$. For $0 \leq t< T$, we set
\begin{equation*}
J(t, y) := - \log p(t, T; y, w), \qquad (t, y) \in [0, T) \times \R^{2d} .
\end{equation*}
For $u < T$ define
\begin{equation} \label{eq:density_martingale}
D_u := \frac{p(u, T; Z_u, w)}{p(0, T; z, w)} = \exp \big( J(0, z) - J(u, Z_u) \big) .
\end{equation}
Since $(u,y) \mapsto p(u, T; y, w)$ solves the Kolmogorov backward equation $\partial_u p + \mcal{L}_u p = 0$, It\^o's formula gives $dD_u = D_u \, \nabla_v \log p(u, T; Z_u, w) \cdot dW_u = - D_u \, \nabla_v J(u, Z_u) \cdot dW_u$, so that $D$ is a positive local martingale under $R_z$ with $D_0 = 1$; by the Gaussian upper bound, for every $t < T$,
\begin{equation} \label{eq:D_bounded}
0 < D_u \leq \frac{C_{T-t}}{p(0, T; z, w)}, \qquad u \in [0,t], \qquad C_{T-t} := \sup_{u \in [0, t]} \ \sup_{y \in \R^{2d}} p(u, T; y, w),
\end{equation}
hence $D$ is a bounded martingale on $[0,t]$. We define the probability measure $Q$ on $\mcal{F}_{t}$, $t < T$, by $dQ/dR_z|_{\mcal{F}_t} := D_t$, and we call the canonical process $Z$ under $Q$ the bridge process from $(0, z)$ to $(T, w)$. By Girsanov's theorem, under $Q$ the process
\begin{equation*}
W^Q_u := W_u + \int_0^u \nabla_v J(r, Z_r) \, dr, \qquad u \in [0,t],
\end{equation*}
is a Brownian motion, and $Z$ satisfies
\begin{equation} \label{eq:bridge_process}
dZ^x_u = Z^v_u \, du, \qquad dZ^v_u = \big( b(u, Z_u) - \nabla_v J(u, Z_u) \big) \, du + dW^Q_u, \qquad Z_0 = z .
\end{equation}
Throughout, $\E$ denotes expectation under $Q$, and $\E^{R_z}$ expectation under $R_z$. The same construction with the starting time $0$ replaced by $s_0 \in [0, T)$ yields the bridge process from $(s_0, z)$ to $(T, w)$, which will be used in the proof of Lemma~\ref{lemma:bounds_by_E(nabla_vJ)}.

\begin{lemma} \label{lemma:control_nabla_J_by_J}
For every $t < T$, there exists $C_{T-t} = C_{T-t}(T, d, \Lambda_0, \Lambda_1)>0$ such that
\begin{equation} \label{eq:bridge-rep}
\E \left[ \frac{1}{2} \int_0^t |\nabla_v J(s, Z_s)|^2 \, ds \right]
= J(0, z) - \E \left[ J(t, Z_t) \right]
= \E \left[ \log D_t \right]
\leq J(0, z) + \log C_{T-t} .
\end{equation}
\end{lemma}

\begin{proof}
The second equality follows from \eqref{eq:density_martingale}, and the inequality from \eqref{eq:D_bounded}. For the first equality, write, using $dW = dW^Q - \nabla_v J \, dr$,
\begin{equation*}
\log D_u = - \int_0^u \nabla_v J(r, Z_r) \cdot dW_r - \frac{1}{2} \int_0^u |\nabla_v J(r, Z_r)|^2 \, dr
= - \int_0^u \nabla_v J(r, Z_r) \cdot dW^Q_r + \frac{1}{2} \int_0^u |\nabla_v J(r, Z_r)|^2 \, dr .
\end{equation*}
Let $\tau_N := \inf\{ s > 0 : |Z_s| \geq N \}$. Since $\nabla_v J$ is bounded on compact sets, the stopped stochastic integral is a $Q$-martingale, and
\begin{equation*}
\E \left[ \frac{1}{2} \int_0^{t \wedge \tau_N} |\nabla_v J(s, Z_s)|^2 \, ds \right]
= \E \left[ \log D_{t \wedge \tau_N} \right]
\leq \log \frac{C_{T-t}}{p(0, T; z, w)} .
\end{equation*}
Since $Q \ll R_z$ and $Z$ does not explode under $R_z$ ($b$ has linear growth), $\tau_N \to \infty$ $Q$-a.s., and the monotone convergence theorem yields $\E [ \frac12 \int_0^t |\nabla_v J(s, Z_s)|^2 \, ds ] < \infty$. Hence $\int_0^{\cdot} \nabla_v J \cdot dW^Q$ is a square-integrable $Q$-martingale on $[0,t]$, and taking expectations in the displayed identity for $\log D_t$ gives the first equality in \eqref{eq:bridge-rep}.
\end{proof}

Our next goal is to obtain the following bounds.
\begin{lemma} \label{lemma:bounds_by_E(nabla_vJ)}
There exists a constant $C = C(d, \Lambda_0, \Lambda_1)>0$, independent of $s_0$, $T$, $\delta$ and $\bs{\mu}$,
such that for any $0 \leq s_0 < T$, any $0 < \delta < T - s_0$ and any $z \in \R^{2d}$, the following bounds hold, where $\E$ is the expectation under the law $Q$ of the bridge process from $(s_0, z)$ to $(T, w)$:
\begin{equation} \label{eq:bounds_by_E(nabla_vJ)}
\begin{split}
|\nabla_x J(s_0, z) |
& \leq C(\delta^{-\frac{3}{2}} + \delta^{-\frac{1}{2}} + \delta^{\frac{1}{2}})
\E\left[\int^{s_0 + \delta}_{s_0}
| \nabla_v J(t, Z_t)|^2  \, dt \right]^{\frac{1}{2}}, \\
|\nabla_v J(s_0, z)|
&\leq C(\delta^{-\frac{1}{2}} + \delta^{\frac{1}{2}} + \delta^{\frac{3}{2}})
\E \left[\int^{s_0 + \delta}_{s_0}
| \nabla_v J(t, Z_t)|^2  \, dt \right]^{\frac{1}{2}}.
\end{split}
\end{equation}
\end{lemma}

In fact, granting these estimates, the linear growth condition of $\nabla_z \log p$ follows immediately. The same property for $\nabla_w \log p$ is detailed in Appendix~\ref{appendix:estimate_for_end_point_variable}.
\begin{proof}[Proof of Theorem~\ref{thm:linear_growth}]
Fix $0 \leq s_0 < T$, set $\delta := (T - s_0)/2$, and let $Q$ be the law of the bridge process from $(s_0, z)$ to $(T, w)$. By Lemma~\ref{lemma:control_nabla_J_by_J} with $t = s_0 + \delta$,
\begin{align*}
    \E \left[ \int^{s_0 + \delta}_{s_0}  |\nabla_v J(t, Z_t)|^2 \, dt \right]^{1/2}
    &= \Big( 2 J(s_0, z) - 2 \, \E \left[ J\big(s_0 + \delta, Z_{s_0 + \delta} \big)  \right] \Big) ^{1/2} \\
    &\leq C \left( 1 + |z|^2 + |w|^2 \right)^{1/2},
\end{align*}
where we have used the upper bound $J(s_0, z) \leq C(1 + |z|^2 + |w|^2)$ of Lemma~\ref{lemma:kernel_bound} for the first term, and, for the second term, the Gaussian upper bound of \cite[Theorem~1.1]{delarue_density_estimates}: writing $w = (w_x, w_v)$, for every $y \in \R^{2d}$,
\begin{align*}
- J(s_0 + \delta, y)
&= \log p(s_0 + \delta, T; y, w) \\
&\leq \log C_{\delta} - C \left( \frac{|\theta^x_{s_0 + \delta, T}(y) - w_x|^2}{\delta^3} + \frac{|\theta^v_{s_0 + \delta, T}(y) - w_v|^2}{\delta} \right) \\
&\leq \log C_{\delta},
\end{align*}
with $\theta_{s,t}(y)$ the flow of the deterministic system $\dot{x} = v$, $\dot{v} = b(t, x, v)$ started at $(s, y)$, as in \cite[Theorem~1.1]{delarue_density_estimates}; here we used that $T - (s_0 + \delta) = \delta$. Hence $-\E[J(s_0 + \delta, Z_{s_0 + \delta})] \leq \log C_{\delta}$. Plugging this into \eqref{eq:bounds_by_E(nabla_vJ)}, we obtain
\begin{equation*}
|\nabla_z J(s_0, z)|
\leq C(\delta^{-\frac{3}{2}} + \delta^{-\frac{1}{2}} + \delta^{\frac{1}{2}} + \delta^{\frac{3}{2}}) \left( 1 + |z|^2 + |w|^2 \right)^{1/2}
\leq C_{T - s_0} \left( 1 + |z| + |w| \right),
\end{equation*}
where $C_{T - s_0}$ depends on $T$ and $T - s_0$ (and blows up as $T - s_0 \to 0$), but not on $z, w$ or $\bs{\mu}$. Since $J(s_0, z) = -\log p(s_0, T; z, w)$, set $s_0 = 0$ to obtain the claimed estimate.
\end{proof}

It now remains to obtain the bounds \eqref{eq:bounds_by_E(nabla_vJ)}.
To this end, we will use the following evolutions, the derivation of which is presented in Appendix~\ref{appendix:derivation_of_evolution_nabla_J}: for $\alpha, \beta \in \{1, \dots, d\}$ and $t \in [s_0, T)$,
\begin{equation} \label{eq:evolutions_of_nabla_J}
\left\{
\begin{split}
d \, \partial_{x_\alpha} J(t, Z_t) &= - \big( \partial_{x_\alpha} b \cdot \nabla_v  J \big)(t, Z_t) \, dt + \nabla_v \partial_{x_\alpha}  J(t, Z_t) \cdot dW^Q_t , \\
d \, \partial_{v_\beta} J(t, Z_t) &= -
\left[
\partial_{x_\beta} J
+ \partial_{v_\beta} b \cdot \nabla_v J
\right] (t, Z_t) \, dt
+  \nabla_v \partial_{v_\beta}  J(t, Z_t) \cdot dW^Q_t .
\end{split}
\right.
\end{equation}

We also need the following estimate.
\begin{lemma} \label{lemma:partial_J_bdd}
Fix $t \in [s_0, T)$ and set $\delta := T - t$. For any stopping time $\tau$ with $s_0 \leq \tau \leq t$, any $A \in \mcal{F}_\tau$, and $\partial \in \{ \partial_{x_\alpha}, \partial_{v_\beta} \}$,
\begin{equation} \label{eq:partial_J_bdd}
\E \left[ \left| \partial J(\tau, Z_\tau) \right| \1_{A} \right]  \leq \frac{K_{\delta}(w)}{p(s_0, T; z, w)} \, R_z(A),
\end{equation}
where
\begin{equation*}
K_{\delta}(w) := \sup_{u \in [0, T - \delta]} \ \sup_{y \in \R^{2d}} \left| \nabla_y p(u, T; y, w) \right| < \infty
\end{equation*}
by Corollary~\ref{cor:bounded_derivative}. Note that $R_z(A)$ is the probability of $A$ under the reference measure, not under $Q$.
\end{lemma}
\begin{proof}[Proof of Lemma~\ref{lemma:partial_J_bdd}]
Recall from \eqref{eq:D_bounded} that $D$ is a bounded $R_z$-martingale on $[s_0, t]$ and that $dQ / dR_z |_{\mcal{F}_t} = D_t$. Since $\left| \partial J(\tau, Z_\tau) \right| \1_A$ is $\mcal{F}_\tau$-measurable, the tower property and the optional sampling theorem give
\begin{align*}
\E \left[ \left| \partial J(\tau, Z_\tau) \right| \1_A \right]
&= \E^{R_z} \left[ \left| \partial J(\tau, Z_\tau) \right| \1_A \, D_t \right]
= \E^{R_z} \left[ \left| \partial J(\tau, Z_\tau) \right| \1_A \, \E^{R_z} \big[ D_t \mid \mcal{F}_\tau \big] \right] \\
&= \E^{R_z} \left[ \left| \partial J(\tau, Z_\tau) \right| \1_A \, D_\tau \right] .
\end{align*}
Now, by the definitions of $J$ and $D$, the ratio $\partial p / p$ cancels against $D_\tau$:
\begin{equation*}
\left| \partial J(\tau, Z_\tau) \right| D_\tau
= \frac{\left| \partial_y p(\tau, T; Z_\tau, w) \right|}{p(\tau, T; Z_\tau, w)} \cdot \frac{p(\tau, T; Z_\tau, w)}{p(s_0, T; z, w)}
= \frac{\left| \partial_y p(\tau, T; Z_\tau, w) \right|}{p(s_0, T; z, w)}
\leq \frac{K_{\delta}(w)}{p(s_0, T; z, w)},
\end{equation*}
since $\tau \leq t = T - \delta$. Taking expectations yields \eqref{eq:partial_J_bdd}.
\end{proof}

We are now ready to prove Lemma~\ref{lemma:bounds_by_E(nabla_vJ)}.

\begin{proof}[Proof of Lemma~\ref{lemma:bounds_by_E(nabla_vJ)}]
Fix $0 \leq s_0 < T$, $0 < \delta < T - s_0$, and let $Q$ be the law of the bridge process from $(s_0, z)$ to $(T, w)$. Set $t_1 := s_0 + \delta$ and $\tau_N := \inf \{ u \geq s_0 :  | Z_{u}| \geq N  \}$ as in the proof of Lemma~\ref{lemma:control_nabla_J_by_J}; recall that $\tau_N \to \infty$ $Q$-a.s.\ and $R_z(\tau_N < t_1) \to 0$ as $N \to \infty$.

\medskip
\noindent \emph{Step 1: two identities.}
Let $s_0 \leq s < t \leq t_1$. From the evolution of $\partial_{v_\beta} J$ in \eqref{eq:evolutions_of_nabla_J},
\begin{align*}
\partial_{v_\beta} J(t \wedge \tau_N, Z_{t \wedge \tau_N})  &=
\partial_{v_\beta} J(s \wedge \tau_N, Z_{s \wedge \tau_N})
-\int^{t\wedge \tau_N}_{s\wedge \tau_N} \partial_{x_\beta} J(u, Z_u) \, du \\
& \quad - \int^{t\wedge \tau_N}_{s\wedge \tau_N} \big( \partial_{v_\beta}b \cdot \nabla_v J \big) (u, Z_u) \, du
+  \int^{t\wedge \tau_N}_{s\wedge \tau_N} \nabla_v \partial_{v_\beta}  J(u, Z_u) \cdot dW^Q_u .
\end{align*}
Since $\nabla_v \partial_{v_\beta} J$ is continuous, hence bounded on $[s_0, t_1] \times \overline{B_N}$, the last term is a true $Q$-martingale. Taking expectations and splitting according to $\{ \tau_N \geq t \}$, $\{ \tau_N \geq s \}$, we obtain
\begin{align*}
&\E\left[\int^{t\wedge \tau_N}_{s\wedge \tau_N} \partial_{x_\beta} J(u, Z_u) \, du \right] \\
&= \E\Bigg[ - \partial_{v_\beta} J(t, Z_{t}) \1_{\{ t \leq \tau_N \}}
+ \partial_{v_\beta} J(s, Z_{s}) \1_{\{ s \leq \tau_N \}}
- \int^{t\wedge \tau_N}_{s\wedge \tau_N} \big( \partial_{v_\beta}b \cdot \nabla_v J \big)(u, Z_u) \, du \Bigg] \\
&\quad - \E\Big[  \partial_{v_\beta} J(\tau_N, Z_{\tau_N}) \1_{\{ s \leq \tau_N < t \}} \Big] .
\end{align*}
We let $N \to \infty$. The event $A := \{ s \leq \tau_N < t \}$ belongs to $\mcal{F}_{\tau}$ for $\tau := \tau_N \wedge t \leq t_1$, and $\tau = \tau_N$ on $A$; hence, by Lemma~\ref{lemma:partial_J_bdd},
\begin{equation*}
\left| \E\Big[  \partial_{v_\beta} J(\tau_N, Z_{\tau_N}) \1_{A} \Big] \right|
\leq \frac{K_{T - t_1}(w)}{p(s_0, T; z, w)} \, R_z(\tau_N < t) \xrightarrow[N \to \infty]{} 0 .
\end{equation*}
The first two terms converge by dominated convergence, with dominating functions $|\partial_{v_\beta} J(t, Z_t)|$ and $|\partial_{v_\beta} J(s, Z_s)|$, which are integrable by Lemma~\ref{lemma:partial_J_bdd} applied with $\tau = t$, resp.\ $\tau = s$, and $A = \Omega$. The third term is dominated by $C \int_s^t |\nabla_v J(u, Z_u)| \, du$, which is integrable by Lemma~\ref{lemma:control_nabla_J_by_J} and the Cauchy--Schwarz inequality; and the left-hand side is dominated by $\int_s^t |\partial_{x_\beta} J(u, Z_u)| \, du$, whose expectation is at most $(t-s) K_{T-t_1}(w) / p(s_0, T; z, w)$ by Tonelli's theorem and Lemma~\ref{lemma:partial_J_bdd}. We conclude that all the terms below are absolutely convergent and
\begin{equation} \label{eq:int_partial_x_J}
\E\left[\int^t_s \partial_{x_\beta} J(u, Z_u) \, du \right]
= \E\Bigg[ - \partial_{v_\beta} J(t, Z_{t})
+ \partial_{v_\beta} J(s, Z_{s}) - \int^t_s \big( \partial_{v_\beta}b \cdot \nabla_v J \big)(u, Z_u) \, du \Bigg] .
\end{equation}
In exactly the same way, starting from the evolution of $\partial_{x_\beta} J$ in \eqref{eq:evolutions_of_nabla_J}, we obtain for every $u \in [s_0, t_1]$
\begin{equation} \label{eq:partial_x_J_identity}
 \partial_{x_\beta} J(s_0, z)
 = \E\left[
 \partial_{x_\beta} J(u, Z_{u}) + \int^{u}_{s_0} \big( \partial_{x_\beta} b \cdot \nabla_v J \big)(r, Z_r) \, dr
\right] ,
\end{equation}
and, in the same way again, for every $t \in [s_0, t_1]$
\begin{equation} \label{eq:partial_v_J_identity}
\partial_{v_\beta} J(s_0, z)
= \E \left[
\partial_{v_\beta} J(t, Z_t)
+ \int^t_{s_0} \Big( \partial_{x_\beta} J + \partial_{v_\beta} b \cdot \nabla_v J \Big)(r, Z_r) \, dr
\right] .
\end{equation}

\medskip
\noindent \emph{Step 2: the estimate for $\nabla_x J$.}
Integrating \eqref{eq:partial_x_J_identity} over $u \in [s, t]$ for $s_0 \leq s < t \leq t_1$ gives
\begin{equation} \label{eq:partial_x_J_before_plug_in}
(t-s) \, \partial_{x_\beta} J(s_0, z)
 = \E\left[
 \int^t_s \partial_{x_\beta} J(u, Z_{u}) \, du + \int^t_s \int^u_{s_0} \big( \partial_{x_\beta} b \cdot \nabla_v J \big)(r, Z_r) \, dr \, du
\right] .
\end{equation}
Plugging \eqref{eq:int_partial_x_J} into \eqref{eq:partial_x_J_before_plug_in}, we compute
\begin{align}
(t-s) \, \partial_{x_\beta} J(s_0, z)
&= \E\Bigg[
- \partial_{v_\beta} J(t, Z_{t}) + \partial_{v_\beta} J(s, Z_{s})
- \int^t_s \big( \partial_{v_\beta}b \cdot \nabla_v J \big)(u, Z_u) \, du \notag\\
&\quad + (t-s) \int^s_{s_0} \big( \partial_{x_\beta} b \cdot \nabla_v J \big)(u, Z_u) \, du
 + \int^t_s (t-u) \big( \partial_{x_\beta} b \cdot \nabla_v J \big)(u, Z_u) \, du
\Bigg] , \label{eq:before_double_integral}
\end{align}
where we have used the following computation:
\begin{align*}
    \int^t_s \int^u_{s_0} f(r) \, dr \, du
    &= \int^t_{s_0} \1_{\{ s \leq u \}} \int^t_{s_0} f(r) \1_{\{ r \leq u \}}  \, dr \, du
    = \int^t_{s_0}  f(r)  \int^t_{s_0} \1_{\{ s \leq u \}} \1_{\{ r \leq u \}}  \, du \, dr\\
    &= \int^t_{s_0}  f(r) \big( t - \max (s, r ) \big) \, dr
    = (t-s) \int^s_{s_0} f(r) \, dr +  \int^t_s  f(r) (t - r) \, dr .
\end{align*}
Since $|\partial_{v_\beta} J| \leq |\nabla_v J|$, $\| \nabla b \|_\infty \leq C$, and $t - s \leq \delta$, $t - u \leq \delta$ on the relevant ranges, \eqref{eq:before_double_integral} yields
\begin{equation} \label{eq:before_averaging}
(t-s) \left| \partial_{x_\beta} J(s_0, z) \right|
\leq \E \Big[ \left| \nabla_v J(t, Z_t) \right| + \left| \nabla_v J(s, Z_s) \right| \Big]
+ C (1 + \delta) \, \E \left[ \int^{t_1}_{s_0} \left| \nabla_v J(u, Z_u) \right| du \right] .
\end{equation}
The identity \eqref{eq:before_double_integral}, hence \eqref{eq:before_averaging}, holds for every $s_0 \leq s < t \leq t_1$; in particular for $s \in [s_0, s_0 + \delta/3]$ and $t \in [s_0 + 2\delta/3, t_1]$, on which range $t - s \geq \delta/3$. Integrating \eqref{eq:before_averaging} over this range and using that each of the two intervals has length $\delta/3$, we obtain
\begin{align*}
\left(\frac{\delta}{3}\right)^3 \left| \partial_{x_\beta} J(s_0, z) \right|
&\leq \frac{\delta}{3} \, \E \left[ \int^{t_1}_{s_0} \left| \nabla_v J(u, Z_u) \right| du \right]
+ \left(\frac{\delta}{3}\right)^2 C (1+\delta) \, \E \left[ \int^{t_1}_{s_0} \left| \nabla_v J(u, Z_u) \right| du \right] \\
&\leq C \, \delta \, (1 + \delta + \delta^2) \, \E \left[ \int^{t_1}_{s_0} \left| \nabla_v J(u, Z_u) \right| du \right] \\
&\leq C \, (\delta^{\frac{3}{2}} + \delta^{\frac{5}{2}} + \delta^{\frac{7}{2}}) \, \E \left[ \int^{t_1}_{s_0} \left| \nabla_v J(u, Z_u) \right|^2 du \right]^{\frac{1}{2}},
\tag*{\text{by the Cauchy--Schwarz inequality.}}
\end{align*}
Dividing by $(\delta/3)^3$, we conclude
\begin{equation} \label{eq:nabla_x_J_final}
 \left|\partial_{x_\beta} J(s_0, z) \right|
 \leq C(\delta^{-\frac{3}{2}} + \delta^{-\frac{1}{2}} + \delta^{\frac{1}{2}}) \, \E \left[ \int^{s_0 + \delta}_{s_0} \left| \nabla_v J(u, Z_u) \right|^2 du \right]^{\frac{1}{2}} ,
\end{equation}
which is the first bound in \eqref{eq:bounds_by_E(nabla_vJ)}. We stress that \eqref{eq:nabla_x_J_final} has been established for every starting point $(s_0, z)$ and every window $0 < \delta < T - s_0$, with a constant $C$ independent of these; this is used in Step 3.

\medskip
\noindent \emph{Step 3: the estimate for $\nabla_v J$.}
Integrating \eqref{eq:partial_v_J_identity} over $t \in [s_0, s_0 + \delta/2]$ and using the computation of Step 2 for the double integral, we estimate
\begin{align}
\frac{\delta}{2} \left| \partial_{v_\beta} J(s_0, z) \right|
&\leq
\E \left[ \int^{s_0 + \delta/2}_{s_0} \left| \nabla_v J(t, Z_t) \right| dt \right]
+ \int^{s_0 + \delta/2}_{s_0} \left( s_0 + \frac{\delta}{2} - u \right) \E \left[ \left| \partial_{x_\beta} J(u, Z_u) \right| \right] du \notag \\
&\quad + C \, \E \left[ \int^{s_0+\delta/2}_{s_0} \left( s_0 + \frac{\delta}{2} - u \right) \left| \nabla_v J(u, Z_u) \right| du \right]
=: \mathrm{I} + \mathrm{II} + \mathrm{III} . \label{eq:nabla_v_three_terms}
\end{align}
By the Cauchy--Schwarz inequality,
\begin{equation*}
\mathrm{I} \leq \delta^{\frac{1}{2}} \, \E \left[ \int_{s_0}^{t_1} \left| \nabla_v J(t, Z_t) \right|^2 dt \right]^{\frac{1}{2}},
\qquad
\mathrm{III} \leq C \, \delta^{\frac{3}{2}} \, \E \left[ \int_{s_0}^{t_1} \left| \nabla_v J(t, Z_t) \right|^2 dt \right]^{\frac{1}{2}} .
\end{equation*}
For $\mathrm{II}$, fix $u \in [s_0, s_0 + \delta/2]$ and set $\delta_u := \delta - (u - s_0) \in [\delta/2, \delta]$, so that $u + \delta_u = t_1$ and $0 < \delta_u < T - u$. Applying \eqref{eq:nabla_x_J_final} with starting point $(u, y)$ and window $\delta_u$ gives, for every $y \in \R^{2d}$,
\begin{equation*}
\left| \partial_{x_\beta} J(u, y) \right|
\leq C \left( \delta_u^{-\frac{3}{2}} + \delta_u^{-\frac{1}{2}} + \delta_u^{\frac{1}{2}} \right)
\E^{Q_{u,y}} \left[ \int_u^{t_1} \left| \nabla_v J(r, Z_r) \right|^2 dr \right]^{\frac{1}{2}},
\end{equation*}
where $Q_{u,y}$ is the law of the bridge process from $(u, y)$ to $(T, w)$. By the evolution of $Z$ being Markovian, the conditional law under $Q$ of $(Z_r)_{r \in [u, T)}$ given $\mcal{F}_u$ is $Q_{u, Z_u}$; hence, substituting $y = Z_u$,
\begin{equation*}
\E^{Q_{u, Z_u}} \left[ \int_u^{t_1} \left| \nabla_v J(r, Z_r) \right|^2 dr \right]
= \E \left[ \int_u^{t_1} \left| \nabla_v J(r, Z_r) \right|^2 dr \,\Big|\, \mcal{F}_u \right] .
\end{equation*}
Since $\delta_u \geq \delta/2$, we have $\delta_u^{-3/2} + \delta_u^{-1/2} + \delta_u^{1/2} \leq C (\delta^{-3/2} + \delta^{-1/2} + \delta^{1/2})$. Taking expectations and using Jensen's inequality for the concave function $\sqrt{\cdot}$,
\begin{align*}
\E \left[ \left| \partial_{x_\beta} J(u, Z_u) \right| \right]
&\leq C (\delta^{-\frac{3}{2}} + \delta^{-\frac{1}{2}} + \delta^{\frac{1}{2}}) \,
\E \left[ \E \left[ \int_u^{t_1} \left| \nabla_v J(r, Z_r) \right|^2 dr \,\Big|\, \mcal{F}_u \right]^{\frac{1}{2}} \right] \\
&\leq C (\delta^{-\frac{3}{2}} + \delta^{-\frac{1}{2}} + \delta^{\frac{1}{2}}) \,
\E \left[ \int_{s_0}^{t_1} \left| \nabla_v J(r, Z_r) \right|^2 dr \right]^{\frac{1}{2}} .
\end{align*}
Therefore, since $s_0 + \delta/2 - u \leq \delta/2$,
\begin{equation*}
\mathrm{II}
\leq \frac{\delta}{2} \int_{s_0}^{s_0+\delta/2} \E \left[ \left| \partial_{x_\beta} J(u, Z_u) \right| \right] du
\leq C \, (\delta^{\frac{1}{2}} + \delta^{\frac{3}{2}} + \delta^{\frac{5}{2}}) \,
\E \left[ \int_{s_0}^{t_1} \left| \nabla_v J(r, Z_r) \right|^2 dr \right]^{\frac{1}{2}} .
\end{equation*}
Plugging the bounds for $\mathrm{I}$, $\mathrm{II}$, $\mathrm{III}$ into \eqref{eq:nabla_v_three_terms} and dividing by $\delta/2$, we obtain
\begin{equation*}
\left| \partial_{v_\beta} J(s_0, z) \right|
\leq C(\delta^{-\frac{1}{2}} + \delta^{\frac{1}{2}} + \delta^{\frac{3}{2}}) \,
\E \left[ \int^{s_0 + \delta}_{s_0} \left| \nabla_v J(t, Z_t) \right|^2 dt \right]^{\frac{1}{2}},
\end{equation*}
which is the second bound in \eqref{eq:bounds_by_E(nabla_vJ)} and concludes the proof.
\end{proof}

\begin{remark}[The endpoint variable]
The estimate for $\nabla_w \log p$ is obtained by running the program of this section for the endpoint variable: we fix the starting point $(0, z)$, set $I(t, w) := - \log p(0, t; z, w)$ as in \eqref{eq:def_I}, and replace the bridge process by its time reversal. Since $p$ now solves the Fokker--Planck equation in the forward variables $(t, w)$ rather than the Kolmogorov backward equation, $I$ satisfies the equation \eqref{eq:HJB_I}, which differs from the HJB equation \eqref{eq:HJB} in two respects: the transport and diffusion terms have the opposite sign, which is why the relevant bridge is the time-reversed one, and there is an additional source term $\nabla_{w_v} \cdot b$. The latter enters the evolution equations \eqref{eq:evolutions_of_nabla_I} for $\nabla_{w_x} I$ and $\nabla_{w_v} I$ through the terms $\nabla_{w_v} \cdot \partial_{w_x} b$ and $\nabla_{w_v} \cdot \partial_{w_v} b$, which are bounded by Assumption~\ref{assump:bounded_derivative} uniformly in $t$ and $\bs{\mu}$, and hence do not affect the structure of the argument. The details are provided in Appendix~\ref{appendix:estimate_for_end_point_variable}.
\end{remark}

\section{Future direction} \label{sec:discussion}
In this paper, we studied a mean-field planning problem for kinetic agents with degenerate noise.
We conclude with some open problems.

\paragraph{Uniqueness.}
A natural direction is to study the uniqueness, or non-uniqueness, of MFE. Whenever uniqueness
holds together with a stability property of the best-response map
$\Psi : \bs{\mu} \mapsto (P^{\bs{\mu}}_t)_{t \in [0,T]}$, the consistency error translates into a
rate of convergence of $n$-MFE to the unique MFE $\hat{\bs{\mu}}$. For instance, if $\Psi$ is a
contraction with constant $\kappa < 1$ for the metric
$d(\bs{\mu}, \bs{\nu}) := \sup_{t \in [0,T]} \mcal{W}_1(\mu_t, \nu_t)$, as one may expect for
short time horizons, then
\[
  d(\bs{\mu}^n, \hat{\bs{\mu}})
  \le d\bigl(\bs{\mu}^n, \Psi(\bs{\mu}^n)\bigr) + d\bigl(\Psi(\bs{\mu}^n), \Psi(\hat{\bs{\mu}})\bigr)
  \le d\bigl(\bs{\mu}^n, \Psi(\bs{\mu}^n)\bigr) + \kappa\, d(\bs{\mu}^n, \hat{\bs{\mu}}),
\]
so that $d(\bs{\mu}^n, \hat{\bs{\mu}}) \le (1-\kappa)^{-1} C n^{-1/2}$ by
Theorem~\ref{thm:consistency_error}.

\paragraph{Optimality of the rate.}
Whether the rate $n^{-1/2}$ is sharp remains open. If $c^{\bs{\mu}^n} + \tilde{\psi}$ were
semiconcave in the terminal variable, where $\tilde{\psi}$ denotes the EOT potential of the
hard-constrained problem, then \cite[Theorem~1.1]{chiarini2024semiconcavity} would yield
$H(\pi \,|\, \tilde{\pi}) \le H(\rho \,|\, \tilde{\rho}) + C \mcal{W}^2_2(\rho, \tilde{\rho})$.
Applied with $\rho = \mu^n_T$ and $\tilde{\rho} = \mu_{\mathrm{fin}}$, together with Talagrand's
inequality and the bound $H(\mu^n_T \,|\, \mu_{\mathrm{fin}}) \le C n^{-2}$ of
Lemma~\ref{lemma:bootstrap_h_n}, this gives $H(\pi^n_{\mathrm{soft}} \,|\, \pi_{\mathrm{hard}}) \le C n^{-2}$.
Since the two path measures share the same bridges, a $T_1$-type inequality as in
Lemma~\ref{lemma:T_1_inequality_for_JS} for the relative entropy in place of the
Jensen--Shannon divergence, would then improve the bound on the consistency error to
$\mcal{O}(n^{-1})$, which matches the lower bound in Remark~\ref{remark:lower_bound}.
However, this semiconcavity requires, for instance, a bound on the Hessian
$\nabla_w \otimes \nabla_w \log p^{\bs{\mu}}_T$ that is uniform in $\bs{\mu}$, which cannot be
expected from the approach of \cite{sheu1991_someestimates} that we adopt for
Theorem~\ref{thm:linear_growth}: even in the non-degenerate case, this approach yields second-order
estimates with quadratic growth only \cite[(1.16) with $|m| = 2$]{sheu1991_someestimates}, which
falls short of a bounded Hessian.

\appendix
\section{Deferred proofs} \label{appendix:deferred_proofs}
\subsection{Proof of Lemma~\ref{lemma:kernel_bound}} \label{sec:proof_of_kernel_lowerbound}
\begin{proof}
We write $z = (x,v), w = (y,u)$.
By \cite[Theorem~1.1]{delarue_density_estimates}, there exists $C= C(T,d, C_f)$, such that
\begin{equation*}
p^{\bs{\mu}}_T( (x,v) , (y,u) )
\geq
\frac{1}{C} T^{-2d}\exp \left(
-  \frac{C}{T^3}|\theta^1_T - y |^2 - \frac{C}{T}|\theta^2_T- u|^2
\right),
\end{equation*}
where $\theta = (\theta^1, \theta^2)$ is the solution of the following initial value problem
\begin{equation*}
\label{eq:reference_SDE_1}
     d \theta^1_t = \theta^2_t dt, \quad
     d \theta^2_t = f(t, \theta^1_t, \theta^2_t, \mu_t)dt, \quad
     \theta^1_0 = x, \theta^2_0 = v.
\end{equation*}
By Assumption \ref{assump:coefficient} \ref{assump:linear_growth} and the standard estimates with Gronwall's inequality, we get
\begin{equation*}
\sup_{t\in[0,T]}|\theta^1_t|^2 + \sup_{t\in[0,T]}|\theta^2_t|^2
\leq e^{CT} (|x|^2 + |v|^2) + e^{CT},
\end{equation*}
where $C$ is a constant independent of $\bs{\mu}$.
Therefore, we have
\begin{equation*}\label{eq:bound_entropy}
\begin{aligned}
p^{\bs{\mu}}_T( (x,v) , (y,u) ) &\geq
\frac{1}{C} T^{-2d}\exp \left(
-  \frac{C}{T^3} \left(|\theta^1_T|^2 + |y|^2 \right)
- \frac{C}{T} \left(|\theta^2_T|^2 +  |u|^2 \right)
\right)
\\
&\geq
\frac{1}{C} T^{-2d}\exp \left(
-  \frac{C}{T^3}\left(e^{C T} (|x|^2 + |v|^2)  + |y |^2 \right)
- \frac{C}{T} \left( e^{C T} (|x|^2 + |v|^2) +  |u|^2
 \right)
\right)\\
&\geq
\frac{1}{C}\exp \left(
- C \left(|x|^2  + |v|^2 +  |y |^2  +  |u|^2  \right)
\right),
\end{aligned}
\end{equation*}
where in the last line, all the $T$ terms were absorbed into the constant $C > 0.$
\end{proof}

\subsection{Non-uniqueness of MFE (Remark~\ref{remark:non-uniqueness})}
\label{appendix:non-uniqueness}

Theorem~\ref{thm:existence_mfe} gives existence of an MFE but not uniqueness. We show that uniqueness may fail under Assumptions~\ref{assump:initial_and_target}, \ref{assump:coefficient} and~\ref{assump:bounded_derivative}, even with Gaussian boundary laws and a bounded, smooth drift.
The mechanism is a resonance between the mean position of the population
and a self-induced harmonic force; it is not of Cucker--Smale type.

Throughout this subsection $d=1$. For $\mu\in\mcal{P}_1(\R^2)$ we
write $m_x(\mu):=\int_{\R^2}y\,\mu(dy,dw)$ for its mean position, and we let
$R^0$ be the law of the driftless kinetic process \eqref{eq:driftless} started from
$Z_0\sim\mcal{N}(0,I_2)$, so that $R^0_t=\mcal{N}(0,\Sigma^0_t)$ with
\begin{equation}
\label{eq:Sigma_kinetic}
  \Sigma^0_t:=
  \begin{pmatrix}
    1+t^2+t^3/3 & t+t^2/2\\
    t+t^2/2 & 1+t
  \end{pmatrix},
  \qquad t\in[0,T].
\end{equation}

\begin{proposition}[A continuum of Gaussian MFEs]
\label{prop:nonuniqueness}
Let $\omega:=2\pi/T$, let $S\in C_b^\infty(\R)$ be globally Lipschitz
with $S(r)=r$ for $|r|\leq1$, and set
\begin{equation}
\label{eq:nonuniqueness-drift}
  f(t,x,v,\mu):=-\omega^2 S\bigl(m_x(\mu)\bigr),
  \qquad
  \mu_{\mathrm{init}}:=\mcal{N}(0,I_2),
  \qquad
  \mu_{\mathrm{fin}}:=\mcal{N}(0,\Sigma^0_T).
\end{equation}
Then $f$ satisfies Assumptions~\ref{assump:coefficient} and~\ref{assump:bounded_derivative},
$(\mu_{\mathrm{init}},\mu_{\mathrm{fin}})$ satisfies Assumption~\ref{assump:initial_and_target}, and for
every $c\in[-\tfrac12,\tfrac12]$ the flow
\begin{equation}
\label{eq:nonuniqueness-family}
  \mu^c_t:=\mcal{N}\bigl((q_c(t),\dot q_c(t)),\Sigma^0_t\bigr),
  \qquad
  q_c(t):=c\bigl(\cos(\omega t)-1\bigr),
\end{equation}
is a mean-field equilibrium. These equilibria are pairwise distinct, since
$m_x(\mu^c_{T/2})=-2c$.
\end{proposition}

\begin{proof}
\emph{Assumptions.}
Since $|m_x(\mu)-m_x(\nu)|\leq \mcal{W}_1(\mu,\nu)$, the drift $f$ is bounded and
Lipschitz in $\mu$; being independent of $(t,z)$, it satisfies
Assumptions~\ref{assump:coefficient} and~\ref{assump:bounded_derivative}. The boundary laws are
nondegenerate Gaussians, so Assumption~\ref{assump:initial_and_target} holds.

\emph{A candidate with constant control.}
Fix $c\in[-\tfrac12,\tfrac12]$ and let $Q^c$ be the law of
$Z_t+(q_c(t),\dot q_c(t))$ under $R^0$; its marginal flow is
$\bs{\mu}^c=(\mu^c_t)_{t\in[0,T]}$ in~\eqref{eq:nonuniqueness-family}.
Because $q_c(0)=q_c(T)=\dot q_c(0)=\dot q_c(T)=0$, we have
$Q^c_0=\mu_{\mathrm{init}}$ and $Q^c_T=\mu_{\mathrm{fin}}$.
Moreover $|q_c|\leq 2|c|\leq1$, so along the frozen flow $\bs{\mu}^c$
the reference drift is deterministic:
$f(t,x,v,\mu^c_t)=-\omega^2q_c(t)$. Since
$\ddot q_c=-\omega^2q_c-c\,\omega^2$, under $Q^c$ the canonical process solves
\begin{equation*}
  dX_t=V_t\,dt,
  \qquad
  dV_t=\bigl(f(t,X_t,V_t,\mu^c_t)+\alpha^c\bigr)\,dt+dW^{Q^c}_t,
  \qquad
  \alpha^c:=-c\,\omega^2,
\end{equation*}
i.e.\ $Q^c$ is admissible for~\eqref{eq:controlled_SDE} with the frozen flow
$\bs{\mu}^c$ and the constant control $\alpha^c$, and by Girsanov's theorem
$H(Q^c\,|\,R^{\bs{\mu}^c})=\tfrac12\int_0^T|\alpha^c|^2\,dt
=\tfrac{T}{2}c^2\omega^4$.

\emph{Optimality.}
Let $Q$ be admissible for \ref{eq:SB_mu} with $\bs{\mu}=\bs{\mu}^c$ and
$H(Q\,|\,R^{\bs{\mu}^c})<\infty$, and let $\alpha$ be the control
given by Lemma~\ref{lemma:entropy_is_kinetic_energy}, so that
$V_T-V_0=\int_0^T\bigl(-\omega^2q_c(t)+\alpha_t\bigr)\,dt+W^Q_T$ with $W^Q$
a $Q$-BM. Both $\mu_{\mathrm{init}}$ and $\mu_{\mathrm{fin}}$ are centered, hence
taking expectations and using $\int_0^Tq_c(t)\,dt=-cT$,
\begin{equation*}
  \int_0^T\E^Q[\alpha_t]\,dt=-c\,\omega^2T .
\end{equation*}
By Jensen's inequality on $[0,T]\times\Omega$,
\begin{equation*}
  H(Q\,|\,R^{\bs{\mu}^c})
  =\frac12\,\E^Q\!\int_0^T|\alpha_t|^2\,dt
  \geq\frac{1}{2T}\Bigl|\int_0^T\E^Q[\alpha_t]\,dt\Bigr|^2
  =\frac{T}{2}c^2\omega^4
  =H(Q^c\,|\,R^{\bs{\mu}^c}).
\end{equation*}
Hence $Q^c$ solves \ref{eq:SB_mu} for $\bs{\mu}^c$, and by the uniqueness
in Proposition~\ref{prop:well-posedness_SB_mu}, $P^{\bs{\mu}^c}=Q^c$. Its
marginal flow is $\bs{\mu}^c$, so $\bs{\mu}^c$ is a mean-field equilibrium.
\end{proof}

\subsection{Proof of Lemma~\ref{lemma:compact_sublevel}} \label{appendix:proof_compact_sublevel}
\begin{proof}
We may assume $H(Q|R^{\bs{\mu}}) < \infty$, since otherwise
\eqref{eq:dataprocessing_ineq} below is trivial. We first show that
\begin{equation}
Q(A) \leq \dfrac{H(Q|R^{\bs{\mu}}) + \log 2}{\log \frac{1}{R^{\bs{\mu}}(A)}},
\label{eq:dataprocessing_ineq}
\end{equation}
for every Borel set $A \subset \Omega$ with $0 < R^{\bs{\mu}}(A) < 1$.
Granted \eqref{eq:dataprocessing_ineq}, tightness of the sublevel set
$\{ Q : H(Q|R^{\bs{\mu}}) \leq c \}$ follows from that of the singleton
$\{R^{\bs{\mu}}\}$. Indeed, fix an arbitrary $\varepsilon > 0$ and set
$\delta := \exp\big(-\frac{c+\log 2}{\varepsilon}\big) \in (0,1)$. Since
$\Omega$ is Polish, there exists a compact set $K^{\delta} \subset \Omega$
such that $R^{\bs{\mu}}(K^{\delta}) > 1 - \delta$; note that $K^{\delta}$
depends only on $c$ and $\varepsilon$, and not on $Q$. Let $Q$ satisfy
$H(Q|R^{\bs{\mu}}) \leq c$. If $R^{\bs{\mu}}((K^{\delta})^{c}) = 0$, then
$Q((K^{\delta})^{c}) = 0$ because $Q \ll R^{\bs{\mu}}$. Otherwise
$0 < R^{\bs{\mu}}((K^{\delta})^{c}) < \delta < 1$, so that
$\log \frac{1}{R^{\bs{\mu}}((K^{\delta})^{c})} > \frac{c + \log 2}{\varepsilon}$,
and \eqref{eq:dataprocessing_ineq} applied with $A = (K^{\delta})^{c}$ gives
\begin{equation*}
    Q\big( (K^{\delta})^{c} \big)
    \leq \frac{c + \log 2}{\log \frac{1}{R^{\bs{\mu}}((K^{\delta})^{c})}}
    \leq \varepsilon .
\end{equation*}
In either case $Q((K^{\delta})^{c}) \leq \varepsilon$ for every $Q$ with
$H(Q|R^{\bs{\mu}}) \leq c$, and since $K^{\delta}$ does not depend on $Q$,
the sublevel set is tight.

It remains to prove \eqref{eq:dataprocessing_ineq}. With the convention
$0 \log 0 = 0$,
\begin{equation*}
\begin{split}
H(Q |R^{\bs{\mu}} )
&= \int_A \frac{dQ}{dR^{\bs{\mu}}} \log\left(\frac{dQ}{dR^{\bs{\mu}}}\right) dR^{\bs{\mu}}
 + \int_{A^c} \frac{dQ}{dR^{\bs{\mu}}} \log\left(\frac{dQ}{dR^{\bs{\mu}}}\right) dR^{\bs{\mu}} \\
&\geq Q(A) \log\left(\frac{Q(A)}{R^{\bs{\mu}}(A)}\right)
 + Q(A^c) \log\left(\frac{Q(A^c)}{R^{\bs{\mu}}(A^c)}\right) \\
&= Q(A) \log\left(Q(A)\right) + (1 - Q(A)) \log\left(1 - Q(A)\right)
 \\
&\quad + Q(A) \log\left(\frac{1}{R^{\bs{\mu}}(A)}\right)
 + \underset{\geq 0}{\underline{(1 - Q(A)) \log\left(\frac{1}{1 - R^{\bs{\mu}}(A)}\right)}} \\
&\geq - \log 2 + Q(A) \log\left(\frac{1}{R^{\bs{\mu}}(A)}\right),
\end{split}
\end{equation*}
where we have used Jensen's inequality for the convex function
$t \mapsto t \log t$ in the second line, and the fact that
$[0,1] \ni t \mapsto t\log t + (1-t)\log(1-t)$ attains its minimum $-\log 2$
at $t = 1/2$ in the last line. Thus \eqref{eq:dataprocessing_ineq} is
established.

Finally, the sublevel set is closed, since $H(\,\cdot\,|R^{\bs{\mu}})$ is
l.s.c.\ with respect to the weak topology.
Being tight and closed, it is compact by Prohorov's theorem.
\end{proof}

\subsection{Proof of Remark~\ref{remark:lower_bound}} \label{appendix:lower_bound}
Let $f\equiv 0$, $G=H(\cdot\,|\,\mu_{\mathrm{fin}})$,
$\mu_{\mathrm{init}}=\mathcal N(m_{\mathrm{init}},\Sigma_{\mathrm{init}})$ and
$\mu_{\mathrm{fin}}=\mathcal N(m_{\mathrm{fin}},\Sigma_{\mathrm{fin}})$ with $\Sigma_{\mathrm{init}},\Sigma_{\mathrm{fin}}$ positive definite.
Since $R:=R^{\mu}$ does not depend on $\mu$, neither do the solutions $P^n$ of $n$-SB$(\mu)$ and $P$ of SB$(\mu)$.
Hence $\mu^n:=(P^n_t)_{t\in[0,T]}$ is the unique $n$-MFE, $P^{\mu^n}=P$, and, testing $\mcal{W}_1$ against a linear function,
\begin{equation*}
  \sup_{t\in[0,T]} \mcal{W}_1\big(\mu^n_t,P^{\mu^n}_t\big)\;\ge\; \mcal{W}_1(\mu^n_T,\mu_{\mathrm{fin}})\;\ge\;|m^n-m_{\mathrm{fin}}|,
  \qquad \nu^n:=P^n_T,\quad m^n:=\textstyle\int y\,\nu^n(dy).
\end{equation*}
To compute $m^n$, recall that $\pi^n:=(Z_0,Z_T)_\#P^n$ minimizes $H(\pi\,|\,R_{0,T})+nH(\pi_T\,|\,\mu_{\mathrm{fin}})$
over $\pi_0=\mu_{\mathrm{init}}$, where by \eqref{eq:affine_structure}--\eqref{eq:Sigma}, $R_{0,T}(dz,dw)=\mu_{\mathrm{init}}(dz)\,\mathcal N(\Theta_T(z),\Sigma_T)(dw)$.
For $e\in\mathbb R^{2d}$, consider the perturbed law $\pi^e$ of $(Z_0,Z_T+e)$, then $\pi^e$ is admissible.
Using the identity $H(\pi^e | R_{0,T}) = H(\pi^n | (\cdot, \cdot - e)\# R_{0,T})$, a direct computation gives
\begin{align*}
  H(\pi^e\,|\,R_{0,T}) &= H(\pi^n\,|\,R_{0,T}) + e^\top\Sigma_T^{-1}\big(m^n-\Theta_T m_{\mathrm{init}}\big)+\tfrac12 e^\top\Sigma_T^{-1}e,\\
  H(\pi^e_T\,|\,\mu_{\mathrm{fin}}) &= H(\nu^n\,|\,\mu_{\mathrm{fin}}) + e^\top\Sigma_{\mathrm{fin}}^{-1}\big(m^n-m_{\mathrm{fin}}\big)+\tfrac12 e^\top\Sigma_{\mathrm{fin}}^{-1}e.
\end{align*}
By the optimality of $\pi^n$, the objective function $J(e) := H(\pi^e\,|\,R_{0,T}) + n H(\pi^e_T\,|\,\mu_{\mathrm{fin}})$ attains its minimum at $e=0$.
Together with the first order condition $\nabla J(0) = 0$, it has to hold $\Sigma_T^{-1}(m^n-\Theta_T m_{\mathrm{init}})+n\Sigma_{\mathrm{fin}}^{-1}(m^n-m_{\mathrm{fin}})=0$, i.e.
\begin{equation*}
  m^n-m_{\mathrm{fin}}=\big(I+n\Sigma_T\Sigma_{\mathrm{fin}}^{-1}\big)^{-1}\big(\Theta_T m_{\mathrm{init}}-m_{\mathrm{fin}}\big).
\end{equation*}
Consequently, for every $n\in\mathbb N$,
\begin{equation*}
  \sup_{t\in[0,T]} \mcal{W}_1\big(\mu^n_t,P^{\mu^n}_t\big)\;\ge\;\frac{|\Theta_T m_{\mathrm{init}}-m_{\mathrm{fin}}|}{1+n\|\Sigma_T\|\,\|\Sigma_{\mathrm{fin}}^{-1}\|}
  \;\ge\;\frac{c}{n},\qquad c:=\frac{|\Theta_T m_{\mathrm{init}}-m_{\mathrm{fin}}|}{1+\|\Sigma_T\|\,\|\Sigma_{\mathrm{fin}}^{-1}\|},
\end{equation*}
and $c>0$ whenever $\Theta_T m_{\mathrm{init}}\neq m_{\mathrm{fin}}$.

\subsection{Proof of Lemma~\ref{lemma:SB_EOT}} \label{appendix:proof_SB_EOT}
\begin{proof}
Let $p_\rho$ be the density of $\rho$ with respect to the Lebesgue measure.
Let $E := \{(z,w) \in \R^{2d}\times\R^{2d} : p_{\mathrm{init}}(z) > 0\}$. Then, $\pi(E^c) = \mu_{\mathrm{init}}(\{p_{\mathrm{init}} = 0\}) =0$.
Likewise, let $F := \{(z,w) \in \R^{2d}\times\R^{2d} : p_{\rho}(w) > 0\}$. Then, $\pi(F^c) = \rho(\{p_{\rho} = 0\}) =0$.
Thus, $\pi(E^c \cup F^c ) \leq \pi(E^c) + \pi(F^c) = 0$, and $p_{\mathrm{init}}, p_\rho > 0$, $\pi$-a.s.

Since $R^{\bs{\mu}}_{0,T}(dz,dw) = \mu_{\mathrm{init}}(dz) \, p^{\bs{\mu}}_T(z,w) \, dw$ and
$(\mu_{\mathrm{init}}\otimes\rho)(dz,dw) = \mu_{\mathrm{init}}(dz)\,p_\rho(w)\,dw$, the initial marginal cancels:
\begin{equation*}
  \frac{d(\mu_{\mathrm{init}}\otimes\rho)}{dR^{\bs{\mu}}_{0,T}}(z,w)
  = \frac{p_\rho(w)}{p^{\bs{\mu}}_T(z,w)} \in (0,\infty), \quad \text{on} \quad E\cap F.
\end{equation*}
Hence, it holds that $R^{\bs{\mu}}_{0,T}\big|_{E\cap F} \sim \mu_{\mathrm{init}}\otimes \rho \big|_{E\cap F}$ and
\begin{equation*}
  \pi  \ll R^{\bs{\mu}}_{0,T} \Leftrightarrow \pi  \ll   \mu_{\mathrm{init}}\otimes \rho.
\end{equation*}
If $\pi \not\ll R^{\bs{\mu}}_{0,T}$, both sides of \eqref{eq:SB_EOT} equal $+\infty$, because
$H(\pi\,|\,\mu_{\mathrm{init}}\otimes\rho) = +\infty$ while the remaining two terms on the right-hand
side are bounded from below by the estimates given next.

We may therefore assume that $\pi \ll R^{\bs{\mu}}_{0,T}$, so that, $\pi$-a.e.,
\begin{equation*}
  \log \frac{d\pi}{dR^{\bs{\mu}}_{0,T}}
  = \log \frac{d\pi}{d(\mu_{\mathrm{init}}\otimes\rho)} + c^{\bs{\mu}}(z,w) + \log p_\rho(w).
\end{equation*}
The integrals with respect to $\pi$ of each of the three terms on the right hand side have finite lower bounds: the first because
of the positivity of the relative entropy, the second because $c^{\bs{\mu}} \ge -\log \sup p^{\bs{\mu}}_T \geq - \log C$ by the
Gaussian upper bound of \cite[Theorem~1.1]{delarue_density_estimates}, and the third because of the assumption of this lemma.
Hence the identity may be integrated term by term, which is \eqref{eq:SB_EOT}.
The last assertion follows because $H(\rho\,|\,\leb)$ does not depend on
$\pi \in \Pi(\mu_{\mathrm{init}},\rho)$.
\end{proof}

\subsection{Derivation of Equation~\eqref{eq:evolutions_of_nabla_J}}
\label{appendix:derivation_of_evolution_nabla_J}
To begin with, $J$ satisfies the HJB equation
\begin{equation} \label{eq:HJB}
\partial_t J + \mcal{L}_t J - \frac{1}{2} |\nabla_v J |^2 = 0 \qquad \text{on } [0, T) \times \R^{2d},
\end{equation}
where $\mcal{L}_t$ is the generator of the reference process,
\begin{equation*}
\mcal{L}_t = \frac{1}{2} \Delta_v + v \cdot \nabla_x  + b(t, \cdot) \cdot \nabla_v ,
\qquad b = b^{\bs{\mu}} .
\end{equation*}
Indeed, writing $p = p(t, T; y, w)$, we have $\partial_t J = - \partial_t p / p$ and, since the first-order terms of $\mcal{L}_t$ are linear in $\log p$ while
\begin{equation*}
\frac{1}{2} \Delta_v J = - \frac{1}{2} \frac{\Delta_v p}{p} + \frac{1}{2} \frac{|\nabla_v p|^2}{p^2} = - \frac{1}{2} \frac{\Delta_v p}{p} + \frac{1}{2} |\nabla_v J|^2 ,
\end{equation*}
we obtain $\mcal{L}_t J = - \mcal{L}_t p / p + \frac{1}{2} |\nabla_v J|^2$. Adding the two identities and using the Kolmogorov backward equation
\begin{equation*}
    \partial_t p + \mcal{L}_t p = 0 \qquad \text{on } [0, T) \times \R^{2d},
\end{equation*}
yields \eqref{eq:HJB}. This is Fleming's logarithmic transformation: $J$ is the value function of a stochastic control problem whose optimal control is $u^* = - \nabla_v J$, and the corresponding optimally controlled process is $Z$ under $Q$.

Let us differentiate the HJB equation \eqref{eq:HJB} with respect to  the space variables $x_\alpha, v_\beta$.
For $\partial_{x_\alpha}$, we obtain
\begin{equation*}
    \partial_t \partial_{x_\alpha} J + v \cdot \nabla_x \partial_{x_\alpha} J
    + \partial_{x_\alpha} b \cdot \nabla_v J + b \cdot \nabla_v \partial_{x_\alpha} J
    + \frac{1}{2} \Delta_v \partial_{x_\alpha}J - \nabla_v J \cdot \nabla_v \partial_{x_\alpha} J = 0,
\end{equation*}
and for $\partial_{v_\beta}$, we have
\begin{equation*}
    \partial_t \partial_{v_\beta}J + \partial_{x_\beta} J + v \cdot \nabla_x \partial_{v_\beta} J
    + \partial_{v_\beta}b \cdot \nabla_v J + b \cdot \nabla_v \partial_{v_\beta}J
    + \frac{1}{2} \Delta_v \partial_{v_\beta}J
    - \nabla_v J \cdot \nabla_v \partial_{v_\beta} J = 0.
\end{equation*}
Organizing these equations, we get
\begin{equation*}
\begin{split}
   &\partial_t \partial_{x_\alpha} J +  \mcal{L}^{u^*}_t  \partial_{x_\alpha} J + \partial_{x_\alpha} b \cdot \nabla_v J = 0, \\
    &\partial_t \partial_{v_\beta} J +
    \mcal{L}^{u^*}_t \partial_{v_\beta} J + \partial_{x_\beta} J  + \partial_{v_\beta}b \cdot \nabla_v J = 0,
\end{split}
\end{equation*}
where
\begin{equation*}
\mcal{L}^{u^*}_t = v \cdot \nabla_{x} + ( b  - \nabla_v J ) \cdot \nabla_v + \frac{1}{2} \Delta_v
\end{equation*}
denotes the generator of the bridge process $Z$ under $Q$ in \eqref{eq:bridge_process}.
Thus, we apply It\^o's formula to $\partial_{x_\alpha}J(t, Z_t)$ and  $\partial_{v_\beta}J(t, Z_t)$ to derive \eqref{eq:evolutions_of_nabla_J}.

\subsection{Estimate of $\nabla_w \log p$} \label{appendix:estimate_for_end_point_variable}
Throughout this section, we fix the starting point $(0, z)$ and let the endpoint vary; we write $w = (w_x, w_v) \in \R^{2d}$ and
\begin{equation} \label{eq:def_I}
I(t, w) := - \log p(0, t; z, w), \qquad (t, w) \in (0, T] \times \R^{2d} .
\end{equation}
We shall prove that $|\nabla_w I(t_1, w)| \leq C_{t_1} (1 + |z| + |w|)$ for every $t_1 \in (0,T]$, with $C_{t_1}$ independent of $z$, $w$ and $\bs{\mu}$. The argument is the exact analogue of the one in Section~\ref{sec:proof_lineargrowth}, the bridge process being replaced by its time reversal; we therefore only indicate the difference and refer to Section~\ref{sec:proof_lineargrowth} for the arguments that are repeated.

\paragraph{The equation for $I$.}
In the forward variables $(t,w)$, $p$ solves the Fokker--Planck equation
\begin{equation} \label{eq:fokker_planck}
\partial_t p
+ \nabla_{w_x} \cdot \left( w_v \, p \right)
+ \nabla_{w_v} \cdot \left( b(t, w) \, p \right)
- \frac{1}{2} \Delta_{w_v} p = 0 .
\end{equation}
Dividing by $p$ and using $\Delta_{w_v} I = - \Delta_{w_v} p / p + |\nabla_{w_v} I|^2$ together with $\nabla_{w_x} \cdot (w_v p) = w_v \cdot \nabla_{w_x} p$, we obtain
\begin{equation} \label{eq:HJB_I}
\partial_t I
+ w_v \cdot \nabla_{w_x} I
+ b \cdot \nabla_{w_v} I
- \frac{1}{2} \Delta_{w_v} I
+ \frac{1}{2} |\nabla_{w_v} I |^2
= \nabla_{w_v} \cdot b .
\end{equation}
Compared with \eqref{eq:HJB}, the transport and diffusion terms have the opposite sign --- which is why the relevant bridge is the time-reversed one --- and there is an additional source term $\nabla_{w_v} \cdot b$, which is bounded  uniformly in $t$ and $\bs{\mu}$ under Assumption~\ref{assump:bounded_derivative}.

\paragraph{The time-reversed bridge.}
Fix $t_1 \in (0, T]$ and $0 < \delta < t_1$, and work on the reversed clock $u \in [0, \delta]$, corresponding to the original times $t = t_1 - u \in [t_1 - \delta, t_1]$. Let $\hat{R}_{t_1, w}$ be the law on path space of the time-reversed reference process started at $w$, under which the canonical process $Z$ satisfies
\begin{equation*}
dZ^x_u = - Z^v_u \, du, \qquad dZ^v_u = - b(t_1 - u, Z_u) \, du + dW_u ,\qquad Z_0 = w,
\end{equation*}
with $W$ a $\hat{R}_{t_1,w}$-Brownian motion, and set $h(u, y) := p(0, t_1 - u; z, y)$. A direct computation from \eqref{eq:fokker_planck} shows that $h$ solves
\begin{equation*}
\partial_u h + \hat{\mcal{L}}_u h = \big( \nabla_{w_v} \cdot b \big)(t_1 - u, \cdot) \, h ,
\qquad
\hat{\mcal{L}}_u := - w_v \cdot \nabla_{w_x} - b(t_1 - u, \cdot) \cdot \nabla_{w_v} + \frac{1}{2} \Delta_{w_v} ,
\end{equation*}
that is, $h$ is space-time harmonic for $\hat{\mcal{L}}$ up to the bounded potential $V(u, \cdot) := (\nabla_{w_v} \cdot b)(t_1 - u, \cdot)$. Accordingly we set
\begin{equation} \label{eq:density_martingale_I}
\hat{D}_u := \frac{p(0, t_1 - u; z, Z_u)}{p(0, t_1; z, w)} \, \exp \left( - \int_0^u V(r, Z_r) \, dr \right), \qquad u \in [0, \delta] .
\end{equation}
By It\^o's formula, $d \hat{D}_u = - \hat{D}_u \, \nabla_{w_v} I(t_1 - u, Z_u) \cdot dW_u$, so that $\hat{D}$ is a positive local martingale with $\hat{D}_0 = 1$; moreover, by the Gaussian upper bound and the boundedness of $V$,
\begin{equation} \label{eq:D_hat_bounded}
0 < \hat{D}_u \leq \frac{\hat{C}_{t_1 - \delta} \, e^{CT}}{p(0, t_1; z, w)}, \qquad u \in [0, \delta], \qquad
\hat{C}_{t_1-\delta} := \sup_{r \in [t_1 - \delta, \, t_1]} \ \sup_{y \in \R^{2d}} p(0, r; z, y) < \infty ,
\end{equation}
hence $\hat{D}$ is a bounded martingale. We define $\hat{Q}$ on $\mcal{F}_\delta$ by $d\hat{Q} / d\hat{R}_{t_1,w} := \hat{D}_\delta$ and call $Z$ under $\hat{Q}$ the reverse bridge process from $(t_1, w)$ to $(0, z)$. By Girsanov's theorem, under $\hat{Q}$ the process $W^{\hat{Q}}_u := W_u + \int_0^u \nabla_{w_v} I(t_1 - r, Z_r) \, dr$ is a Brownian motion and
\begin{equation} \label{eq:reverse_bridge_process}
dZ^x_u = - Z^v_u \, du, \qquad
dZ^v_u = - \big( b + \nabla_{w_v} I \big)(t_1 - u, Z_u) \, du + dW^{\hat{Q}}_u , \qquad Z_0 = w .
\end{equation}
In the sequel $\E$ denotes the expectation under $\hat{Q}$.

\paragraph{The energy identity.}
From \eqref{eq:density_martingale_I} we have $I(t_1 - u, Z_u) = I(t_1, w) - \log \hat{D}_u - \int_0^u V(r, Z_r) \, dr$, and, exactly as in the proof of Lemma~\ref{lemma:control_nabla_J_by_J}, $\E[\log \hat{D}_u] = \frac{1}{2} \E \int_0^u |\nabla_{w_v} I(t_1 - r, Z_r)|^2 dr$. Hence, for every $u \in [0, \delta]$,
\begin{align} \label{eq:bridge-rep_I}
\E \left[ \frac{1}{2} \int_0^{u} \left| \nabla_{w_v} I(t_1 - r, Z_r) \right|^2 dr \right]
&= I(t_1, w) - \E \left[ I(t_1 - u, Z_u) \right] - \E \left[ \int_0^u V(r, Z_r) \, dr \right] \\
&\leq I(t_1, w) + \log \hat{C}_{t_1 - \delta} + C T , \notag
\end{align}
where we have used the fact that $V$ is bounded.

\paragraph{Evolutions of $\nabla I$.}
Differentiating \eqref{eq:HJB_I} in $w_{x,\alpha}$ and $w_{v,\beta}$ and applying It\^o's formula under $\hat{Q}$ --- whose generator is $\hat{\mcal{L}}_u - \nabla_{w_v} I \cdot \nabla_{w_v}$ --- we obtain, for $u \in [0, \delta]$,
\begin{equation} \label{eq:evolutions_of_nabla_I}
\left\{
\begin{split}
d \, \partial_{w_{x,\alpha}} I(t_1 - u, Z_u) &= \left[
\partial_{w_{x,\alpha}} b \cdot \nabla_{w_v} I
- \nabla_{w_v} \cdot \partial_{w_{x,\alpha}} b
\right](t_1 - u, Z_u) \, du  \\
& \quad + \nabla_{w_v} \partial_{w_{x,\alpha}} I(t_1-u, Z_u) \cdot dW^{\hat{Q}}_u , \\
d \, \partial_{w_{v,\beta}} I(t_1 - u, Z_u) &= \left[
 \partial_{w_{x,\beta}} I
+ \partial_{w_{v,\beta}} b \cdot \nabla_{w_v} I
- \nabla_{w_v} \cdot \partial_{w_{v,\beta}} b
\right](t_1 - u, Z_u) \, du \\
& \quad +  \nabla_{w_v} \partial_{w_{v,\beta}} I(t_1-u, Z_u) \cdot dW^{\hat{Q}}_u .
\end{split}
\right.
\end{equation}
Apart from the overall sign of the drifts --- which is irrelevant, since only absolute values are used below --- the only difference with \eqref{eq:evolutions_of_nabla_J} is the presence of the extra terms $\nabla_{w_v} \cdot \partial_{w_{x,\alpha}} b$ and $\nabla_{w_v} \cdot \partial_{w_{v,\beta}} b$. These are bounded by $\| \nabla^2 b \|_\infty$, so that, after integration in time, they contribute an additive constant $C \delta$ to each estimate and do not affect the structure of the argument.

\paragraph{Conclusion.}
The analogue of Lemma~\ref{lemma:partial_J_bdd} reads: for every stopping time $\tau \leq \delta$, every $A \in \mcal{F}_\tau$ and $\partial \in \{ \partial_{w_{x,\alpha}}, \partial_{w_{v,\beta}} \}$,
\begin{equation*}
\E \left[ \left| \partial I(t_1 - \tau, Z_\tau) \right| \1_A \right]
\leq \frac{\hat{K}_{t_1 - \delta}(z) \, e^{CT}}{p(0, t_1; z, w)} \, \hat{R}_{t_1, w}(A),
\qquad
\hat{K}_{t_1-\delta}(z) := \sup_{r \in [t_1-\delta, \, t_1]} \ \sup_{y \in \R^{2d}} \left| \nabla_y p(0, r; z, y) \right| ,
\end{equation*}
which is finite by Corollary~\ref{cor:bounded_derivative}; the proof is identical to that of Lemma~\ref{lemma:partial_J_bdd}, the ratio $\partial p / p$ cancelling against $\hat{D}_\tau$ and $ \exp \left( - \int_0^\tau V(r, Z_r) \, dr \right)$ being bounded by $e^{CT}$. With this at hand, the localization and averaging procedure of the proof of Lemma~\ref{lemma:bounds_by_E(nabla_vJ)} applies word for word --- the roles of $\partial_{x_\beta} J$ and $\partial_{v_\beta} J$ being played by $\partial_{w_{x,\beta}} I$ and $\partial_{w_{v,\beta}} I$--- and yields
\begin{equation*}
\left| \nabla_{w} I(t_1, w) \right|
\leq C(\delta^{-\frac{3}{2}} + \delta^{-\frac{1}{2}} + \delta^{\frac{3}{2}})
\, \E \left[ \int_0^{\delta} \left| \nabla_{w_v} I(t_1 - r, Z_r) \right|^2 dr \right]^{\frac{1}{2}}
+ C \left( 1 + \delta + \delta^2 \right) .
\end{equation*}
Taking $\delta := t_1 / 2$ and bounding the expectation by \eqref{eq:bridge-rep_I}, together with the two-sided Gaussian estimates of \cite[Theorem~1.1]{delarue_density_estimates} as in the proof of Theorem~\ref{thm:linear_growth}, and using $\delta \leq T$ to absorb the additive term, we conclude
\begin{equation*}
\left| \nabla_w \log p(0, t_1; z, w) \right| \leq C_{t_1} \left( 1 + |z| + |w| \right),
\end{equation*}
with $C_{t_1}$ depending on $T$ and $t_1$ but not on $z$, $w$ or $\bs{\mu}$.

\section{Smoothness of $p$} \label{appendix:smoothness}
Throughout this section we work under Assumptions~\ref{assump:coefficient} and
\ref{assump:bounded_derivative}. Our goal is to prove that the transition kernel
$p$ is in $C^\infty$ in both space variables, and in $C^1$ in both time variables. We also provide the uniform bounds on its first derivatives (Corollary~\ref{cor:bounded_derivative}) as used in Section~\ref{sec:proof_lineargrowth}.

\subsection*{Step 1: Driftless Kinetic process}
Fix $\bs{\mu} \in C([0,T]; \mcal{P}_1(\R^{2d}))$. We write $p(s,t;z,w)$ for the
transition kernel of the reference process $Z$ under the path measure $R^{\bs{\mu}}_z$, suppressing
throughout the dependence on $\bs{\mu}$ of $p$, of the path measures, and of all the quantities
introduced below; all constants are independent of $\bs{\mu}$. We also write $\E := \E^{Q_z}$ when no
confusion arises, $Q_z$ being introduced below.

\paragraph{The driftless kinetic process.}
Set
\begin{equation*}
\sigma(s,r) := \begin{pmatrix} (s-r) I_d \\ I_d \end{pmatrix} \in \R^{2d \times d},
\qquad
Z^0_s := \int_0^s \sigma(s,r)\, dW_r ,
\qquad
\Theta_s(x,v) := (x + s v,\ v) .
\end{equation*}
For $s \le T$ the map $\Theta_s$ is linear with $\|\Theta_s\| \le 1+T$, and the driftless kinetic
process started at $z$ satisfies
\begin{equation} \label{eq:affine_structure}
Z^z_s = \Theta_s(z) + Z^0_s , \qquad s \in [0,T] .
\end{equation}
Indeed, $X_s = x + sv + \int_0^s (s-r)\, dW_r$ and $V_s = v + W_s$ solve $dX = V\,ds$, $dV = dW$. We
denote by $Q_z$ the path measure of this process. The dependence on $z$ in \eqref{eq:affine_structure}
is thus through the deterministic linear map $\Theta_s$ only; this is the structural fact on which the
whole section rests.

The random vector $Z^0_t$ is centred Gaussian with covariance matrix
\begin{equation} \label{eq:Sigma}
\Sigma_t := \int_0^t \sigma(t,r)\,\sigma(t,r)^* \, dr
= \begin{pmatrix} t^3/3 & t^2/2 \\ t^2/2 & t \end{pmatrix} \otimes I_d .
\end{equation}
Its $2\times 2$ block has determinant $t^4/3 - t^4/4 = t^4/12 > 0$, so that
$\det \Sigma_t = (t^4/12)^d > 0$ and
\begin{equation} \label{eq:Sigma_inverse}
\Sigma_t^{-1} = \begin{pmatrix} 12/t^3 & -6/t^2 \\ -6/t^2 & 4/t \end{pmatrix} \otimes I_d .
\end{equation}
In the language of Malliavin calculus, $\Sigma_t$ is the Malliavin covariance matrix of $Z^0_t$. The fact that
it is deterministic and explicitly invertible is what replaces H\"ormander's condition in the argument
below.

\paragraph{Putting back the drift.}
For $z \in \R^{2d}$ write $b_s(z) := b\big(s, \Theta_s(z) + Z^0_s\big)$ and
\begin{equation} \label{eq:girsanov_density}
\mcal{E}_t(z) := \exp\left( \int_0^t b_s(z) \cdot dW_s - \frac{1}{2}\int_0^t |b_s(z)|^2\, ds \right) .
\end{equation}
Since $\Lambda_0= \sup_{s,z}|b(s,z)| < \infty$ (recall Assumption~\ref{assump:bounded_derivative}), Novikov's condition holds, $\E[\mcal{E}_t(z)] = 1$, and moreover
\begin{equation} \label{eq:E_moments}
\sup_{z \in \R^{2d}} \big\| \mcal{E}_t(z) \big\|_{L^p}
\le \exp\left( \frac{(p-1)}{2}\, T \Lambda_0^2 \right), \qquad p \ge 1 ,
\end{equation}
as follows from writing
$\mcal{E}_t^p = \exp\big(\int_0^t p\,b_s \cdot dW_s - \frac12\int_0^t |p\,b_s|^2 ds\big)
\exp\big(\frac{p^2-p}{2}\int_0^t |b_s|^2 ds\big)$ and taking expectations.

By Girsanov's theorem, under $\mcal{E}_t(z) \cdot Q_z$ the process $W - \int_0^\cdot b_s(z)\, ds$ is a
Brownian motion, so that $Z^z$ solves the SDE with drift $b$; by weak uniqueness for that SDE,
\begin{equation} \label{eq:girsanov}
R_z = \mcal{E}_t(z) \cdot Q_z \qquad \text{on } \mcal{F}_t,
\end{equation}
recalling $R_z$ is the reference path measure for the reference process starting at $z$.
Consequently, for every bounded measurable $f$,
\begin{equation} \label{eq:kernel_representation}
\int_{\R^{2d}} p(0,t;z,w)\, f(w)\, dw
= \E^{Q_z}\big[ \mcal{E}_t(z)\, f\big(\Theta_t(z) + Z^0_t\big) \big] ,
\end{equation}
that is, $p(0,t;z,\cdot)$ is the density of the finite measure
$A \mapsto \E[\mcal{E}_t(z)\,\1_A(\Theta_t(z)+Z^0_t)]$.

\subsection*{Step 2: Representation of $p$}
\paragraph{Malliavin calculus: notation.}
We use the standard Malliavin calculus on the Wiener space of $W$ over $[0,t]$. For details, see
\cite[Chapter~1]{Nualart2006malliavin}. Let
\begin{equation*}
H := L^2\big([0,t]; \R^d\big),
\qquad
\langle u, w \rangle_H := \int_0^t u_r \cdot w_r \, dr,
\qquad
W(u) := \int_0^t u_r \cdot dW_r \ \ (u \in H) .
\end{equation*}
Let $\mcal{S}$ denote the set of random variables
$F = \varphi\big( W(u_1), \dots, W(u_n) \big)$ with $u_i \in H$ and $\varphi$ smooth with polynomially
bounded derivatives, and define the derivative operator
\begin{equation*}
D F := \sum_{i=1}^n \partial_i \varphi\big( W(u_1), \dots, W(u_n) \big)\, u_i \ \in H .
\end{equation*}
The operator $D$ is closable from $L^p(\Omega)$ into $L^p(\Omega; H)$ for every $p \ge 1$; we write
$\mbb{D}^{1,p}$ for the closure of $\mcal{S}$ under
$\|F\|_{1,p} := \big( \E|F|^p + \E\|DF\|_H^p \big)^{1/p}$, and $\mbb{D}^{k,p}$, $\mbb{D}^\infty :=
\bigcap_{k,p} \mbb{D}^{k,p}$ for the higher-order spaces. We shall use the chain rule
\begin{equation} \label{eq:chain_rule}
F \in \mbb{D}^{1,p},\ \varphi \in C^1_b(\R^n)
\quad \Longrightarrow \quad
\varphi(F) \in \mbb{D}^{1,p},
\qquad D \varphi(F) = \sum_j \partial_j\varphi(F)\, DF_j ,
\end{equation}
and the duality relation between $D$ and its adjoint $\delta$: for $u \in H$ and
$F \in \mbb{D}^{1,2}$, then
\begin{equation} \label{eq:duality}
\E\big[ \langle DF, u \rangle_H \big] = \E\big[ F\, \delta(u) \big] = \E \big[ F \int^t_0 u_r dW_r \big],
\end{equation}
where $\delta(u) = W(u) = \int^t_0 u_r dW_r$ is the adjoint operator of $D$ (Skorohod integral).

Since $Z^0_t = \int_0^t \sigma(t,r)\, dW_r$ is a Wiener integral with deterministic kernel, it lies in $\mbb{D}^\infty$ and
\begin{equation} \label{eq:D_Z0}
D_r Z^0_t = \sigma(t,r) \ \in \R^{2d \times d} .
\end{equation}

\paragraph{The steering control.}
For $h \in \R^{2d}$ define
\begin{equation} \label{eq:gamma}
\gamma^h_r := \sigma(t,r)^* \Sigma_t^{-1} h \in \R^d ,
\qquad r \in [0,t] .
\end{equation}
By \eqref{eq:Sigma} and \eqref{eq:D_Z0},
\begin{equation} \label{eq:steering}
\big\langle D Z^0_t,\ \gamma^h \big\rangle_H
= \int_0^t \sigma(t,r)\, \gamma^h_r \, dr = \Sigma_t \Sigma_t^{-1} h = h ,
\end{equation}
that is, $\gamma^h$ is the direction in the Cameron--Martin space along which $Z^0_t$ moves exactly by
$h$; it is the minimal-norm such direction, with
\begin{equation} \label{eq:gamma_norm}
\big\| \gamma^h \big\|_{H}^2 = \big\langle \Sigma_t^{-1} h,\ h \big\rangle
= \frac{12 |h_x|^2}{t^3} - \frac{12\, h_x \cdot h_v}{t^2} + \frac{4 |h_v|^2}{t} .
\end{equation}

\begin{remark}
Identity \eqref{eq:steering} is the reason the whole argument works: it is a deterministic linear
equation because, by \eqref{eq:D_Z0}, the Malliavin covariance matrix of $Z^0_t$ is the deterministic
matrix $\Sigma_t$, invertible by \eqref{eq:Sigma_inverse}.
\end{remark}

\begin{lemma} \label{lemma:IBP}
For $h \in \R^{2d}$ and $F \in \mbb{D}^{1,2}$ set
\begin{equation} \label{eq:H_operator}
\mcal{H}_h(F) := F\, W(\gamma^h) - \big\langle DF,\ \gamma^h \big\rangle_H .
\end{equation}
Then, for every $f \in C^1_b(\R^{2d})$ and every $z \in \R^{2d}$,
\begin{equation} \label{eq:IBP_master}
\E\Big[ \big\langle \nabla f\big( \Theta_t(z) + Z^0_t \big),\ h \big\rangle\, F \Big]
= \E\Big[ f\big( \Theta_t(z) + Z^0_t \big)\, \mcal{H}_h(F) \Big] .
\end{equation}
\end{lemma}

\begin{proof}
Note that $\Theta_t(z)$ is deterministic, so $D_r Z^z_t = D_r Z^0_t =
\sigma(t,r)$ by \eqref{eq:D_Z0}. By the chain rule \eqref{eq:chain_rule} and \eqref{eq:steering},
\begin{equation} \label{eq:chain_steering}
\big\langle D f(Z^z_t),\ \gamma^h \big\rangle_H
= \int_0^t \gamma^h_r \cdot \sigma(t,r)^* \nabla f(Z^z_t)\, dr
= \Big( \int_0^t \sigma(t,r)\gamma^h_r\, dr \Big) \cdot \nabla f(Z^z_t)
= \big\langle \nabla f(Z^z_t),\ h \big\rangle .
\end{equation}
Since $f$ and $\nabla f$ are bounded and $F \in \mbb{D}^{1,2}$, the product $F f(Z^z_t)$ belongs to
$\mbb{D}^{1,2}$, with $D\big( F f(Z^z_t) \big) = f(Z^z_t)\, DF + F\, D f(Z^z_t)$. Applying the duality relation
\eqref{eq:duality} to $Ff(Z^z_t)$ with the deterministic $u = \gamma^h$, for which
$\delta(\gamma^h) = W(\gamma^h)$, and using \eqref{eq:chain_steering},
\begin{equation*}
\E\Big[ f(Z^z_t)\, \big\langle DF, \gamma^h \big\rangle_H \Big]
+ \E\Big[ F\, \big\langle \nabla f(Z^z_t), h \big\rangle \Big]
= \E\Big[ \big\langle D\big( F f(Z^z_t) \big), \gamma^h \big\rangle_H \Big]
= \E\Big[ F f(Z^z_t)\, W(\gamma^h) \Big] .
\end{equation*}
Rearranging gives \eqref{eq:IBP_master}. All the expectations are finite: $f$ is bounded,
$F, \|DF\|_H \in L^2$ and $W(\gamma^h) \in L^2$.
\end{proof}

\begin{remark}\label{remark:H_is_Skorohod}
For readers familiar with Malliavin calculus, $\mathcal{H}_h(F)$ is nothing but the
Skorohod integral of the $H$-valued random variable $F\gamma^h$,
\begin{equation}\label{eq:H_as_divergence}
  \mathcal{H}_h(F) \;=\; \delta\big(F\,\gamma^h\big)
  \;=\; F\,W(\gamma^h) - \big\langle DF, \gamma^h\big\rangle_H.
\end{equation}
\end{remark}

For
$\xi,\eta \in \R^{2d}$ we write $\xi < \eta$ for the componentwise
inequality $\xi_j < \eta_j$, $j = 1,\dots,2d$.

Since $\delta$ maps
$\mathbb{D}^{\infty}(H)$ into $\mathbb{D}^{\infty}$
by \cite[Prop.~1.5.7]{Nualart2006malliavin}, and $\mathcal{E}_t(z) \in \mathbb{D}^{\infty}$ by Lemma~\ref{lemma:E_bounds}, we can therefore define,
for $t \in (0,T]$ and $z \in \mathbb{R}^{2d}$,
\begin{equation}\label{eq:G_def}
  G_t(z) \;:=\; \mathcal{H}_{e_{2d}}\circ \mathcal{H}_{e_{2d-1}}\circ\cdots\circ
  \mathcal{H}_{e_1}\big(\mathcal{E}_t(z)\big) \;\in\; \mathbb{D}^{\infty}
  \;\subset\; \bigcap_{p<\infty} L^p(\Omega).
\end{equation}

Using $G_t(z)$, we can now obtain the representation of the kernel $p$.
\begin{proposition}\label{prop:density_rep}
Let $t \in (0,T]$ and define
\begin{equation}\label{eq:Pi_def}
  \Pi_t(z,\xi) \;:=\; \E\Big[ \1_{\{Z^0_t \,>\, \xi\}}\, G_t(z) \Big],
  \qquad (z,\xi) \in \R^{2d}\times\R^{2d} .
\end{equation}
Then the transition density $p$ admits the representation
\begin{equation}\label{eq:density_rep}
  p(0,t;z,w) \;=\; \Pi_t\big(z,\; w - \Theta_t(z)\big),
  \qquad z,w \in \R^{2d} .
\end{equation}
\end{proposition}

\begin{proof}
Fix $f \in C_c^\infty(\R^{2d})$ and define its iterated integral
\begin{equation}\label{eq:primitive}
  \phi(y) \;:=\; \int_{-\infty}^{y_1}\!\!\cdots\int_{-\infty}^{y_{2d}}
  f(u_1,\dots,u_{2d})\,du_{2d}\cdots du_1 = \int_{\R^{2d}} \1_{\{y > w \}} \, f(w)\,dw ,
\end{equation}
so that $\partial_1\partial_2\cdots\partial_{2d}\,\phi = f$.
By \eqref{eq:G_def} and $\langle \nabla f, e_i \rangle = \partial_i f$, applying  \eqref{eq:IBP_master} with $h = e_j$ and the
test function $\phi$, we have
\begin{equation}\label{eq:IBP_iterated}
  \E\big[ f(Z^z_t)\,\mcal{E}_t(z)\big]
  \;=\; \E\big[ \phi(Z^z_t)\, G_t(z)\big] .
\end{equation}
We now use Fubini's theorem on the right hand side of \eqref{eq:IBP_iterated} to get
\begin{equation}\label{eq:fubini_step}
  \E\big[\phi(Z^z_t)\,G_t(z)\big]
  \;=\; \int_{\R^{2d}} \E\Big[\1_{\{Z^z_t \,>\, w\}}\,G_t(z)\Big] f(w)\,dw .
\end{equation}
Because $Z^z_t = \Theta_t(z)+Z^0_t$, the event $\{Z^z_t > w\}$ coincides with
$\{Z^0_t > w - \Theta_t(z)\}$, so the inner expectation in
\eqref{eq:fubini_step} equals $\Pi_t\big(z, w-\Theta_t(z)\big)$.

Combining \eqref{eq:kernel_representation}, \eqref{eq:IBP_iterated} and
\eqref{eq:fubini_step}, we obtain
\[
  \int_{\R^{2d}} p(0,t;z,w)\, f(w)\,dw
  \;=\; \int_{\R^{2d}} \Pi_t\big(z, w-\Theta_t(z)\big)\, f(w)\,dw
  \qquad \text{for all } f \in C_c^\infty(\R^{2d}) ,
\]
whence $p(0,t;z,\cdot) = \Pi_t(z,\,\cdot-\Theta_t(z))$ Lebesgue-a.e.

Finally, both sides are continuous in $w$. Indeed, applying the Cauchy--Schwarz inequality, we find
\begin{align*}
\left| \Pi_t(z,\xi) - \Pi_t(z,\xi') \right|
&= \left| \E \left[
    \left(\1_{\{Z_t^0 > \xi \}} - \1_{\{Z_t^0 > \xi' \} } \right) G_t(z)
\right] \right| \\
&\leq \| G_t(z)\|_{L^2} \, Q_z\left(\{Z_t^0 > \xi \} \mathbin{\Delta} \{Z_t^0 > \xi' \} \right)^{1/2} \xrightarrow[\xi' \to \xi ]{} 0,
\end{align*}
since $Z^0_t$ is a
nondegenerate Gaussian vector, and $\|G_t(z)\|_{L^2} <\infty $ by Proposition~\ref{prop:G_smooth}. Hence \eqref{eq:density_rep} holds for every $w$, for the continuous version of $p(0,t;z,\cdot)$.

\end{proof}

\subsection*{Step 3: Smoothness of $p$}
\begin{proposition}\label{prop:derivatives}
For all multi-indices $\alpha,\beta \in \N_0^{2d}$ the derivative
$\partial^\alpha_\xi\partial^\beta_z \Pi_t$ exists, is continuous on
$\R^{2d}\times\R^{2d}$, and
\begin{equation}\label{eq:Pi_derivatives}
  \partial^\alpha_\xi \partial^\beta_z \Pi_t(z,\xi)
  = (-1)^{|\alpha|}\, \E\Big[ \1_{\{Z^0_t > \xi\}}\
  \partial_z^\beta \mathcal{H}_{e}^{\alpha}\big( G_t(z) \big) \Big] ,
  \qquad
  \mathcal{H}_e^{\alpha} := \mathcal{H}_{e_1}^{\alpha_1}\circ\cdots\circ\mathcal{H}_{e_{2d}}^{\alpha_{2d}} .
\end{equation}
Consequently $p(0,t;\cdot,\cdot) \in C^\infty(\R^{2d}\times\R^{2d})$ and, by the chain
rule applied to \eqref{eq:density_rep},
\begin{equation}\label{eq:unified}
  \nabla_w\, p(0,t;z,w) = \big( \nabla_\xi \Pi_t \big)(z,\xi_w),
  \qquad
  \partial_{z_i} p(0,t;z,w) = \big( \partial_{z_i} \Pi_t \big)(z,\xi_w)
  - k^{(i)}_t \cdot \big(\nabla_\xi \Pi_t\big)(z,\xi_w) ,
\end{equation}
where $\xi_w := w - \Theta_t(z)$ and
$k^{(i)}_t := \partial_{z_i}\Theta_t(z)$, explicitly
\begin{equation}\label{eq:k_columns}
  k^{(x_j)}_t = (e_j, 0), \qquad k^{(v_j)}_t = (t\,e_j, e_j),
  \qquad j = 1,\dots,d .
\end{equation}
\end{proposition}

\begin{proof}
\emph{Step 1 ($\xi$-derivatives).}
Fix $\beta = 0$ and $f \in C_c^\infty(\R^{2d})$.  Since
$\mathcal{E}_t(z) \in \D^\infty$ and the $\mathcal{H}_{e_j}$ map $\D^\infty$ into itself,
$|\alpha|$ applications of \eqref{eq:IBP_master} give
$\E\big[(\partial^\alpha f)(Z^z_t)\,\mathcal{E}_t(z)\big]
= \E\big[ f(Z^z_t)\, \mathcal{H}^\alpha_e(\mathcal{E}_t(z))\big]$.
As in Proposition~\ref{prop:density_rep}, we apply the integration by parts formula \eqref{eq:IBP_master} for $\partial_1\partial_2\cdots\partial_{2d}\,\phi = f$ with $F=\mcal{H}^\alpha_e(\mcal{E}_t(z)) $ to derive
\[
  \int_{\R^{2d}} (\partial^\alpha f)(w)\, p(0,t;z,w)\,dw
  = \int_{\R^{2d}} f(w)\, \E\Big[\mathbf 1_{\{Z^z_t > w\}}\,\mathcal{H}^\alpha_e\big(G_t(z)\big)\Big] dw .
\]
Hence $\partial^\alpha_w p(0,t;z,\cdot)$ exists in the sense of distributions.
Moreover, the expectation in the integrand is continuous in $w$: for $w,w' \in \R^{2d}$, with
$1/p+1/q=1$,
\[
  \Big| \E\Big[\big( \mathbf 1_{\{Z^z_t>w\}} - \mathbf 1_{\{Z^z_t>w'\}}\big)
  \mathcal{H}^\alpha_e(G_t(z))\Big]\Big|
  \le \big\| \mathcal{H}^\alpha_e(G_t(z))\big\|_{L^q}\;
  Q_z\big(\{Z^z_t >w\} \mathbin{\Delta} \{Z^z_t >w'\} \big)^{1/p} \xrightarrow[\ w'\to w\ ]{} 0.
\]
The probability tends to $0$ because
$Z^z_t = \Theta_t(z)+Z^0_t$ has the bounded Gaussian density
$\mathcal{N}(\Theta_t(z),\Sigma_t)$, $\Sigma_t$ being invertible by
\eqref{eq:Sigma_inverse}, while $\|\mathcal{H}^\alpha_e(G_t(z))\|_{L^q}<\infty$ since
$\mathcal{H}^\alpha_e(G_t(z)) \in \D^\infty$.
Therefore the distributional derivative $\partial^\alpha_w p$ is continuous, which proves
\eqref{eq:Pi_derivatives} for $\beta=0$.

\emph{Step 2 ($z$-derivatives).}
First, set $\alpha = 0$.
By Proposition~\ref{prop:G_smooth}, the map
$z \mapsto G_t(z)$ is $C^\infty$ with
$\sup_z \|\partial^\beta_z G_t(z)\|_{L^2} < \infty$, and in particular, $\{  \partial^\beta_z G(z) \}_{z\in\R^{2d}}$ is uniformly integrable. Together with the fact that the integrand $\mathbf 1_{\{Z^0_t>\xi\}}$ does not depend on
$z$, we can use Vitali's convergence theorem to transfer the derivative into the expectation.

For general $\alpha$, we repeat the same argument with $\mcal{H}_e^\alpha (\mcal{E}_t(z))$ in place of $\mcal{E}_t(z)$.
This gives
\eqref{eq:Pi_derivatives} in general, and joint continuity in $(z,\xi)$ follows by
combining the two estimates above.

\emph{Step 3.}  Formula \eqref{eq:unified} is the chain rule applied to
$p(0,t;z,w) = \Pi_t(z, w - \Theta_t(z))$, with $\Theta_t(x,v) = (x+tv, v)$ giving
\eqref{eq:k_columns}.
\end{proof}

Finally, we obtain the uniform bound for the first order derivatives, which is used in Section~\ref{sec:proof_lineargrowth}.
\begin{corollary}\label{cor:bounded_derivative}
For every $\delta \in (0,T]$ there is $C_\delta = C(\delta,T,d, \Lambda_0, \ldots \Lambda_{2d+1}) < \infty$,
such that
\begin{equation}\label{eq:K_delta}
  \sup_{0 \le s \le t-\delta,\ t \le T}\ \sup_{z,w\in\R^{2d}}
  \Big(\big|\nabla_z\, p(s,t;z,w)\big| + \big|\nabla_w\, p(s,t;z,w)\big| \Big)
  \;\le\; C_\delta ,
\end{equation}
uniformly in $\boldsymbol\mu$.
\end{corollary}

\begin{proof}
It suffices to consider $s=0$ and $t \in [\delta,T]$: on $[s,t]$ the construction is
identical with $t$ replaced by $t-s$.

In the context of \cite[Prop.~2.1.4]{Nualart2006malliavin},  setting $F = Z^z_t$, we have $DF = \sigma(t,\cdot)$ and $\gamma_F = \Sigma_t$, and the weights are $H_{(i)}(Z^z_t,G) = \mathcal{H}_{e_i}(G)$.  Since $\gamma_F^{-1}DF$ is deterministic, its
$\D^{k,r}$-norm equals its $H$-norm,
$\|\Sigma_t^{-1}\sigma(t,\cdot)\|_H = (\operatorname{tr}\Sigma_t^{-1})^{1/2}
\le C(\delta)$ for $t \in [\delta,T]$.  Hence \cite[(2.32)]{Nualart2006malliavin} gives, for
$|\alpha| \le 2d+1$ and $1<p<q$,
\[
  \big\| H_\alpha(F, G)\big\|_{L^p} \;\le\; c_{p,q}\,C(\delta)^{|\alpha|}\,\|G\|_{\D^{|\alpha|,q}} .
\]
Discarding the indicator in \eqref{eq:density_rep} and \eqref{eq:unified}, and applying this with $G = \mathcal{E}_t(z)$ or
$G = \partial_{z_i}\mathcal{E}_t(z)$, Lemma~\ref{lemma:E_bounds} provides bounds on
$\nabla_w p$ and $\partial_{z_i}\Pi_t$ uniformly in $z,w$ and $t\in[\delta,T]$.  Finally,
by \eqref{eq:unified}, $\partial_{z_i} p = \partial_{z_i}\Pi_t - k^{(i)}_t\cdot\nabla_w p$ with
$|k^{(i)}_t| \le 1+T$.

\end{proof}

Given the smoothness of $p$ with respect to the spatial variables, we can deduce the smoothness in time.
\begin{proposition}\label{prop:time_regularity}
$p(s,t;z,w)$ is $C^1$
in $t\in(s,T]$ and in $s\in[0,t)$, and it satisfies
\[
  \partial_t\, p(s,t;z,w) = \big(\mathcal{L}^*_t\, p(s,t;z,\cdot)\big)(w),
  \qquad
  \partial_s\, p(s,t;z,w) = -\big(\mathcal{L}_s\, p(s,t;\cdot,w)\big)(z).
\]
\end{proposition}

\begin{proof}
We prove the statement in $t$; the statement in $s$ follows in the same way from the
Chapman--Kolmogorov relation and It\^o's formula applied to $p(s',t;\cdot,w)$.

Fix $(s,z)$ and put $q(r,\cdot) := p(s,r;z,\cdot)$.  By the argument similar to the one in
Corollary~\ref{cor:bounded_derivative}, we have that $\nabla^k_y q(r,\cdot)$ is bounded, uniformly in
$r\in[s+\delta,T]$, for every $k\le 3$ and every $\delta>0$.  For $f\in C_c^\infty(\R^{2d})$ and
$s<t<t'$, It\^o's formula followed by an integration by parts gives
\[
  \int f(w)\,\big[q(t',w)-q(t,w)\big]\,dw
  = \int_t^{t'}\!\!\int f(w)\;\mathcal{L}^*_rq(r,w)\,dw\,dr ,
\]
where $\mcal{L}^*$ is the adjoint of the generator:
\begin{equation}\label{eq:adjoint_generator}
  \mathcal{L}^*_r q(r,w)
  = -\,w_v\cdot\nabla_{w_x} q(r,w) - b\cdot\nabla_{w_v} q(r,w) - (\nabla_{w_v}\cdot b)\,q(r,w) + \tfrac12\,\Delta_{w_v} q(r,w).
\end{equation}
Both sides are continuous in
$w$, so
\[
  q(t',w)-q(t,w) = \int_t^{t'} \mathcal{L}^*_r q(r,w) \,dr
  \qquad \text{for every } w .
\]
It remains to show that the integrand is continuous in $r$. Once we show this continuity, the fundamental theorem of calculus applies to guarantee that $q$ is continuously differentiable in time.

For $r \in [t, t']$ and a real sequence $a_n \to 0$, we set $g^n(w) := q(r+a_n, w) - q(r, w)$. By the boundedness of $\nabla^k_y q(r,\cdot)$, uniformly in $r \in [t, t']$, the integrand $\mathcal{L}^*_r q(r,w)$ is uniformly bounded on $[t, t'] \times V$ for a compact set $V \subset \R^{2d}$. Therefore, it holds that $g^n \to 0$ uniformly on $V$.
We know that $\{\nabla g^n \}_{n\in \N}$ and $\{\nabla^2 g^n \}_{n\in \N}$ are both bounded uniformly on $V$.
Thus, by the Arzel\`a--Ascoli theorem, there exists a subsequence of $\{\nabla g^n \}_{n\in \N}$, which is still denoted by the same symbol and some function $k$, such that $\nabla g^n \to k$ uniformly on $V$. Since $ g^n \to 0$, we deduce $k=0$. Since any subsequence of $\{\nabla g^n \}_{n\in \N}$ has another subsequence which converges to $0$, we deduce the original sequence $\nabla g^n \to 0$. This means that $\nabla q(\cdot, w)$ is continuous in time. Exactly same argument applies to $\Delta_{w_v} q$, and we conclude that for each $w$, $\mathcal{L}^*_r q(\cdot,w)$ is continuous.
\end{proof}

\subsection*{Proof of the smoothness of $z\ \mapsto G(z)$}

\paragraph{A stable class of functionals.}
Let $\mathcal{B}_k$ ($k \in \{1,d\}$) denote the set of Borel maps
$\varphi : [0,T]\times\R^{2d} \to \R^{k}$ that are $C^\infty$ in the space
variable and satisfy
$\|\varphi\|_{(j)} := \sup_{s,y}\big|\nabla^j_y \varphi(s,y)\big| < \infty$ for every
$j \ge 0$.  For $\varphi \in \mathcal{B}_d$, $\psi \in \mathcal{B}_1$ and $t \in (0,T]$ put
\begin{equation}\label{eq:atoms}
  \mathcal{I}_t[\varphi](z) := \int_0^t \varphi\big(s, \Theta_s(z)+Z^0_s\big)\cdot dW_s ,
  \qquad
  \mathcal{J}_t[\psi](z) := \int_0^t \psi\big(s, \Theta_s(z)+Z^0_s\big)\,ds ,
\end{equation}
and let $\mathcal{A}_t(z)$ be the algebra (over $\R$) generated by these two families, i.e.\ elements of $\mathcal{A}_t(z)$ are thus finite sums of finite products of the
atoms \eqref{eq:atoms}. Notice that for any $p >1$, it holds that $\|\mcal{I}_t[\varphi](z)\|_{L^p} \leq C_p\|\varphi\|_{(0)} T^{1/2}<\infty$, as a consequence of BDG inequality, and the same holds for $\mcal{J}_t[\psi](z)$. Therefore, any element in $\mcal{A}_t(z)$ has a finite $p$-moment for any $p>1$ by H\"older's inequality.

Our first goal is to figure out the structure of $G(z)$.
\begin{lemma}\label{lemma:G_structure}
For every $\Psi \in \mathcal{A}_t(z)$ and every $h$,
\begin{equation}\label{eq:H_on_product}
  \mathcal{H}_h\big(\mathcal{E}_t(z)\,\Psi\big)
  \;=\; \mathcal{E}_t(z)\Big[\, \beta^h_t(z)\,\Psi \;-\; \big\langle D\Psi, \gamma^h\big\rangle_H \Big] ,
\end{equation}
where
\begin{equation}\label{eq:beta_def}
  \beta^h_t(z) \;:=\; W(\gamma^h) - \big\langle D\log\mathcal{E}_t(z),\gamma^h\big\rangle_H
  \;\in\; \mathcal{A}_t(z) .
\end{equation}
The bracket again lies in $\mathcal{A}_t(z)$.  In particular there is
$\Psi_t \in \mathcal{A}_t(z)$, obtained from $\beta^{e_1},\dots,\beta^{e_{2d}}$ by finitely
many products and applications of $\langle D\,\cdot\,,\gamma^{e_j}\rangle_H$, with
\begin{equation}\label{eq:G_structure}
  G_t(z) \;=\; \mathcal{E}_t(z)\,\Psi_t(z) .
\end{equation}
\end{lemma}

To this goal, the following lemma guarantees that the operator $\langle D\cdot, \gamma^h\rangle$ is stable within $\mcal{A}_t(z)$.
\begin{lemma}\label{lemma:algebra_stable}
For every $h \in \R^{2d}$,
\begin{equation}\label{eq:D_atoms}
  \big\langle D\,\mathcal{I}_t[\varphi],\ \gamma^h\big\rangle_H
  = \mathcal{J}_t\big[\gamma^h\!\cdot\!\varphi\big] + \mathcal{I}_t\big[(\nabla_y\varphi)\,\ell^h\big],
  \qquad
  \big\langle D\,\mathcal{J}_t[\psi],\ \gamma^h\big\rangle_H
  = \mathcal{J}_t\big[(\nabla_y\psi)\!\cdot\!\ell^h\big] ,
\end{equation}
for $\ell_s^h := \int^s_0 \sigma(s,r) \gamma^h_r dr$,
and all the integrands on the right belong again to $\mathcal{B}_d$, resp.\
$\mathcal{B}_1$.
\end{lemma}

\begin{proof}
By $D_r Z^0_s = \sigma(s,r)\mathbf 1_{\{r<s\}}$ \eqref{eq:D_Z0} and the chain rule,
$D_r\big[\varphi(s,\Theta_s(z)+Z^0_s)\big] = \nabla_y\varphi(s,\Theta_s(z)+Z^0_s)\,\sigma(s,r)$ for $r<s$.
Hence
\[
  D_r \mathcal{I}_t[\varphi] = \varphi\big(r,\Theta_r(z)+Z^0_r\big)
  + \int_r^t \big(\nabla_y\varphi(s,\Theta_s(z)+Z^0_s)\sigma(s,r)\big)^{\!*} dW_s ,
\]
and pairing with $\gamma^h$ and applying the stochastic Fubini theorem (cf.\ \cite[Theorem~65]{protter}; the integrand is bounded and jointly measurable,
and $\gamma^h \in L^2([0,t])$) gives
$\int_0^t \gamma^{h*}_r\!\int_r^t\nabla_y\varphi(s,\Theta_s(z)+Z^0_s)\sigma(s,r)dW_s\,dr
= \int_0^t \big(\nabla_y\varphi(s,\Theta_s(z)+Z^0_s)\,\ell^h_s\big)\cdot dW_s$.
The second identity in \eqref{eq:D_atoms} follows from the same computation with an ordinary
Fubini.
Boundedness of the new integrands and of their $y$-derivatives is clear from
$\varphi,\psi \in \mathcal{B}$ and from the boundedness of $\gamma^h,\ell^h$ on $[0,t]$.
\end{proof}

\begin{proof}[Proof of Lemma~\ref{lemma:G_structure}]
$D(\mathcal{E}_t\Psi) = \mathcal{E}_t\big(\Psi\, D\log\mathcal{E}_t + D\Psi\big)$, so
$\mathcal{H}_h(\mathcal{E}_t\Psi) = \mathcal{E}_t\Psi W(\gamma^h) - \mathcal{E}_t\Psi\langle D\log\mathcal{E}_t,\gamma^h\rangle
- \mathcal{E}_t\langle D\Psi,\gamma^h\rangle$, which is \eqref{eq:H_on_product}.

Applying \eqref{eq:D_atoms} to $\log\mathcal{E}_t(z) = \mathcal{I}_t[b] - \tfrac12\mathcal{J}_t[|b|^2]$,
\begin{equation}\label{eq:D_logE}
  \big\langle D \log \mathcal{E}_t(z), \gamma^h\big\rangle_H
  = \mathcal{J}_t\big[\gamma^h\!\cdot\! b\big]
  + \mathcal{I}_t\big[(\nabla_y b)\,\ell^h\big]
  - \mathcal{J}_t\big[ b \cdot (\nabla_y b)\,\ell^h \big] \;\in\; \mathcal{A}_t(z) .
\end{equation}
Lemma~\ref{lemma:algebra_stable} and \eqref{eq:D_logE} show the bracket stays in
$\mathcal{A}_t(z)$.  Iterating $2d$ times from $\Psi = 1$ gives \eqref{eq:G_structure}.
\end{proof}

\begin{proposition}\label{prop:G_smooth}
For each $t \in (0,T]$ the map $z \mapsto G_t(z)$ admits a modification of class
$C^\infty$ on $\R^{2d}$.  Moreover, for every multi-index $\beta$, there is
$\Psi_{t,\beta} \in \mathcal{A}_t(z)$ with
\begin{equation}\label{eq:dG}
  \partial^\beta_z G_t(z) \;=\; \mathcal{E}_t(z)\,\Psi_{t,\beta}(z).
\end{equation}
and for every $p \in [1,\infty)$, we have
\begin{equation}\label{eq:dG_Lp}
  \sup_{z\in\R^{2d}} \big\| \partial^\beta_z G_t(z) \big\|_{L^p(\Omega)} \;<\; \infty .
\end{equation}
\end{proposition}

The following lemma guarantees the smoothness of the integrals in \eqref{eq:atoms}, and that for each $\beta$, differentiation $\partial^\beta$ is stable as an operator from $\mcal{A}_t(z)$ to itself.
\begin{lemma}\label{lemma:atoms_smooth}
Let $\varphi \in \mathcal{B}_d$.  Then $z \mapsto \mathcal{I}_t[\varphi](z)$ admits a modification
which is of class $C^\infty$ on $\R^{2d}$, and for every multi-index
$\beta = (i_1,\dots,i_n)$,
\begin{equation}\label{eq:atom_derivative}
  \partial^\beta_z\, \mathcal{I}_t[\varphi](z) \;=\; \mathcal{I}_t\big[\varphi_\beta\big](z),
  \qquad
  \varphi_\beta(s,y) := \big(\nabla^{n}_y \varphi\big)(s,y)\big[\theta^{(i_1)}_s,\dots,\theta^{(i_n)}_s\big] ,
\end{equation}
where $\theta_s := \partial_z\Theta_s(z) = \begin{pmatrix} I_d & s I_d \\ 0 & I_d\end{pmatrix}$
and $\theta^{(i)}_s$ is its $i$-th column.  Moreover $\varphi_\beta \in \mathcal{B}_d$ with
$\|\varphi_\beta\|_{(j)} \le (1+T)^{n}\,\|\varphi\|_{(n+j)}$.  The same statements hold for
$\mathcal{J}_t[\psi]$.
\end{lemma}

Granting this lemma, we can prove the smoothness result of $G$:
\begin{proof}[Proof of Proposition~\ref{prop:G_smooth}]
Since $b \in \mcal{B}_d$, we have the following representation:
\begin{equation*}
\log \mcal{E}_t(z) = \mcal{I}_t[b](z) - \frac{1}{2} \mcal{J}_t[|b|^2](z) \in \mcal{A}_t(z).
\end{equation*}
By Lemma~\ref{lemma:atoms_smooth}, this has a modification which is smooth in $z$ almost surely, and thus so does $\mcal{E}_t(z)$. By Lemma~\ref{lemma:G_structure}, $G_t(z) \;=\; \mathcal{E}_t(z)\,\Psi_t(z)$, and product rule, we have
\begin{equation*}
\partial^{(i)}_z \left(\mathcal{E}_t(z)\,\Psi_t(z) \right) = \mathcal{E}_t(z)  \left( \Psi_t(z)\partial^{(i)}_z \log \mcal{E}_t(z) + \partial^{(i)}_z\Psi_t(z) \right).
\end{equation*}
By Lemma~\ref{lemma:atoms_smooth}, we know $\partial^{(i)}_z \left( \log \mcal{E}_t(z) + \Psi_t(z)\right) \in \mcal{A}_t(z)$. Inductively, we derive the structure \eqref{eq:dG}.

The bounds for the $L^p$-moments follow from the product rule, BDG inequality, and H\"older's inequality, along with \eqref{eq:E_moments}.
\end{proof}

\begin{proof}[Proof of Lemma~\ref{lemma:atoms_smooth}]
Set $\phi_z(s) := \varphi\big(s,\Theta_s(z)+Z^0_s\big)$.  Because $\Theta_s$ is linear in $z$, its
Jacobian $\theta_s$ is independent of $z$ and $\theta_s$ has no higher derivatives;
the chain rule therefore produces no terms beyond those in
\eqref{eq:atom_derivative}, and
\[
  \big| \partial^\beta_z \phi_z(s)\big| \;\le\; (1+T)^{|\beta|}\, \|\varphi\|_{(|\beta|)}
  \qquad \text{for all } s \in [0,t],\ z \in \R^{2d},\ \omega \in \Omega .
\]
Thus all the $z$-derivatives of $\phi_z$ are bounded by deterministic constants,
uniformly in $(s,z,\omega)$; Lipschitz continuity of $\partial^\beta_z \phi_z(s)$ in $z$  follows from the mean value theorem applied with the bound of $\nabla_z \left(\partial^\beta_z \phi_z(s)\right)$ for any $\beta$.  Hence
$\phi \in \mathbf{L}^{n+\mathrm{Lip},p}$ for every $n$ and every
$p \in (2d \vee 2, \infty)$, i.e.\ the condition
\cite[Prop.~2.3.1]{kunita2019stochastic} holds, which guarantees the existence of a modification of
$z \mapsto \int_0^t \phi_z\cdot dW$ of class $C^{n}$ together with the commutation
formula \eqref{eq:atom_derivative}.  As $n$ is arbitrary, the modification is $C^\infty$.
The case of $\mathcal{J}_t[\psi]$ is elementary using dominated convergence pathwise.
\end{proof}

We use the following bounds for Corollary~\ref{cor:bounded_derivative}.
\begin{lemma}\label{lemma:E_bounds}
$\mathcal{E}_t(z)$ has a modification such that the map $z \to \mathcal{E}_t(z)$ is in class $C^\infty$, almost surely.
Moreover, for every $k \in \N_0$ and $q \in [1,\infty)$, it holds
\[
  \sup_{t \in (0,T]}\ \sup_{z \in \R^{2d}}\ \Big( \big\|\mathcal{E}_t(z)\big\|_{\D^{k,q}}
  + \big\|\nabla_z \mathcal{E}_t(z)\big\|_{\D^{k,q}} \Big) \;<\; \infty .
\]
\end{lemma}
\begin{proof}
By \eqref{eq:girsanov_density}, $\log\mathcal E_t(z)=\mathcal I_tb-\tfrac12\mathcal J_t|b|^2\in\mathcal A_t(z)$ (cf.\ the proof of Proposition~\ref{prop:G_smooth}). By Lemma~\ref{lemma:atoms_smooth}, it has a $C^\infty$-modification, and hence so does $\mathcal E_t(z)$.

The Malliavin derivatives of $\log\mathcal{E}_t(z)$, and their $z$-derivatives, are
iterated stochastic and Lebesgue integrals over $[0,t]$ whose integrands involve only
$\nabla^j_y b$ and the bounded deterministic kernels $\sigma(s,r)$ and
$\partial_z\Theta_s$ (Lemma~\ref{lemma:atoms_smooth}); by the BDG
inequality they are bounded in every $L^q$, uniformly in $z$ and $t \le T$.  Since
$|b| \le \Lambda_0$, $\sup_{z,t}\E[\mathcal{E}_t(z)^q] \le \exp\big(\tfrac{q(q-1)}{2}T\Lambda_0^2\big)$,
and the claim follows from the chain rule for $\exp$ and H\"older's inequality.
\end{proof}

\begingroup
\emergencystretch=2em
\printbibliography
\endgroup
\end{document}